\documentclass{article}
\usepackage{amsmath,amsfonts,amssymb,amsthm}
\usepackage{color}

\providecommand{\R}{\mathbb{R}}
\providecommand{\C}{\mathbb{C}}
\providecommand{\N}{\mathbb{N}}

\providecommand{\V}{\mathcal{V}}

\providecommand{\calC}{\mathcal{C}}

\renewcommand{\leq}{\leqslant}
\renewcommand{\geq}{\geqslant}

\renewcommand{\div}{\operatorname{div}}
\newcommand{\curl}{\operatorname{curl}}
\newcommand{\trace}{\operatorname{tr}}
\newcommand{\as}{\operatorname{as}}
\newcommand{\tr}{\operatorname{tr}}
\newcommand{\dist}{\operatorname{dist}}
\newcommand{\Id}{\operatorname{Id}}
\newtheorem{Theorem}{Theorem}
\newtheorem{Definition}{Definition}
\newtheorem{Corollary}{Corollary}
\newtheorem{Proposition}{Proposition}
\newtheorem{Lemma}{Lemma}
\newtheorem{Remark}{Remark}

\begin{document}

\author{Olivier Glass\footnote{CEREMADE,
Universit\'e Paris-Dauphine, 
Place du Mar\'echal de Lattre de Tassigny,
75775 Paris Cedex 16, FRANCE
},
Franck Sueur\footnote{Laboratoire Jacques-Louis Lions,
Universit\'e Pierre et Marie Curie - Paris 6
175, rue du Chevaleret
75013 Paris
FRANCE
}, 
Tak\'eo Takahashi\footnote{Institut \'Elie Cartan de Nancy,
INRIA-CNRS-Universit\'es de Nancy
B.P. 239, 
54506 Vand\oe uvre-l\`es-Nancy Cedex 
FRANCE}
}
\date{\today}
\title{Smoothness of the motion of a rigid body immersed in an incompressible perfect fluid. R\'egularit\'e du mouvement d'un solide plong\'e  dans un fluide parfait incompressible. }
\maketitle

\begin{abstract}
We consider  the motion of a rigid body immersed in an incompressible perfect fluid which occupies a three-dimensional bounded domain.
For such a system the Cauchy problem is well-posed locally in time if 
the initial velocity of the fluid  is in the H\"older space $C^{1,r}$.
In this paper we prove that the smoothness of the motion of the rigid body may be only limited by the smoothness of the boundaries (of the body and of the domain).
In particular for analytic  boundaries the motion of the rigid body is analytic (till the classical solution exists and till the solid does not hit the boundary). Moreover in this case this motion depends smoothly on the initial data. 
\\ \\ On consid\`ere  le mouvement d'un corps solide plong\'e dans un fluide parfait incompressible qui occupe un domaine born\'e de $\R^3$. Pour ce syst\`eme le probl\`eme de Cauchy est bien pos\'e localement en temps si la vitesse initiale du fluide est dans l'espace de H\"older  $C^{1,r}$. Dans cet article on montre que la r\'egularit\'e du mouvement du corps solide ne peut \^etre limit\'ee que par la r\'egularit\'e des bords (du corps solide et du domaine).
En particulier si les bords sont analytiques alors le mouvement du corps solide est analytique (tant que la solution classique existe et que le corps solide ne touche pas le bord). De plus, dans ce cas, le mouvement d\'epend de mani\`ere $C^{\infty}$ des donn\'ees initiales.
\end{abstract}

\section{Introduction}
\label{Intro}

The main result of this paper is about the  motion of a rigid body immersed in an incompressible perfect fluid which occupies a three-dimensional bounded domain.
However our investigation of the problem also yields a slightly new result concerning  the case  without any rigid body, that is when the fluid fills the whole domain.
We first present our result in this case as a warm-up.
\subsection{Analyticity of the flow of a perfect fluid in a bounded domain}
We consider  a perfect  incompressible fluid filling a bounded regular domain $\Omega \subset \R^3$ with impermeable boundary $\partial \Omega$, so that 
the  velocity and  pressure fields $u(t,x)$ and  $p(t,x)$  satisfy  the Euler equations:
\begin{eqnarray}
	\displaystyle \frac{\partial u}{\partial t}+(u\cdot\nabla)u + \nabla p &=&0 ,  \ \text{for}  \ x\in  \Omega ,    \ \text{for}  \ t\in  (-T,T) \label{Euler1a} ,
	\\	\div u &=& 0  , \ \text{for}  \ x\in  \Omega,    \ \text{for}  \ t\in  (-T,T)\label{Euler2a}  , 
\\	u |_{t= 0} &=& u_0, \ \text{for}  \ x\in  \Omega   \label{Eulerci} ,
\\	u\cdot n &=& 0 , \ \text{for}  \ x\in \partial \Omega  ,  \ \text{for}  \ t\in  (-T,T)\label{Euler3a}.
\end{eqnarray}
Here $n$ denotes the unit outward normal on $\partial \Omega$.
The existence (locally in time) and uniqueness of classical solutions  to this problem is well-known,  since the classical works of  Lichtenstein, G\"{u}nter and Wolibner who deal with  the H\"older spaces $C^{\lambda,r}(\Omega)$ for $\lambda$ in $\N$ and $r \in (0,1)$, endowed with the norms:
\begin{align*}
  \| u   \|_{  C^{\lambda,r} ( \Omega ) } := \sup_{ |\alpha| \leqslant \lambda}   \big(  \|   \partial^\alpha u  \|_{L^\infty (  \Omega ) }  
+ \sup_{ x \neq y \in   \Omega } \frac{ |\partial^\alpha u (x) -  \partial^\alpha u (y)| }{  |x - y|^r }  \big) <  + \infty  .
\end{align*}
For $\lambda$ in $\N$ and $r \in (0,1)$, we consider the space
\begin{equation*}
{C^{\lambda,r}_\sigma ( \Omega ) } := \Big\{ 
u \in C^{\lambda,r}  ( \Omega  ) \ \Big/ \ \div u =0 \text{ in } \Omega \ \text{ and }\ u.n=0 \text{ on } \partial \Omega
\Big\}.
\end{equation*}
\begin{Theorem} \label{start}
There exists a constant $C_* = C_*  (\Omega ) > 0$ such that, for any $\lambda$ in $\N$ and $r \in (0,1)$,  for any  $u_0$ in  $C^{\lambda+1,r}_\sigma  ( \Omega  )$, there exist
$$T > T_* ( \Omega , \| u_0  \|_{  C^{\lambda +1,r} ( \Omega ) }) := C_* / \| u_0  \|_{  C^{\lambda +1,r} ( \Omega ) }$$
 and  a unique solution 
$u   \in  C_{w} ( (-T,T), C^{\lambda+1,r}  (\Omega )   )$ of  (\ref{Euler1a})--(\ref{Euler3a}).
\end{Theorem}
Above, and in the sequel, $ C_{w}$ refers to continuity with respect to the weak-$*$ topology of $C^{\lambda+1,r}  (\Omega )$.
Let us refer to the recent papers \cite{dutrifoy,katoholder,koch}.
\begin{Remark} \rm
We consider here (and in what follows) the earlier works cited above by
using H\"older spaces. Meanwhile, Theorem \ref{start} holds also true
for, say, any Sobolev space $H^{s}(\Omega)$ with $s > 5/2$ or even any
inhomogeneous Besov spaces  $B_{p,q}^{s}  (\Omega )$, with $1
\leqslant p,q \leqslant +\infty $ and with $s > \frac{3}{p}+1 $ (see
\cite{dutrifoy}) or $s \geqslant \frac{3}{p}+1$ if $q=1$ (so that
$B_{p,q}^{s}  (\Omega )$ is continuously embedded in Lip$ (\Omega )$).  
Let us recall that it is still not known whether the classical solutions of Theorem \ref{start} remain smooth for all times or blow up in finite times. 
Let us  also mention the recent work by Bardos and Titi \cite{BT} which shows that the $3$d Euler equations are not well-posed in the H\"older spaces $C^{0,r}(\Omega)$, for $r \in (0,1)$.
\end{Remark}
To the solution given by Theorem \ref{start} one associates
the flow $ \Phi$ defined on $(-T,T) \times \Omega$ by 
\begin{equation*}
\partial_t  \Phi(t,x)  = u (t, \Phi(t,x) )  \text{ and }
\Phi(0,x) = x .
\end{equation*}
The flow $\Phi$ can be seen as a continuous function of the time with values in the volume and orientation preserving diffeormorphisms defined on $\Omega$; in the sequel, in order to focus on the regularity properties, we consider $\Phi$ as a continuous function of $t \in [0,T]$ with values in the functions from $\Omega$ to $\R^{3}$. \par
\ \par
The first result of this paper shows that the smoothness of  the trajectories is  only limited by the smoothness of the domain boundary.
\begin{Theorem} \label{start1.5}
Under the hypotheses of Theorem \ref{start}, and assuming moreover that the boundary $\partial \Omega$ is $C^{k+\lambda+1,r}$, with $k  \in \N$, the flow  $ \Phi$    is $C^k$ from $(-T,T)$ to $ C^{\lambda+1,r}  (\Omega )$.
\end{Theorem}
Theorem \ref{start1.5} entails in particular that if  the boundary $\partial \Omega$ is $C^{\infty}$ then  the flow  $ \Phi$    is $C^{\infty}$ from $(-T,T)$ to $ C^{\lambda+1,r}  (\Omega )$.
 We will precisely study this limit case ``$k= \infty$" thanks to
 general ultradifferentiable classes, which emcompass in particular
 the class of analytic functions, as well as Gevrey and quasi-analytic
 classes. 
 Let $N := (N_s)_{s \geq 0} $ be a sequence of positive numbers.
Let $U$ be a domain in $\R^n$ and let $E$ be a Banach space endowed with the norm $\|   \cdot \|_{E} $.
 We denote by $\calC \{N\} (U; E) $ the class of functions $f:U  \rightarrow E$ such that
  there exist $L_{f} , C_{f} >0$ such that  for all $s \in
  \mathbb{N}$ and for all $x \in U$,
\begin{equation}
\label{croissderive}
\|  \nabla^s f (x)  \|  \leq C_f L_f^s N_{s}  ,
  \end{equation}
as a function with values in the set of symmetric $s$-linear continuous operators on $U$.
 Since for any positive $\lambda > 0$ there holds $\calC\{N\} = \calC\{\lambda N\}$, there is no loss of generality to assume $N_0 = 1$.
 When $N$ is  increasing, logarithmically convex (i.e. when the sequence $(N_{j + 1} / N_j )_{j \geq 0}$ does not decrease) then the class $\calC\{N\} (U; E) $ is an algebra with respect to pointwise multiplication.
Theorem \ref{start1.5} extends as follows:
\begin{Theorem} \label{start2}
Assume that the hypotheses of Theorem \ref{start} hold, and moreover that the boundary $\partial \Omega$ is in $\calC\{ N \}$, where $N :=  ( s! M_s)_{s \geq 0}$ with $ ( M_s)_{s \geq 0}$ an increasing, logarithmically convex sequence of real numbers, with $M_0 = 1$, and satisfying 
\begin{equation}
\label{stablediff}
\sup_{s \geq 1} \left(\frac{M_s}{M_{s-1} } \right)^{1/s}   \leq C_d < \infty .
 \end{equation}
Then the flow  $ \Phi$  is in $\calC\{N\} ((-T,T) ;C^{\lambda+1,r}  (\Omega )  ) $.
 In particular if the  boundary $\partial \Omega$ is  analytic (respectively Gevrey of order $m>1$) then  $ \Phi$  is  analytic (respectively Gevrey of order $m>1$) from  $(-T,T)$ to  $C^{\lambda+1,r}  (\Omega ) $.
\end{Theorem}
The particular cases of the last sentence are obtained 
 when $N$ is the sequence $N_j := (j!)^m$, with $m=1$ (respectively $m>1$); in these cases $\calC\{N\}  ( E) $ is the set of analytic functions (respectively Gevrey of order $m$).
 An important difference between the class of analytic functions and  the class of Gevrey  functions of order $m>1$ is that only the first one is quasi-analytic.\footnote{Actually the Denjoy-Carleman theorem states that  $\calC\{N\} ( E) $  is quasi-analytic, with $N$ as in Theorem  \ref{start2} if and only if $\sum_{j \geq 0}  \frac{M_j}{(j+1)M_{j+1} } < \infty$.}
 The logarithmic convexity of $M$  entails that for any\footnote{In the whole paper the notation $\N^*$ stands for $\N \setminus  \{0\}.$} $s \in \N^*$ and for any $ \alpha := ( \alpha_1,\ldots, \alpha_s )  \in \N^s$,
\begin{equation}
\label{subadditif}
M_{ \alpha_1}  \cdot  \cdot  \cdot M_{ \alpha_s} \leq M_{| \alpha | } ,
  \end{equation}
where the notation $ | \alpha |$ stands for $| \alpha | := \alpha_1 + \ldots+  \alpha_s$.
The condition \eqref{stablediff} is necessary and sufficient for the class  $\calC\{N\}  $ to be stable under derivation (cf. for instance \cite[Corollary 2]{thilliez}).
We will prove Theorem \ref{start2} by induction in such a way that Theorem \ref{start1.5} will be a simple byproduct of the proof of Theorem \ref{start2}.
\begin{Remark} \rm
Theorem \ref{start2} fills the gap between the results of Chemin \cite{cheinventiones}, \cite{cheminsmoothness}, Serfati  \cite{Serfati1}, \cite{Serfati2}, \cite{Serfati3}, Gamblin \cite{gamblinana}, \cite{gamblinana2} which prove analyticity of the flow for fluids filling the whole space
 and the paper \cite{katoana} of  Kato
which proves the smoothness in time for classical solutions in a smooth bounded domain.
\end{Remark}
\begin{Remark} \rm
It is fair to point out that the works of Gamblin and Kato cover the
more general case of spatial dimension $d \geqslant 2$. Moreover
Gamblin succeeds to prove that the flow of Yudovich's solutions (that
is, having merely bounded vorticity) is Gevrey $3$, when the fluid
occupies the whole plane. We will address the extension of this property in a bounded domain in a subsequent work.
\end{Remark}
\begin{Remark} \rm
As emphasized by Kato (cf. Example $(0.2)$ in \cite{katoana}) the
smoothness of the trajectories can only be proved under some kind of
global constraint, namely the wall condition (\ref{Euler3a}) in the
case studied here of a bounded domain. In the unbounded case one would
have to restrict the behavior of $u$ or $p$ at infinity (for instance
Gamblin   \cite{gamblinana} considers initial velocities $u_0$ which
are in $L^q (\R^3 )$ with
$1<q<+\infty$, in addition to be in $ C^{\lambda,r}$). 
  \end{Remark}
\begin{Remark} \rm
\label{Rk2}
It is natural to wonder if Theorem \ref{start2} admits a local (in space) counterpart. We do not adress this issue here since it does not seem relevant for considering the smoothness of the motion of an immersed body.
\end{Remark}
 \begin{Remark} \rm
 Gamblin's approach, following Chemin's one, uses a representation of the pressure via a singular integral operator, and relies on the repeated action on it of the material field. 
 On the opposite Kato's approach for bounded domains lies on the analysis of the action of the material field with differential operators, the non-local features being tackled with a classical elliptic regularity lemma.
Here we will refine the combinatorics in Kato's approach to obtain the analyticity, motivated by Gamblin's result.
 \end{Remark}
In the case where the boundary is analytic, the flow depends smoothly on the initial velocity. More precisely let us introduce, for any $R > 0$, 
\begin{equation*}
C^{\lambda+1,r}_{\sigma ,R}  ( \Omega ) := \Big\{ u \in C^{\lambda+1,r}_\sigma  ( \Omega  ) \ \Big/ \
 \| u  \|_{  C^{\lambda +1,r} ( \Omega ) } < R
\Big\}.
\end{equation*}
Then the following holds true.
\begin{Corollary} \label{Hadamard}
Let $\lambda$ in $\N$, $r \in (0,1)$ and $R > 0$. Suppose that $\partial \Omega$ is analytic.
  Then the mapping 
  $$u_0 \in C^{\lambda+1,r}_{\sigma ,R}  (\Omega ) \mapsto \Phi  \in C^\omega  ((-T_* ,T_*  ) ;C^{\lambda+1,r}  (\Omega )  )$$
   is $C^\infty$,  where $T_* = T_* ( \Omega ,R)$ is given by Theorem \ref{start}.
\end{Corollary}
Above the notation $C^\omega $ stands for the space of real-analytic functions.
\subsection{Analyticity of  the motion of a rigid body immersed in an incompressible perfect fluid}
The second and main result of this paper is about the motion of a rigid body immersed in an incompressible homogeneous perfect fluid, so that the system fluid-rigid body now occupies $\Omega$.
The solid is supposed to occupy at each instant $t \geq 0$  a closed connected subset $\mathcal{S}(t) \subset \Omega$ which is surrounded by a perfect incompressible fluid filling   the domain $\mathcal{F}(t) :=   \Omega \setminus  \mathcal{S}(t)$.
The equations modelling the dynamics of the system read
\begin{eqnarray}
\displaystyle \frac{\partial u}{\partial t}+(u\cdot\nabla)u + \nabla p &=& 0 ,  \ \text{for}  \ x\in   \mathcal{F}(t) , \label{Euler1a2}
\\	\div u &=& 0  ,  \ \text{for}  \ x\in   \mathcal{F}(t), \label{Euler2a2}
\\ m x_{B}''(t) &=&  \int_{\partial \mathcal{S}(t)} pn \ d\Gamma,   \label{Solide1}
\\  (\mathcal{J} r )'(t) &=&  \int_{\partial \mathcal{S}(t)} (x- x_{B})\wedge pn \ d\Gamma  , \label{Solide2} 
\\  u\cdot n &=& 0 , \ \text{for}  \ x\in \partial \Omega , \label{Euler3a2}
\\ u\cdot n &=& v \cdot n , \ \text{for}  \    \ x\in \partial \mathcal{S}(t),  \label{Euler3b}
\\	u |_{t= 0} &=& u_0 ,\label{Eulerci2}
\\  	 x_{B}(0)= x_0 , \ \ell(0)&=&\ell_0,\  r  (0)=  r _0. \label{Solideci}
\end{eqnarray}
The equations  (\ref{Solide1})  and (\ref{Solide2}) are the laws of
conservation of  linear momentum and  angular momentum. Here we denote by $m$ the mass of the rigid body (normalized in order that the density of the fluid is $\rho_{F} = 1$), by $x_B (t)$  the position of its center of mass, 
$n(t,x)$ denotes the unit normal vector pointing outside the fluid
and $ d\Gamma(t) $ denotes the surface measure on $\partial \mathcal{S}(t)$. The time-dependent vector $\ell(t) :=  x_{B}'(t) $ denotes the velocity of the center of mass of the solid and $r$ denotes its angular speed. The vector field $u$ is the fluid velocity, $v$ is the solid velocity and $p$ is the pressure field in the fluid domain. Finally in (\ref{Solide2})  the  matrix $\mathcal{J}$ denotes the moment of inertia (which depends on time). \par
The solid velocity is given by
\begin{equation} \label{vietendue}
v(t,x) :=\ell(t) +   r (t) \wedge (x -  x_{B}(t)).
\end{equation}
The rotation matrix $Q \in SO(3)$ is deduced from $r$ by the following differential equation (where we use the convention to consider the operator $r(t) \wedge \cdot$ as a matrix):
\begin{equation} \label{LoiDeQ}
Q'(t) = r(t) \wedge Q(t) \text{ and } Q(0) = \Id_3.
\end{equation}
According to Sylvester's law,  $\mathcal{J}$ satisfies
\begin{equation}\label{Sylvester}
 \mathcal{J} =Q \mathcal{J}_0 Q^{*} ,
 \end{equation}
where $\mathcal{J}_0$ is the initial value of $\mathcal{J}$. Finally, the domains occupied by the solid and the fluid are given by
\begin{equation} \label{PlaceDuSolideEtDuFluide}
{\mathcal S}(t) = \Big\{ x_{B}  (t) + Q(t)(x-x_{0}), \ x \in {\mathcal S}_{0} \Big\} \text{ and } {\mathcal F}(t)=\Omega \setminus \overline{{\mathcal S}(t)}.
\end{equation}
\ \par
Given a positive function $\rho_{S_{0}} \in L^{\infty}({\mathcal S}_{0};\R)$ describing the density in the solid (normalized in order that the density of the fluid is $\rho_{F} = 1$), the data $m$, $x_{0}$ and ${\mathcal J}_{0}$ can be computed by it first moments
\begin{equation} \label{EqMasse}
m :=  \int_{S_0} \rho_{S_0} dx  > 0,
\end{equation}
\begin{equation} \label{Eq:CG}
m x_0 :=   \int_{S_0} x \rho_{S_0} (x) dx,
\end{equation}
\begin{equation}\label{eqJ}
\mathcal{J}_{0} (t) := \int_{ \mathcal{S}_{0}}  \rho_{S_{0}}(x)   \big( | x- x_{0} |^2 \Id_3 -(x- x_{0})  \otimes   (x- x_{0})    \big) dx  .
\end{equation}
\ \\
For potential flows the first studies of the problem  (\ref{Euler1a2})--(\ref{Solideci}) dates back to D'Alembert, Kelvin and Kirchoff.
In the general case, the existence and uniqueness of classical solutions to the problem  (\ref{Euler1a2})--(\ref{Solideci}) is now well-understood thanks to the works of  Ortega, Rosier and Takahashi \cite{ort1}-\cite{ort2}, Rosier and Rosier \cite{rosier} in the case of a body in $\R^{3}$ and  Houot, San Martin and Tucsnak \cite{ht} in the case (considered here) of a bounded domain, in Sobolev spaces $H^{m}$, $m \geq 3$. We will use a rephrased version of their result in H\"older spaces, which reads as follows. Let
\begin{multline} \nonumber
\tilde{C}^{\lambda,r}_\sigma ( \mathcal{F}_0 ,x_0):= \Big\{ (\ell_0,r_0,u_0) \in  \R^3  \times  \R^3  \times C^{\lambda,r}  (\mathcal{F}_0 ) \ \Big/ \div(u_0)=0 \text{ in } \mathcal{F}_0, \\
u_0 \cdot n = 0 \text{ on } \partial \Omega \ \text{ and } \ (u_0 - v_0) \cdot n =  0 \text{ on } \partial \mathcal{S}_0 \text{ with } v_0  := \ell_0 +   r_0 \wedge (x -  x_0) \Big\}.
\end{multline}
\begin{Theorem}
  \label{start3}
Let be given $\lambda$ in $\N$, $r \in (0,1)$ and a regular closed connected subset $\mathcal{S}_0 \subset \Omega$. Consider a positive function $\rho_{\mathcal{S}_0} \in L^{\infty}(\mathcal{S}_0)$. We denote $m$ the mass, $x_0$ the position of the center of mass of  $\mathcal{S}_0$, ${\mathcal J}_{0}$ the initial matrix of inertia and $\mathcal{F}_0 :=  \Omega \setminus  \mathcal{S}_0$. There exists a constant $C_* = C_*  (\Omega ,\mathcal{S}_0,\rho_{\mathcal{S}_0} ) > 0$ such that the following holds. 
Consider $(\ell_0 ,r_0 ,u_0 )$ in  $\tilde{C}^{\lambda+1,r}_\sigma ( \mathcal{F}_0 ,x_0)$. 
Then there exists 
$$T > T_* ( \Omega,\mathcal{S}_0 , \rho_{\mathcal{S}_0},  \| u_0  \|_{  C^{\lambda+1,r} ({\mathcal F}_{0} ) }  +\|\ell_0\| + \|r_0\|   ) := \frac{C_*}{\| u_0  \|_{  C^{\lambda+1,r} ( {\mathcal F}_{0} )} +\|\ell_0\| + \|r_0\|},$$
such that the problem  (\ref{Euler1a2})--(\ref{PlaceDuSolideEtDuFluide})  admits a unique solution 
$$
(x_{B},  r,u)  \in  C^1 ( (-T,T) ) \times  C^0 ( (-T,T) )  \times   L^ \infty ( (-T,T), C^{\lambda+1,r}  (\mathcal{F} (t) )   ).
$$
Moreover $(x_{B},  r )  \in  C^2 ( (-T,T) ) \times  C^1 ( (-T,T) )$, $ u  \in C_w ( (-T,T) ; C^{\lambda+1,r}  (\mathcal{F} (t) ) )$ and $u \in C ( (-T,T) ; C^{\lambda+1,r'}  (\mathcal{F} (t) )   )$, for $ r'  \in (0,r)$; and the same holds for $\partial_t u$ instead of $u$ with $\lambda$ instead of $ \lambda+1$. 
\end{Theorem}
\begin{Remark} \label{RemEspTX}
The notation $L^ \infty ( (-T,T), C^{\lambda+1,r}  (\mathcal{F} (t) )
)$ is slightly improper since the domain ${\mathcal F}(t)$ depends on
$t$. One should more precisely think of $u$ as the section of a vector
bundle. However, since we think that there should not be any
ambiguity, we will keep this notation in what follows. The space $C ( (-T,T) ; C^{\lambda+1,r'}  (\mathcal{F} (t) ) )$ stands for the space of functions defined in the fluid domain, which can be extended to functions in $C ( (-T,T) ; C^{\lambda+1,r'} (\R^{3}))$.
\end{Remark}
\begin{Remark}
The regularity of $\rho_{S_{0}}$ is not an issue here since the solid density only intervenes through $m$, $x_{0}$ and ${\mathcal J}_{0}$.
\end{Remark}
For the sake of completeness, we prove Theorem \ref{start3} in the appendix. This proof will also allow us to get the following result concerning the continuous dependence of the solution with respect to initial data, which we will use later. Let us denote, for any $R > 0$,
\begin{equation*}
\tilde{C}^{\lambda,r}_{\sigma ,R} ( \mathcal{F}_0 ,x_0):= 
 \Big\{ (\ell_{0},r_{0},u_{0}) \in \tilde{C}^{\lambda,r}_{\sigma} ( \mathcal{F}_0 ,x_0) \ \Big/ 
 \ \| u_{0} \|_{  C^{\lambda,r} ( {\mathcal F}_{0} )} +\|\ell_{0} \| + \|r_{0}\|< R \Big\}.
\end{equation*}
\begin{Proposition}\label{Prostart3}
Let $R>0$. In the context of Theorem \ref{start3}, consider $(\ell^{1}_0 ,r^{1}_0 ,u^{1}_0 )$ and $(\ell^{2}_0 ,r^{2}_0 ,u^{2}_0 )$ in  $\tilde{C}^{\lambda+1,r}_{\sigma,R} ( \mathcal{F}_0 ,x_0) $.
Let 
\begin{equation*}
T =T_* ( \Omega,\mathcal{S}_0 , \rho_{\mathcal{S}_0},  R ).
\end{equation*}
Consider $(\ell^{1},r^{1},u^{1})$ and  $(\ell^{2},r^{2},u^{2})$ the corresponding solutions of (\ref{Euler1a2})--(\ref{PlaceDuSolideEtDuFluide}) in $[-T;T]$, and let $\eta_{1}$ and $\eta_{2}$ be the flows of $u^{1}$, $u^{2}$ respectively. Then for some $K=K( \Omega,\mathcal{S}_0 , \rho_{\mathcal{S}_0},  R )>0$ one has
\begin{multline*}
\| \eta_{1} - \eta_{2} \|_{L^{\infty}(-T,T;C^{\lambda+1,r}({\mathcal F}_{0}))} \\
+ \| u_{1}(t,\eta_{1}(t,\cdot)) - u_{2}(t,\eta_{2}(t,\cdot)) \|_{L^{\infty}(-T,T;C^{\lambda+1,r}({\mathcal F}_{0}))} 
+ \|(\ell_{1},r_{1}) - (\ell_{2},r_{2}) \|_{ L^{\infty}(-T,T;\R^{6})} \\
\leq K \Big[ \|\ell^{1}_0 - \ell^{2}_0\| + \|r^{1}_0 - r^{2}_0\| + \| u^{1}_0 - u^{2}_0 \|_{  C^{\lambda+1,r} ({\mathcal F}_{0} ) }  \Big].
\end{multline*}
\end{Proposition}
The aim of this paper is to prove additional smoothness of the motion of the solid and of the trajectories of the fluid particles. We define the flow corresponding to the fluid as 
\begin{equation*}
\partial_t	 \Phi^{{\mathcal F}}(t,x)  = u (t, \Phi^{{\mathcal F}}(t,x) )  \text{ and }
\Phi^{{\mathcal F}}(0,x)	= x  , \text{ for } (t,x) \in (-T,T) \times {\mathcal F}_{0},
\end{equation*}
and the flow corresponding to the solid as 
\begin{equation*}
\partial_t \Phi^{{\mathcal S}}(t,x)  = v (t,  \Phi^{{\mathcal S}}(t,x) ) \text{ and } \Phi^{{\mathcal S}}(0,x) = x \text{ for } (t,x) \in (-T,T) \times {\mathcal S}_{0}.
\end{equation*}
The flow corresponding to the solid is a rigid movement, that can be considered as a function of $t \in (-T,T)$ with values in the special Euclidean group $SE(3)$. 
Let us emphasize that in the previous result $T$ is sufficiently small in order that there is no collision between ${\mathcal S}(t)$ and the boundary $\partial \Omega$. \par
\ \par
We introduce, for $T>0$,  $\lambda \in \N$  and $r \in (0,1)$, 
\begin{equation*}
{\mathcal A}^{\lambda,r}_{{\mathcal S}_{0}} (T):=C^\omega  ((-T ,T ) ; SE(3) \times C^{\lambda,r}({\mathcal F}_{0})   ),
\end{equation*}
the space of real-analytic functions from $(-T,T)$ to $SE(3) \times C^{\lambda,r}({\mathcal F}_{0})$. \par
The main result of this paper is the following.
\begin{Theorem}
\label{start4}
 Assume that the boundaries  $\partial \Omega$ and $\partial \mathcal{S}_0$ are analytic and that the assumptions of Theorem \ref{start3} are satisfied.
Then $(\Phi^{{\mathcal S}}, \Phi^{{\mathcal F}} ) \in {\mathcal A}^{\lambda+1,r}_{{\mathcal S}_{0}} (T) $.
\end{Theorem}
The proof of Theorem \ref{start4} establishes that the motion of the solid and the trajectories of the fluid particles are at least as smooth as the boundaries $\partial \Omega$ and $\partial \mathcal{S}_0$. It would also be possible to consider general ultradifferentiable classes as in Theorem \ref{start2} or a limited regularity for the boundary as in Theorem \ref{start1.5}.
\begin{Remark} \rm
Theorem  \ref{start4} does not involve the concept of energy. However it gives as a corollary that the energy of the fluid $E^{{\mathcal F}} (t) :=  \frac{1}{2}  \int_{ \mathcal{F}(t)} u^2 \ dx $ is analytic on $(-T,T)$, since the total energy of the fluid-body system $E^{{\mathcal F}} (t)+ E^{{\mathcal S}} (t)$ is constant, where the energy of the body reads $E^{{\mathcal S}} := \frac{1}{2} m \ell^2 +  \frac{1}{2} \mathcal{J} r\cdot r$.
\end{Remark}
Let us now state the following corollary of Theorem  \ref{start4}, which is the counterpart of Corollary  \ref{Hadamard} in the case where a rigid body is immersed in an incompressible homogeneous perfect fluid.
\begin{Corollary} \label{Hadamard2}
Let be given  $\lambda$ in $\N$, $r \in (0,1)$, $R>0$ and a  closed connected regular subset $\mathcal{S}_0 \subset \Omega$, a  positive function  $ \rho_{S_0}$ in $L^\infty (\mathcal{S}_0)$. Assume that the boundaries  $\partial \Omega$ and $\partial \mathcal{S}_0$ are analytic. Then the mapping
$$(\ell_0 ,  r_0 , u_0 ) \in  \tilde{C}^{\lambda+1,r}_{\sigma ,R} ( \mathcal{F}_0 ,x_0)
\mapsto (\Phi^{{\mathcal S}}  , \Phi^{{\mathcal F}})  \in  {\mathcal A}^{\lambda+1,r}_{{\mathcal S}_{0}} (T_*) $$
is $C^\infty$, where $T_* = T_* ( \Omega ,  \mathcal{S}_0,\rho_{\mathcal{S}_0},R)$ is given by Theorem \ref{start3} .
\end{Corollary}
The proof of Corollary \ref{Hadamard2} is omitted since its proof is
similar to the proof of Corollary \ref{Hadamard}. It is the equivalent of the one given in Section \ref{Pr} for Corollary \ref{Hadamard}.
\begin{Remark}
When the boundary is merely $C^{\infty}$, we do not prove the analyticity of the flow, hence Corollary \ref{Hadamard2} cannot be deduced. However a simple compactness argument shows that the operator
$$(\ell_0 ,  r_0 , u_0 ) \in  \tilde{C}^{\lambda+1,r}_{\sigma ,R} ( \mathcal{F}_0 ,x_0)
\mapsto (\Phi^{{\mathcal S}}  , \Phi^{{\mathcal F}})  \in  C^{\infty}([-T_*,T_*];SE(3) \times C^{\lambda+1,r'}({\mathcal F}_{0})), $$
is continuous for $r'<r$ (even, to $C^{\infty}([-T_*,T_*];SE(3) \times C^{\lambda+1,r}_{w}({\mathcal F}_{0}))$ ). Indeed, for a sequence $(\ell_0 ,  r_0 , u_0 )$ converging to $(\ell_0 ,  r_0 , u_0 )$, we have both compactness of the images in $C^{k}([-T_*,T_*];SE(3) \times C^{\lambda+1,r'}({\mathcal F}_{0}))$, (by the uniform estimates in $C^{k+1}([-T_*,T_*]; SE(3) \times C^{\lambda+1,r}({\mathcal F}_{0}))$) and the continuity for a weaker norm in the range such as $C^{0}([-T_*,T_*];SE(3) \times C^{ \lambda+1,r}({\mathcal F}_{0}))$, which follows from Proposition \ref{Prostart3}).
\end{Remark}
\ \par
Another Corollary of Theorem \ref{start4}, or, to be more precise, of the estimates leading to Theorem \ref{start4}, deals with an inverse problem on the trajectory of the solid. A trivial consequence of the analyticity in time of the trajectory of the solid, is that, if we know this trajectory for some time interval $[-\tau,\tau]$ inside $[-T_{*},T_{*}]$ where the solution is defined (see Theorem \ref{start3}) -- without knowing precisely $u_{0}$ --, then we know it for the whole time interval (in the sense of unique continuation). The following corollary states that we can be a little more quantitative on this unique continuation property. \par
\begin{Corollary} \label{CorollaireInverse}
We consider $\Omega$, $\mathcal{S}_0$, and $\rho_{\mathcal{S}_0}$ fixed as previously.
Let $R>0$. Consider $\tau >0$ such that
\begin{equation*}
\tau < T_* ( \Omega,\mathcal{S}_0 , \rho_{\mathcal{S}_0},R ),
\end{equation*}
where $T_{*}$ is defined in Theorem \ref{start3}. There exist $C=C(\tau,\Omega,\mathcal{S}_0 , \rho_{\mathcal{S}_0},R) >0$ and $\delta=\delta(\tau,\Omega,\mathcal{S}_0 , \rho_{\mathcal{S}_0},R)$ in $(0,1)$ such that the following holds. Let $(\ell^{1}_{0},r^{1}_{0},u^{1}_{0})$ and $(\ell^{2}_{0},r^{2}_{0},u^{2}_{0})$ in $\tilde{C}^{\lambda+1,r}_{\sigma,R} ( \mathcal{F}_0 ,x_0)$.
Let $(\ell_{i},r_{i},u_{i})$ be the corresponding solution, and $\Phi_{i}^{\mathcal S}$ the corresponding solid flows. Then one has
\begin{equation} \label{EstimeeInverse}
\| \Phi_{1}^{{\mathcal S}} - \Phi_{2}^{{\mathcal S}} \|_{L^{\infty}(-T_{*},T_{*})} \leq C
\| \Phi_{1}^{{\mathcal S}} - \Phi_{2}^{{\mathcal S}} \|_{L^{\infty}(-\tau,\tau)}^{\delta}.
\end{equation}
\end{Corollary}
Let us emphasize that the constants $C>0$ and $\delta \in (0,1)$ depend on the knowledge (of an estimate) of the size of the initial data, but not on the initial data itself.
\begin{Remark}
As will follow from the proof, we could in fact replace the norm in the left hand side by a stronger norm such as $C^{k}([-T_{*},T_{*}])$.
\end{Remark}
Corollary \ref{CorollaireInverse} will be proven in Section \ref{Sec:PreuveCoroInverse}. \par
\ \par
Let us now briefly describe the structure of the paper. In Section \ref{Sec:FluideSeul}, we prove the claims concerning the system without immersed body, namely, Theorem \ref{start2} and Corollary \ref{Hadamard}. In Section \ref{Sec:SchemeDePreuve}, we describe the structure of the proof, reduced to the proof of two main propositions. In Section \ref{Sec:NewFormalIdentities}, we describe some formal identities needed in the proof. Section \ref{Proof2} establishes the two main propositions. In Section \ref{Sec:PreuveP1} we prove the formal identities. Finally, in Section \ref{Sec:PreuveCoroInverse}, we prove Corollary \ref{CorollaireInverse}. \par
\ \par
\begin{Remark} \rm
In the last years, several papers have been devoted to the study of
the dynamics of a rigid body immersed into a fluid governed by the
Navier--Stokes equations. We refer to the introduction of \cite{ort2}
for a survey of these results. 
\end{Remark}
%

%
%
%
\section{Proofs of Theorem \ref{start2} and Corollary \ref{Hadamard}} 
\label{Sec:FluideSeul}
\subsection{Proof of Theorem \ref{start2}} 
\label{Proof1}
From now on, we fix $\lambda \in \N$ and $r \in (0,1)$, and we introduce the following norms for functions defined in $\Omega$ or $\partial \Omega$
\begin{eqnarray*}
|\cdot|:=\|\cdot\|_{C^{\lambda,r}(\Omega)} \text{ and } |\cdot|_{\partial \Omega}:=\|\cdot\|_{C^{\lambda,r}(\partial \Omega)}, \\
\|\cdot\|:=\|\cdot\|_{C^{\lambda+1,r}(\Omega)} \text{ and } \|\cdot\|_{\partial \Omega}:=\|\cdot\|_{C^{\lambda+1,r}(\partial \Omega)}.
\end{eqnarray*}

First  it is classical to get that the flow map $ \Phi$ is $L^\infty (  ( -T,T),  C^{\lambda+1,r} (\Omega ))$ from its definition and Gronwall's Lemma.
In order to tackle the higher time derivatives of $ \Phi$ we will use the  material derivative

\begin{equation*}
D := \partial_{t} + u.\nabla .
\end{equation*}

Let us also introduce $\rho$ as a function defined on a neighborhood of $\partial \Omega$ as the signed distance to $\partial \Omega$, let us say, negative inside $\Omega$.
Since we assume that the boundary $\partial \Omega$ is in $\calC\{ N
\}$ with $N$ satisfying the hypothesis of Theorem \ref{start2}, there exists $c_\rho > 1$ such  that for all $s \in \mathbb{N}$,
\begin{equation} \label{RhoCN}
\|\nabla^s \rho \| \leq c_\rho^s \, s! M_s ,
\end{equation}
as a function with values in the set of symmetric $s$-linear forms. Let us introduce for $L>0$ the following function
\begin{equation} \label{Eq:DefGammaL}
\gamma(L) := \sup_{k  \geq 1} \left\{ 3 \sum_{s=2}^{k+1} \ L^{1-s} s (c_\rho C_d)^s \left(\frac{ k+1 }{ k-s+2 }\right)^2 20^{s} \, + \, C_{\Omega} \sum_{s=1}^{k-1}  \frac{ k-s }{L k s } \left(\frac{ k+1}{ (k-s+1) s } \right)^2 \right\}.
\end{equation}
The constant $C_{d}$ (which can be assumed to be larger than 1) above was introduced in \eqref{stablediff}; the constant $C_{\Omega}$ depends only on the geometry of $\Omega$ and will be introduced below in \eqref{DefCOmega}. Without loss of generality, we suppose that $c_\rho C_d \geq 1$. Now we fix $L$ large enough such that  
\begin{eqnarray}
\label{le}
\gamma(L) \leq \frac{1 }{3 c_{\mathfrak r}}.
 \end{eqnarray}
The constant $c_{\mathfrak r}$ appearing in \eqref{le} will be introduced in Lemma \ref{Lemme1}.
We are going to prove by induction that for all $k\in \mathbb{N}$, all $t \in (-T,T)$,
\begin{equation}
\label{indu}
\| D^k u \| + | \nabla D^{k-1} p | \leq \frac{k! M_k L^k }{(k+1)^2} \|u\|^{k+1} ,
\end{equation}
where the second term is omitted when $k=0$.
Since 
\begin{equation*}
\partial^{k +1 }_t \Phi(t,x)= D^k u (t, \Phi (t,x)),
\end{equation*}
this will prove Theorem \ref{start2}. We will proceed by regularization, working from now on a smooth flow, with the same notation. Since the estimates that we are going to prove are uniform with respect to the regularization parameter, the result will follow. We refer to  \cite{gamblinana} for more details on this step.
 \begin{Remark} \rm
One should ask whether the flow could get smoother in $x$ at some time: it could be that some 
cancellations arise in the composition of the field with the flow. Loosely speaking Theorem $5$ of Shnirelman's paper  \cite{SH} indicates that it is never the case, despite the fact that its setting is slightly different, Shnirelman considering fluid motions on the two-dimensional torus, in a Besov space $B^s_{2,\infty}$ with $s > 3$.
\end{Remark}
For $k=0$, there is nothing to prove. Let us assume that \eqref{indu} holds up to order $k-1$. 
To estimate $ D^{k} u$ we will use the following regularity lemma for the div-curl system.
\begin{Lemma}[Regularity]\label{Lemme1}
Let $\Gamma_i$ $(i=1,\ldots,g)$ a family of smooth oriented loops which generates a basis of the first singular homology space of $\Omega$ with real coefficients.
For any $u \in C^{\lambda,r}(\Omega)$ such that
$$
\div  u\in C^{\lambda,r}(\Omega), \quad \curl  u \in C^{\lambda,r}(\Omega), \quad  u\cdot n \in C^{\lambda+1,r}(\partial \Omega),
$$
one has $ u\in C^{\lambda+1,r}(\Omega)$ and there exists a constant $c_{\mathfrak r}$ depending only on $\Omega$ and $\Gamma_i$ $(1\leq i\leq g)$ such that
\begin{equation}
\label{reg1}
\| u\| \leq c_{\mathfrak r} \left(|\div  u| + |\curl  u| + \| u\cdot n\| + |\Pi  u|_{\R^{g}} \right),
\end{equation}
where $\Pi$ is the mapping defined by 
$u\mapsto \left(\oint_{\Gamma_1} u\cdot \tau d\sigma,\ldots, \oint_{\Gamma_g} u\cdot \tau d\sigma\right)$.
\end{Lemma}
\begin{proof}
This is more or less classical. The same result appears for instance in Kato's paper (see \cite[Lemma 1.2]{katoana}) with $|\Pi(u)|_{\R^{g}}$ replaced by $\| \tilde{\Pi}(u)\|_{C^{\lambda,r}(\Omega)}$, where $\tilde{\Pi}$ is the $L^{2}(\Omega)$ projector on the tangential harmonic vector fields (that is, having null divergence, curl and normal trace):
\begin{equation}
\label{reg1Kato}
\| u\| \leq c \left(|\div  u| + |\curl  u| + \| u\cdot n\| + |\tilde{\Pi}  u| \right).
\end{equation}
Given $u \in C^{\lambda,r}(\Omega)$, we apply \eqref{reg1Kato} to $u - \tilde{\Pi}(u)$, so that 
\begin{equation}
\label{reg1Kato2}
\| u - \tilde{\Pi}(u)\| \leq c \left(|\div  u| + |\curl  u| + \| u\cdot n\| \right).
\end{equation}
Now we notice that on the space of tangential harmonic vector fields, $\Pi$ is injective, since a vector field $v$ satisfying $\curl v=0$ and $\Pi(v)=0$ is a global gradient field (as a matter of fact, $\Pi$ is even bijective on this space as a consequence of de Rham's theorem). Since the space of tangential harmonic vector fields is finite-dimensional, it follows that for some $C>0$ independent of $u$, one has
\begin{equation*}
\| \tilde{\Pi}(u) \| \leq C | \Pi(\tilde{\Pi}(u)) |_{\R^{g}}.
\end{equation*}
Using the continuity of $\Pi$ and \eqref{reg1Kato2}, we infer
\begin{equation*}
| \Pi(u - \tilde{\Pi}(u)) |_{\R^{g}} \leq C \left(|\div  u| + |\curl  u| + \| u\cdot n\| \right).
\end{equation*}
From the above inequalities we deduce \eqref{reg1}. 
\end{proof}
\ \par
Now applying Lemma \ref{Lemme1} to the solution of \eqref{Euler1a}--\eqref {Euler3a}
we get
\begin{equation}
\label{indu2}
\| D^{k} u \| \leq c_{\mathfrak r} \left(| \div D^k u | + | \curl D^k u| +  \| D^{k} u\cdot n \|_{\partial \Omega}+ | \Pi D^k u | \right) .
\end{equation}
We establish formal identities for  $\div D^k u$,  $\curl D^k u$, (respectively the normal trace $n\cdot D^k u$ on the boundary $ \partial \Omega$), for  $ k \in \N^* $, as combinations of the functionals 
\begin{gather} \label{DefsFetH}
f( \theta)  [u] : = \nabla D^{\alpha_1} u \cdot \ldots \cdot \nabla D^{\alpha_s} u, \\
\text{respectively } h( \theta)  [u] : =\nabla^s \rho \{ D^{\alpha_1} u , \ldots  ,D^{\alpha_s} u \} ,
\end{gather}
with 
$$
\theta := (s, \alpha),
$$
where  $s \in \N^*$ and $ \alpha := ( \alpha_1,\ldots, \alpha_s )  \in \N^s$. Furthermore, these combinations will only involve indices $(s,\alpha)$ belonging to
\begin{equation} \label{TheDefinitionOfA}
\mathcal{A}_{k} := \{ \theta = (s, \alpha) \ / \ 2 \leqslant s  \leqslant k+1  \text{ and }
\alpha := ( \alpha_1,\ldots, \alpha_s )  \in \N^s / \ | \alpha |    = k+1 - s \}.
\end{equation}
Here the notation $ | \alpha |$ stands for $| \alpha | := \alpha_1 + \ldots+  \alpha_s$.
 \par
We will need to estimate the coefficients of these combinations. To that purpose, we introduce the following notations: for $\alpha := ( \alpha_1,\ldots, \alpha_s )  \in \N^s$ we will denote $\alpha ! :=  \alpha_1 ! \ldots \alpha_s !$. We will denote by $\trace\{A\} $ the trace 
of $A \in {\mathcal M}_{3}(\R)$ and by $\as\{A\}:=A-A^*$ the antisymmetric part of $A \in {\mathcal M}_{3}(\R)$. In the sequel, we use the convention that the curl is a square matrix rather than a vector. \par
The precise statement is the following (compare to \cite[Proposition 3.1]{katoana}).
\begin{Proposition}\label{P1}
For $k \in \N^*$, we have in $\Omega$
\begin{eqnarray}\label{P1f}
\div D^k u=\trace\left\{F^k [u]\right\} \text{ where } 
F^k [u] := \sum_{\theta   \in \mathcal{A}_{k}  } c^1_k (\theta  )  \,  f(\theta)  [u], \\ 
\label{P2f}
\curl D^k u=\as\left\{G^k [u]\right\} \text{ where } 
G^k [u] :=  \sum_{\theta  \in \mathcal{A}_{k}  }  c^2_k (\theta )   \, f(\theta)  [u],
\end{eqnarray}
where, for $i=1$, $2$, the $c^i_k (\theta )$ are integers satisfying
\begin{equation} \label{Ci:1}
|c^i_k ( \theta) | \leqslant \frac{k ! }{\alpha ! }, 
\end{equation}
and on the boundary $ \partial \Omega$:
\begin{eqnarray}\label{P4f}
n\cdot D^k u =H^k [u] \text{ where } 
H^k [u] :=  \sum_{ \theta  \in \mathcal{A}_{k}  }  c^3_k (\theta )  \, h(\theta)  [u],
\end{eqnarray}
where  the $ c^3_k (\theta)  $ are negative integers satisfying
\begin{equation} \label{Ci:2}
| c^3_k (\theta)  | \leqslant \frac{k ! }{\alpha ! (s-1)!}.
\end{equation}
\end{Proposition}
Proposition \ref{P1} is a particular case of a more general statement, namely Proposition \ref{P1New}, which will be proven in Section~\ref{Sec:PreuveP1}. \par
\ \par
Now thanks to Proposition \ref{P1}, \eqref{subadditif}, the fact that the sequence $(M_{s})_{s \geq 0}$ is increasing and the induction hypothesis (see \eqref{indu}), we have
%
\begin{eqnarray}
\nonumber
| F^k  [u] | &\leq & \sum_{(s,\alpha) \in \mathcal{A}_{k}} \frac{k!}{\alpha!} \ \prod_{i=1}^{s} \frac{\alpha_{i}! M_{\alpha_{i}} L^\alpha_{i} }{(\alpha_{i}+1)^2} \|u\|^{\alpha_{i}+1} \\
%
%
\label{new0.0}
&\leq &  k! M_k L^k  \,  \| u \|^{k+1}  \,  \sum_{s=2}^{k+1} \ L^{1-s} 
\sum_{ \alpha  /  \, | \alpha | = k+1 - s }  \,  \prod_{i=1}^s  \frac{1}{ (1+\alpha_i )^2}  .
\end{eqnarray}
We now use \cite[Lemma 7.3.3]{cheminsmf}, which we recall for the reader's convenience.  
\begin{Lemma} \label{LemmeCheminSMF}
For any couple of positive integers $(s,m)$ we have
\begin{equation}\label{DefUpsilon}
\sum_{\substack{{\alpha \in \N^{s}} \\ {|\alpha|=m} }}  \Upsilon(s,\alpha) \leq \frac{20^{s}}{(m+1)^{2}}, \text{ where }
\Upsilon(s,\alpha):=\prod_{i=1}^s \frac{1}{(1+\alpha_{i})^2}.
\end{equation}
\end{Lemma}
We deduce from \eqref{new0.0} and from the above lemma
\begin{eqnarray} \label{DeCadix}
| F^k  [u] | \leq  \frac{k!  M_k L^k }{(k+1)^2}  \| u \|^{k+1} \,  \sum_{s=2}^{k+1} \ L^{1-s}  \,  20^{s}   \, \frac{(k+1)^2}{ (k-s+2 )^2}.
\end{eqnarray}
We have the exact same bound on $ | G^k[u]|$ using \eqref{P2f}. For what concerns $H^k[u]$, by using \eqref{stablediff}, \eqref{subadditif}, \eqref{RhoCN} and \eqref{P4f} we obtain
\begin{eqnarray}
\nonumber
\| H^k [u] \|_{\partial \Omega} &\leq& \sum_{s=2}^{k+1} \, \sum_{\alpha / \, |\alpha| = k+1-s}
\frac{k!}{\alpha!(s-1)!}\, s!M_{s}c_{\rho}^{s} \ \prod_{i=1}^{s} \frac{\alpha_{i}! M_{\alpha_{i}} L^\alpha_{i} }{(\alpha_{i}+1)^2} \|u\|^{\alpha_{i}+1}
\\
\label{new1.4} 
&\leq& \frac{k!  M_k L^k }{(k+1)^2}  \| u \|^{k+1}  \sum_{s=2}^{k+1} \ s L^{1-s} \, (c_\rho C_d)^s   \,  20^{s} \,   \frac{(k+1)^2}{ (k-s+2 )^2}  .
\end{eqnarray}
Concerning the pressure it is possible to get by induction from (\ref{Euler1a}) the following identities, due to Kato, see \cite[Proposition 3.5]{katoana}.
\begin{Proposition} \label{P3}
For $k\geq 1$, we have in the domain $\Omega$
\begin{equation} \label{t3.4}
D^k u +\nabla D^{k-1} p= K^k[u] 
\end{equation}
where $K^{1}[u]=0$ and for $k\geq 2$,
$$
K^{k}[u]=-\sum_{r=1}^{k-1}  \dbinom{k-1}{ r} \nabla D^{ r-1} u \cdot D^{k-r} u.
$$
\end{Proposition}
Now using Proposition \ref{P3} we have 
\begin{equation*}
\Pi (D^{k} u) = \Pi(K^{k}[u])
\end{equation*}
and, together with Proposition \ref{P1},
\begin{eqnarray}
\nonumber
| K^k  [u] | &\leq& \sum_{r=1}^{k-1} \frac{(k-1)!}{r!(k-r-1)!}. \frac{(r-1)!}{r^{2}}. \frac{(k-r)!}{(k-r+1)^{2}} . M_{r-1}M_{k-r} L^{k-1} \|u\|^{k+1} \\
\nonumber
&\leq& \frac{k!  M_k L^k }{(k+1)^2}  \| u \|^{k+1} L^{-1}
\sum_{r=1}^{k-1} \  \frac{k-r}{kr}  \left( \frac{k+1 }{r(k-r+1)}\right)^2 .
\end{eqnarray}
Taking
\begin{equation} \label{DefCOmega}
C_{\Omega}:= \left(\sum_{i=1}^{g} |\Gamma_{i}|^{2}\right)^{1/2},
\end{equation}
where $|\Gamma_{i}|$ is the length of $\Gamma_{i}$, we deduce that
\begin{equation*}
|\Pi(K^{k}[u])|_{R^{g}} \leq C_{\Omega} \frac{k!  M_k L^k }{(k+1)^2}  \| u \|^{k+1} L^{-1}
\sum_{r=1}^{k-1} \  \frac{k-r}{kr}  \left( \frac{k+1 }{r(k-r+1)}\right)^2 .
\end{equation*}
Plugging the previous bounds into inequality (\ref{indu2}), using \eqref{P1f}-\eqref{P2f}-\eqref{P4f} and \eqref{t3.4}, and thanks to (\ref{le}) we get 
\begin{eqnarray}\label{new1.7}
\| D^{k} u \|  \leq c_{\mathfrak r} \gamma(L) \frac{k! M_k L^k }{(k+1)^2} \|u\|^{k+1}  
\leq \frac{1}{3} \frac{k! M_k  L^k }{(k+1)^2} \|u\|^{k+1} .
\end{eqnarray}
Finally going back to (\ref{t3.4}) to estimate the pressure we get (\ref{indu}) at rank $k$ and Theorem \ref{start2} is proved. \par
%
%
%
%
%
%
\subsection{Proof of Corollary \ref{Hadamard}}
\label{Pr}

Let us denote  by $ \Phi [u_0]$ the flow associated to an initial velocity $u_0 \in C^{\lambda+1,r}_{\sigma,R}(\Omega)$. We consider $\overline{R}$ such that
\begin{equation*}
\sup_{t \in (-T,T)} \| u(t) \|_{C^{\lambda+1,r}} \leq \overline{R}.
\end{equation*}
By the time-invariance of the equation, it is sufficient to prove that there exists $T_a >0$ depending on $\overline{R}$ and $\Omega$ only, such that $u_0 \in C^{\lambda+1,r}_{\sigma,\overline{R}}  (\Omega ) \mapsto   \Phi [u_0] \in  C^\omega  ((-T_a,T_a) ;C^{\lambda+1,r}  (\Omega ) )$ is  $C^\infty$. \par
According to Theorem \ref{start2}, $\Phi [u_0] \in C^\omega  ((-T,T) ;C^{\lambda+1,r} (\Omega )  )$, so 
 there holds for $(t,x)  \in (-T,T) \times \Omega $, 
\begin{equation} 
\label{serie}
 \Phi [u_0] (t,x) = \sum_{k  \geq 0}  \Phi_k [u_0] (t,x)  .
   \end{equation}
where 
\begin{eqnarray*} 
 \Phi_k [u_0] (t,x) =  \frac{t^k}{k!} (\partial_t^k  \Phi) [u_0] (0,x)
 = 
 \left\{\begin{array}{ll} 
 x , \text{ if } k = 0,
 \\  \frac{t^k}{k!} (D^{k-1} u ) (0,x) ,   \text{ if } k  \geq 1.
 \end{array}\right.
   \end{eqnarray*}
Proceeding by iteration as in Section \ref{Proof1}, we obtain that for any $k
\geq  1$, the operator $u_0 \in C^{\lambda+1,r}  (\Omega ) \mapsto
(D^{k-1} u ) (0,\cdot )  \in C^{\lambda+1,r}  (\Omega ) $ is the
restriction to the diagonal of a $k$-linear continuous operator from
$C^{\lambda+1,r}  (\Omega )^k $ to  $C^{\lambda+1,r}  (\Omega ) $,
with the estimate  
\begin{equation*}
\| (D^k u  )|_{t= 0} \|  \leq \frac{k! L^k }{(k+1)^2} \|u_0 \|^{k+1} .
\end{equation*}
Therefore,  for any $k   \geq  0$, the mapping $u_0 \in C^{\lambda+1,r}_{\sigma}  (\Omega ) \mapsto  \Phi_k [u_0] \in C^\omega  ( \R ;C^{\lambda+1,r}  (\Omega ))$ is  $C^\infty$ and
there exists $T_a> 0$ depending only on $\Omega$ and $\overline{R}$ such that the series 
\begin{equation*}
\sum_{k  \geq 0} \|    \Phi_k [u_0]     \|_{ L^\infty (B_\C (0,T_a ),  C^{\lambda+1,r}  (\Omega )) }
\end{equation*}
converges.
Since the $l$-th order derivatives with respect to $u_0$ of $\Phi_k [u_0]$ can be bounded as above with an extra multiplicative constant $k^l$, the series 
\begin{equation*}
\sum_{k  \geq 0} \|  D^l  \Phi_k [u_{0,1} ,\cdot \cdot \cdot  ,  u_{0,l} ]     \|_{ L^\infty (B_\C (0,T_a ),  C^{\lambda+1,r}  (\Omega )) }
\end{equation*}
also converge. 
We obtain that $u_0 \in C^{\lambda+1,r}_{\sigma,R}  (\Omega ) \mapsto   \Phi [u_0] \in  C^\omega  ((-T_a,T_a) ;C^{\lambda+1,r}  (\Omega ) )$ is  $C^\infty$.
Repeating the same process on a finite number of small time intervals yields the result.

\section{Skeleton of the proof of Theorem \ref{start4}}
\label{Sec:SchemeDePreuve}
Before entering the core of the proof, let us explain its general
strategy as a motivation for the next sections. In what follows, $T>0$
is chosen suitably small so that the distance from ${\mathcal S}(t)$
to $\partial \Omega$ is bounded from below by a positive real number
$\underline{d}$ and so that we have a uniform constant for the $\div$-$\curl$
elliptic estimate on $\Omega \setminus {\mathcal S}(t)$ for $t \in (-T,T)$,
see Lemma \ref{AutreRegdivcurl} below. \par 
As in the proof of Theorem \ref{start2} we introduce $\rho_{\Omega}$
as a function defined on a neighborhood of $\partial \Omega$ as the
signed distance to $\partial \Omega$, negative inside $\Omega$. 
Since here the boundary $\partial \Omega$ is analytic, there exists
$c_\rho > 1$ such  that for all $s \in \mathbb{N}$, 
\begin{equation} \label{RhoAnalytique}
\|\nabla^s \rho_{\Omega} \| \leq c_\rho^s \, s!.
\end{equation}
The norm in \eqref{RhoAnalytique} is the $C^{\lambda+1,r}$ norm in the
above neighborhood. We also introduce an analytic function
$\rho_{B}(t,x)$ defined on a neighborhood of the body's boundary
$\partial  S (t)$ as the signed distance function to $\partial S (t)$
(let us say, positive inside $S(t)$), so that the inward  unit normal
to the body boundary is $n(t,x) :=  \nabla  \rho(t,x)$ (defined in a
neighborhood of $\partial S(t)$). We denote by $\rho_0$ the initial
value of $\rho$ which is therefore an analytic function $\rho_0 (x)$
defined on a neighborhood of the  body's boundary $\partial  S_0 $ at
initial time and satisfying $\rho_0 (x)=\dist(x,\partial S_0)$. Note
that the norms $\| \nabla^{s} \rho \|$ are independent of $t$ (see
\eqref{RhoRho}-\eqref{XChapeau} below). The analytic estimates on
$\partial S$  read 
\begin{equation}
\label{laciao}
\|\nabla_{x}^{s}\, \rho_{B} \| \leq c_\rho^s \, s! .
\end{equation}
The norm considered in \eqref{laciao} is again the $C^{\lambda+1,r}$
norm in the neighborhood where $\rho_{B}$ is defined. In the sequel we
will omit to write the index $x$ in the derivation of $\rho_{B}$. \par 
\ \par
As in Section \ref{Sec:FluideSeul}, we consider solutions $(\ell,r,u)$ which are smooth. For this, we proceed by regularization and prove estimates which not depend on the regulartization parameter. Note that the passage to the limit requires Proposition \ref{Prostart3}. \par
As for Theorem \ref{start2}, the goal is to prove by induction an estimate on the $k$-th material derivative of the fluid and the body velocities. Precisely, what we want to prove is the following inequality: there exists $L>0$ such that for all $k \in \N$,
\begin{equation} \label{Graal}
 \|D^k u\| + \|\ell^{(k)}\| + \|r^{(k)}\|
\leq  \V_k   ,
\text{ with }
\V_k := \frac{k! L^k}{(k+1)^{2}} \V^{k+1} ,
\text{ where }
\V := \|u\|+\|\ell\| + \|r\| .
\end{equation}
The norm on vectors (here $\ell$, $r$ and their derivatives) of $\R^{3}$ is the usual Euclidean one.  We will also the notation $ \|\cdot  \|$ for the associated  operator norm.
Here the spaces and norms are the following:
\begin{gather*}
X(t):=C^{\lambda,r}({\mathcal F}(t)), \, |\cdot|:=\|\cdot\|_{X(t)}, \quad
\widetilde{X} (t):=C^{\lambda,r}(\partial{\mathcal S}(t)), \, |\cdot|_{\partial{\mathcal S}(t)}:=\|\cdot\|_{{\widetilde X}(t)}, \, \\
X_{\partial \Omega}:=C^{\lambda,r}(\partial \Omega), \, |\cdot|_{\partial \Omega}:=\|\cdot\|_{X_{\partial \Omega}}, \\
Y(t):=C^{\lambda+1,r}({\mathcal F}(t)), \, \|\cdot\|:=\|\cdot\|_{Y(t)}, \quad
\widetilde{Y} (t) :=C^{\lambda+1,r}(\partial{\mathcal S}(t)), \, \|\cdot\|_{\partial{\mathcal S}(t)}:=\|\cdot\|_{{\widetilde Y}(t)}, \, \\
Y_{\partial \Omega}:=C^{\lambda+1,r}(\partial \Omega), \, \|\cdot\|_{\partial \Omega}:=\|\cdot\|_{Y_{\partial \Omega}}.
\end{gather*}
The inequality  \eqref{Graal}  is true for $k=0$. 
Now in order to propagate the induction hypothesis we will proceed in two parts looking first at the estimates of the pressure and then deducing estimates for the velocities of the solid and of the fluid.
These two steps are summed up into two propositions below. Their proof is based on estimates of the pressure. The idea is to decompose the pressure into pieces which we estimate separately. 
\begin{Lemma}
Equation \eqref{Euler1a2} can  be written as
\begin{equation}
	Du=-\nabla \mu + \nabla ( \Phi \cdot \begin{bmatrix}  \ell \\ r \end{bmatrix}' ),
	\label{new2.0}
\end{equation}
with $\Phi:=(\Phi_{i})_{i=1\dots 6}$, where the functions $\Phi_{i}$ and $\mu$ are the solutions of the following problems: 
\begin{equation}
	-\Delta \Phi_i = 0 \quad   \text{for}  \ x\in \mathcal{F}(t),
	\label{t1.3}
\end{equation}
\begin{equation}
	\frac{\partial \Phi_i}{\partial n}=  0 \quad  \text{for}  \ x\in \partial \Omega, 
	\label{t1.4}
\end{equation}
\begin{equation}
	\frac{\partial \Phi_i}{\partial n}=K_i 
		  \quad  \text{for}  \  x\in \partial \mathcal{S}(t),
	\label{t1.5}
\end{equation}
\begin{equation} \label{IntPhiNul}
\int_{\mathcal{F}(t)} \Phi_{i} \ dx =0,
\end{equation}
%
where
\begin{equation} \label{t1.6}
K_i:=\left\{\begin{array}{ll} 
n_i & \text{if} \ i=1,2,3 ,\\ \relax
[(x- x_{B})\wedge n]_{i-3}& \text{if} \ i=4,5,6, 
\end{array}\right.
\end{equation}
and
\begin{equation}
	-\Delta \mu = \trace\{ F^{1}[u] \} = \trace\{ \nabla u \cdot \nabla u \} =  \quad   \text{for}  \ x\in \mathcal{F}(t), 
	\label{t1.0}
\end{equation}
\begin{equation}
	\frac{\partial \mu}{\partial n}= -H^1[u] = - \nabla^{2} \rho(u,u) \quad   \text{for}  \ x\in \partial \Omega, 
	\label{t1.1}
\end{equation}
\begin{equation} \label{t1.2}
\frac{\partial \mu}{\partial n}= \sigma ,  \quad   \text{for}  \ x\in \partial \mathcal{S}(t), 
\end{equation}
\begin{equation} \label{IntMuNul}
\int_{\mathcal{F}(t)} \mu \ dx =0.
\end{equation}
where $F^{1}[u]$ and $H^1[u]$ were introduced in Proposition \ref{P1}, where $v=v(t,x)$ is given by \eqref{vietendue} and where 
\begin{equation}
\sigma := \nabla^{2} \rho \, \{ u- v , u- v \} - n \cdot \big(r\wedge \left(2u-v-\ell \right)\big).	\label{new0.1}
\end{equation}
\end{Lemma}
\begin{proof}
Using  \eqref{Euler1a2} 
we have the following equations:
\begin{equation*}
	-\Delta p = \trace\{ F^{1}[u] \}
	 \quad   \text{for}  \ x\in \mathcal{F}(t), 
\end{equation*}
\begin{equation*}
	\frac{\partial p}{\partial n}= - H^1[u]  \quad   \text{for}  \ x\in \partial \Omega, 
	\label{dec2}
\end{equation*}
\begin{equation} \label{dec3}
\frac{\partial p}{\partial n}=  - n \cdot Du   \quad   \text{for}  \ x\in \partial \mathcal{S}(t) ,
\end{equation}
which determine uniquely $p$, with the extra condition:
\begin{equation} \label{dec4}
\int_{\mathcal{F}(t)} p \ dx =0.
\end{equation}
Let us now deal with the boundary condition \eqref{dec3}.
Since the motion of the body is rigid  there holds
\begin{equation} \label{RhoRho}
\rho  (t,x)  =  \rho_0  (\hat{\mathcal X}(t,x)),
\end{equation}
where 
\begin{equation} \label{XChapeau}
\hat{\mathcal X}(t,x) := x_0 + Q(t)^*(x-x_B (t)).
\end{equation} 
Hence by spatial derivation we infer that  for any  $(u^1 , u^2 ) $ in $(\R^3 )^2$,
\begin{equation} \label{nbouge}
  n  (t,x) \cdot   u^1 =  \nabla \rho_0  (\hat{\mathcal X}(t,x)) \cdot   Q(t)^* u^1 ,
\quad  \nabla^{2} \rho  (t,x) \{ u^1 , u^2 \} =  \nabla^{2} \rho_0  (\hat{\mathcal X}(t,x)) \{ Q(t)^* u^1 , Q(t)^* u^2 \}.
\end{equation}
Due to \eqref{vietendue}, \eqref{LoiDeQ} and \eqref{XChapeau}, we have
\begin{equation} \label{DeriveeXChapeau}
\frac{\partial}{\partial t}\hat{\mathcal X}(t,x) := -Q(t)^* v(t,x),
\end{equation}
Then applying a time derivative to \eqref{nbouge} and using \eqref{LoiDeQ} and \eqref{DeriveeXChapeau}, we get for any $u^1$ in $\R^3$,
\begin{equation*}
D ( n (t,x) \cdot   u^1 ) 
= - n (t,x) \cdot (r \wedge u^1 )+  \nabla^{2} \rho  (t,x) \{ u-v , u^1  \} .
\end{equation*}
We now use Leibniz's formula to get that for any smooth vector field $ \psi $
\begin{equation} \label{nDPsi}
n \cdot D \psi = D\left(n\cdot \psi\right)+ n \cdot (  r\wedge \psi ) - \nabla^2 \rho  \{ u-v,\psi \} .
\end{equation}
Next we apply this to $ \psi = u-v$ and  we use the identity
$Dv = \ell' + r' \wedge (x-x_B ) + r\wedge (u-\ell) $ (obtained from \eqref{vietendue})
 and the boundary condition  (\ref{Euler3b}) to obtain 
\begin{equation} \label{dec3bis}
\frac{\partial p}{\partial n}= \sigma - K \cdot \begin{bmatrix}  \ell \\ r \end{bmatrix}' 
   \quad   \text{for}  \ x\in \partial \mathcal{S}(t) ,
\end{equation}
where $K :=(K_{i})_{i=1\dots 6}$.
We therefore obtain that the pressure can be  decomposed  into $p=\mu-\Phi\cdot \begin{bmatrix}  \ell \\ r \end{bmatrix}'$. 
Let us stress that the functions $\Phi_{i}$ and $\mu$ are well-defined because the compatibility conditions are fulfilled; their uniqueness are granted from \eqref{IntPhiNul} and \eqref{IntMuNul}.

\end{proof}
\begin{Lemma}
The equations \eqref{Solide1}-\eqref{Solide2} can be written as
\begin{equation} \label{EvoMatrice}
\mathcal{M} \begin{bmatrix} \ell \\[0.5cm] r \end{bmatrix}' 
= \begin{bmatrix} 0 \\[0.5cm] \mathcal{J}r \wedge r \end{bmatrix} + \Xi,
\end{equation}
where 
\begin{equation*}
\Xi (t) :=  \begin{bmatrix}  \displaystyle\int_{ \mathcal{F}(t)} \nabla \mu \cdot \nabla \Phi_a \, dx   \end{bmatrix}_{a \in \{1,\ldots,6\}}, \ \  
\mathcal{M}(t) :=\mathcal{M}_1(t)+\mathcal{M}_2(t),
\end{equation*}
\begin{equation} \label{EvoMatrice2}
 \mathcal{M}_1(t) := \begin{bmatrix} m \Id_3 & 0 \\ 0 & \mathcal{J}\end{bmatrix}, \quad
\mathcal{M}_2(t) := \begin{bmatrix} \int_{\mathcal{F}(t)} \nabla \Phi_a \cdot \nabla \Phi_b \ dx \end{bmatrix}_{a,b \in \{1,\ldots,6\}},
\end{equation}
and ${\mathcal J}$ was defined in \eqref{eqJ}. Furthermore the matrix $\mathcal{M}$ is symmetric and positive definite. 
\end{Lemma}
%
%
\begin{proof}
It is sufficient to use the previous lemma and to notice that for $i \in\{1,\ldots,6\}$, for any $t \in (-T,T)$, for any  function $f \in C^1 (\overline{\mathcal{F}(t)}; \R)$,
$$
\int_{\partial \mathcal{S}(t)} K_i \ f   \ d\Gamma =\int_{\mathcal{F}(t)} \nabla \Phi_i \cdot \nabla f \ dx .
$$

\end{proof}
\begin{Remark} \rm
The matrix $\mathcal{M}$ is referred as the ``virtual  inertia tensor'', it incorporates the
 ``added  inertia tensor''  $\mathcal{M}_2$ which, loosely speaking,  measures how much the  surrounding fluid resists the acceleration as the body moves through it.
This effect was  probably first identified by du Buat in $1786$, and the efficient way of evaluating this effect through the functions  $\Phi_a $ dates back to Kirchoff.
\end{Remark}
We will first  prove the following (see Subsection \ref{Proofpesti}).
\begin{Proposition}\label{pesti}
The functions $\Phi_i$ ($i=1,\ldots,6)$ and $\mu$ satisfy the following assertions.
\begin{itemize}
	\item There exists a positive constant $C_0=C_0(\Omega,{\mathcal S}_{0},\underline{d})$ such that
$$
\sum_{1\leq i \leq 6}  \|\nabla \Phi_i \|  \leq C_0.
$$
\item There exists $\gamma_{2}$ a positive decreasing function with $\displaystyle {\lim_{L \rightarrow + \infty}} \gamma_{2}(L)=0$ such that if for all $j \leq k$, 
\begin{equation}	\label{pressurehyp1} 
\|D^ j u\|+\|\ell^{( j)}\| + \|r^{( j)}\| \leq   \V_{ j} ,
\end{equation}
then for all $1\leq j \leq k+1$, 
\begin{equation} \label{pressureclphi}
\sum_{1\leq i \leq 6}  \|D^j  \nabla \Phi_i \| \leq  \gamma_2(L) \frac{\V_{j}}{\V}.
\end{equation}
\item There exists a positive constant $C_0=C_0(\Omega,{\mathcal S}_{0},\underline{d})$ such that
\begin{equation} \label{pressuremudepart}
 \|\nabla \mu \| \leq C_0 \, \V^{2}.
\end{equation}
\item There exists $\gamma_{2}$ a positive decreasing function with $\displaystyle {\lim_{L \rightarrow + \infty}} \gamma_{2}(L)=0$ such that if for all $j \leq k$, 
\eqref{pressurehyp1} holds true then for all $1\leq j \leq k$, 
\begin{equation}\label{pressureclmu}
\|D^ j \nabla \mu \|  \leq \gamma_{2} (L) \V \, \V_{j} .
\end{equation}
\end{itemize}
%
%
\end{Proposition}
The second proposition allows to propagate the induction hypothesis. 
\begin{Proposition}\label{PropositionSolideFluide}
There exist a positive decreasing functions $\gamma_{3}$ with $\displaystyle {\lim_{L \rightarrow + \infty}} \gamma_{3}(L)=0$ such that for any $k\in \mathbb{N}^*$, if for all $j \leq k$, \eqref{pressurehyp1} holds,
then
\begin{equation} \label{MvtSolideFluide}
\|\ell^{(k+1)}\| + \|r^{(k+1)}\| + \|D^{k+1} u\| \leq  \V_{k+1} \,  \gamma_{3}(L).
\end{equation}
\end{Proposition}
The  proof of Proposition \ref{PropositionSolideFluide} is given in Subsection \ref{ProofPropositionSolide}. It consists in differentiating $k$ times relations \eqref{new2.0} and \eqref{EvoMatrice} and relies on Proposition \ref{pesti}. Once Proposition \ref{PropositionSolideFluide} established, Theorem \ref{start4} is deduced by induction in a straightforward manner.
%
%
%
%
%
%
%
%
%
%
\section{Formal identities}
\label{Sec:NewFormalIdentities}
In this section, we give several formal identities used to prove Theorem \ref{start4}. 
Some of them generalize Proposition \ref{P1} and Proposition \ref{P3}. Their proofs are given in Section \ref{Sec:PreuveP1}.
We first recall the following commutator rules to exchange $D$ and the other differentiations, which are valid for $\psi$ a scalar/vector field defined in the fluid domain:
\begin{eqnarray}
	D(\psi_1\psi_2)=(D\psi_1)\psi_2+\psi_1(D\psi_2) \label{t3.0},\\
	\nabla (D\psi) -D(\nabla \psi) = (\nabla u) \cdot (\nabla \psi), \label{t3.1}\\
	\div D\psi - D\div \psi = \trace \left\{(\nabla u)\cdot(\nabla \psi) \right\}, \label{t3.2}\\
	\curl D\psi - D\curl \psi = \as \left\{(\nabla u)\cdot(\nabla \psi) \right\}. \label{t3.3}
\end{eqnarray}
\subsection{Formal identities in the fluid and on the fixed boundary}
Let us be given a smooth vector field $\psi$. We will establish formal identities for $\div D^k \psi$, for $\curl D^k \psi$, respectively for the normal trace $n\cdot D^k \psi$ on the boundary $ \partial \Omega$, of the iterated material derivatives $ (D^k \psi)_{ k \in \N^* }$ as combinations of the functionals 
\begin{gather} \label{DefsFetHNew}
f( \theta)  [u,\psi] : = \nabla D^{\alpha_1} u \cdot \ldots \cdot \nabla D^{\alpha_{s-1}} u\cdot \nabla D^{\alpha_s} \psi, \\
\text{respectively } h( \theta)  [u,\psi] : =\nabla^s \rho \{ D^{\alpha_1} u , \ldots  ,D^{\alpha_{s-1}} u, D^{\alpha_s} \psi \},
\end{gather}
with $ \theta := (s, \alpha)$, where  $s \in \N^*$ and $ \alpha := ( \alpha_1,\ldots, \alpha_s )  \in \N^s$. Furthermore, these combinations only involve indices $(s,\alpha)$ belonging to the set $\mathcal{A}_{k}$ defined in \eqref{TheDefinitionOfA}.
The precise statement is the following.
\begin{Proposition}\label{P1New}
For $k \in \N^*$, we have in $ \mathcal{F}(t) $
\begin{eqnarray}\label{P1fNew}
\div D^k \psi=D^k\left(\div \psi\right) + \trace\left\{F^k [u,\psi]\right\} \text{ where } 
F^k [u,\psi] := \sum_{\theta   \in \mathcal{A}_{k}  } c^1_k (\theta  )  \, f(\theta)  [u,\psi], \\ 
\label{P2fNew}
\curl D^k \psi=D^k\left(\curl \psi\right) + \as\left\{G^k [u,\psi]\right\} \text{ where } 
G^k [u,\psi] :=  \sum_{\theta  \in \mathcal{A}_{k}  }  c^2_k (\theta ) \,  f(\theta)  [u,\psi],
\end{eqnarray}
where, for $i=1$, $2$, the $c^i_k (\theta )$ are integers satisfying 
\begin{equation} \label{Ci:3}
|c^i_k ( \theta) | \leqslant \frac{k ! }{\alpha ! },
\end{equation}
with $ \theta := (s, \alpha)$, and on  the boundary $ \partial \Omega$
\begin{eqnarray}\label{P4fNew}
n\cdot D^k \psi =D^k(n\cdot \psi)+ H^k [u,\psi] \text{ where } 
H^k [u,\psi] :=  \sum_{ \theta  \in \mathcal{A}_{k}  }  c^3_k (\theta )  \,  h(\theta)  [u,\psi],
\end{eqnarray}
where  the $c^3_k(\theta)$ are negative integers satisfying 
\begin{equation} \label{Ci:4}
| c^3_k (\theta)  | \leqslant \frac{k ! }{\alpha ! (s-1)!}.
\end{equation}
\end{Proposition}

The above proposition will be proved in Section \ref{Sec:PreuveP1}. It generalizes Proposition \ref{P1}, where $F^k[u]=F^k[u,u]$ (and similarly for the other functionals). \par
\ \par
We can also establish identities for the gradient $\nabla D^k\psi$ for a smooth scalar-valued function $\psi$:
\begin{Proposition}\label{P3New}
For $k\geq 1$, we have in the domain $ \mathcal{F}(t) $
\begin{equation}
D^{k} \nabla \psi = \nabla D^k\psi + K^{k}[u,\psi] ,
	\label{t3.4New}
\end{equation}
where for $k\geq 1$,
$$
K^{k}[u,\psi]=-\sum_{r=1}^{k}  \dbinom{k}{ r} \nabla D^{ r-1} u \cdot D^{k-r} \nabla \psi.
$$
\end{Proposition}
The proof of this proposition is completely identical to the proof of \cite[Prop. 3.5]{katoana} and is therefore omitted.
\subsection{Formal identities on the body boundary}
\label{bodyformal}

The aim of this section is to present some formal identities for normal traces on the boundary $ \partial S (t)$ of the rigid body of iterated material derivatives $D^{k} \psi$, and for iterated material derivatives of the functions $K_i$ defined in \eqref{t1.6}. With respect to the previous section the analysis is complicated by the dynamic of the body. 

To a vector $r \in \R^3$ we associate the operator $ \mathcal{R} (r) := r \wedge \cdot$. To any $\beta  \in   \N^s$ and $r \in C^{|\beta|}((-T,T);\R^{3})$ we define the functional $\mathcal{R}_\beta [r]$ which associates to the time-dependent function $r$ the rotation operator
\begin{equation} \label{RComposes}
\mathcal{R}_\beta [r] := \mathcal{R} ( r^{(\beta_1) } ) \circ \cdots \circ  \mathcal{R} (r^{(\beta_s) }) .
\end{equation}
For any $s  \in  \N^*$,  we will use some indices ${\bf s}' := (s'_1,...,s'_s )$ in $\N^s$. Then $s'$ will denote $s' := |{\bf s}'| =s'_1 + ... + s'_s $, $ ( \underline{\alpha}_1 ,..., \underline{\alpha}_{s} )$ will be in $\N^{s'_1} \times ... \times \N^{s'_s} $ and $\alpha := ( \underline{\alpha}_1 ,..., \underline{\alpha}_{s} , \alpha_{s' +1} , ...,\alpha_{s' +s } ) $ will be an element of $ \N^{ s' + s}$.
The  bricks of  the formal identity will be the functionals, defined for smooth vector fields $\varphi$ and $\psi$ and a multi-index $\zeta :=(s,{\bf s}',\alpha)  \in  \N^* \times \N^s \times \N^{s+s'}$:
\begin{eqnarray}
\label{defh}
h (\zeta) [r,\varphi,\psi] :=   \nabla^{s} \rho  (t,x) \{ \mathcal{R}_{\underline{\alpha}_1 } [r] D^{\alpha_{s' +1} } \varphi ,...,   \mathcal{R}_{\underline{\alpha}_{s-1} } [r] D^{\alpha_{s' +s-1} } \varphi,\mathcal{R}_{\underline{\alpha}_s } [r] D^{\alpha_{s' +s} } \psi \}  .
\end{eqnarray}
In  (\ref{defh}) the term  $\mathcal{R}_{\underline{\alpha}_i } [r] $ should be omitted when  $s'_{i} := 0 $. 
We introduce the following set 
\begin{equation} \label{Eq:Defck}
\mathcal{B}_{k} := \{ \zeta =(s,{\bf s}',\alpha)  /  \ 2  \leqslant s + s' \leqslant k + 1 \text{ and } |\alpha| + s +s' =k+1\}.
\end{equation}
We have the following formal identity. Here $\psi$ is a smooth vector field.
\begin{Proposition}
\label{Prop:BodyDirichlet}
For $k \in \N^*$, there holds on  the boundary $ \partial S(t)$
\begin{eqnarray}
\label{P4fbody}
n\cdot D^k \psi = D^k\left(n\cdot \psi\right) +H^k [r,u-v,\psi] \text{ where } 
H^k [r,u-v,\psi] := \sum_{ \zeta  \in  \mathcal{B}_{k} } \ d^{1}_k (\zeta ) \  h (\zeta) [r,u-v,\psi]  , \\ 
\label{t7.4}
D^{k} K_i = \widetilde{H}^k [r,u-v, \sigma_i] \text{ where }  \widetilde{H}^{k} [r,u-v, \sigma_i] := \sum_{ \zeta  \in  {\mathcal{B}}_{k} } \ {d}^{2}_{k} (\zeta ) \  h (\zeta) [r,u-v, \sigma_i] ,
\end{eqnarray}
where the $K_i$ are defined in \eqref{t1.6},
\begin{equation} \label{DefSigmai}
\sigma_i:= e_i  \text{ if } i=1,2,3, \text{ and }
\sigma_i:=  e_{i-3}\wedge (x-x_B) \text{ if } i=4,5,6,  
\end{equation}
and where the $d^{j}_k (\zeta)$, $j=1,2$, are integers satisfying  
\begin{equation} \label{Di}
|d^{j}_k (\zeta) | \leqslant \frac{ 3^{s+s'} k ! }{\alpha ! (s-1)!},
\end{equation}
for any $  \zeta := (s,{\bf s}',\alpha)\in \mathcal{B}_k.$
\end{Proposition}
\subsection{Estimates on the body rotation}
We  state a formal identity for the iterated time derivatives of the rotation matrix. 
\begin{Proposition}\label{Rota1}
For $k \in \N^*$, we have
\begin{equation} 
\label{Rota1f}
Q^{(k)} = \sum_{s=1}^{k} \ \sum_{\alpha \in \mathcal{A}_{k-1,s} } c_k (\alpha ) \mathcal{R}_\alpha [r] Q,
\end{equation}
where
$$
\mathcal{A}_{k,s} := \{ \alpha \in  \N^s / \ | \alpha |    = k+1 - s \}, 
$$
and where the $c_k (\alpha )$ are integers satisfying
\begin{equation} \label{Ci:5}
|c_k (\alpha ) | \leqslant \frac{(k-1) ! }{\alpha ! (s-1) !}.
\end{equation}
\end{Proposition}
\section{Proofs of the results of Section \ref{Sec:SchemeDePreuve}}
\label{Proof2}
This section is devoted to the proofs of the main steps of the proof of Theorem \ref{start4}: Proposition \ref{pesti}
and Proposition \ref{PropositionSolideFluide}. They are proved by using the formal estimates of the above section.%
%
%
%
\subsection{A regularity lemma}
To establish Propositions \ref{pesti} and \ref{PropositionSolideFluide}, we use as in the proof of Theorem \ref{start2} a regularity lemma, but here we need to take the modification of the geometry into account. Thus, we modify the regularity lemma (Lemma \ref{Lemme1} in the proof of Theorem \ref{start2}). \par
First, we establish the following.
\begin{Lemma} \label{regdivcurl}
Let ${\mathcal S}$ a regular closed subset of $\Omega$.
Let $\Gamma_i$ $(i=1,\ldots,g)$ a family of smooth oriented loops in $\overline{\Omega \setminus {\mathcal S}}$ which give a basis of the first singular homology space of $\overline{\Omega} \setminus {\mathcal S}$ with real coefficients. There exist two constants $c,C>0$ such that for any $C^{\lambda+2,r}$-diffeomorphism $\eta: \Omega \setminus {\mathcal S} \rightarrow {\mathcal G}:=\eta( \Omega \setminus {\mathcal S})$ satisfying
\begin{equation} \label{Cond:regdivcurl}
\| \eta - \Id \|_{C^{\lambda+2,r}( \Omega \setminus {\mathcal S})} <c,
\end{equation}
one has the following estimate. Let $u \in C^{\lambda,r}(\mathcal G)$ such that
$$
\div  u\in C^{\lambda,r}(\mathcal G), \quad \curl  u \in C^{\lambda,r}(\mathcal G), \quad  u\cdot n \in C^{\lambda+1,r}(\partial \mathcal G),
$$
where $n$ is the unit outward normal on $\partial {\mathcal G}$. Then $ u\in C^{\lambda+1,r}(\mathcal G)$ and
\begin{equation}
\label{reg3}
\| u \|_{C^{\lambda+1,r}({\mathcal G})} \leq  C\left( \| \div u \|_{C^{\lambda,r}({\mathcal G})} + \| \curl  u \|_{C^{\lambda,r}({\mathcal G})} + \| u \cdot n \|_{C^{\lambda+1,r}(\partial {\mathcal G})} + | \Pi_{\eta} u |  \right),
\end{equation}
where $\Pi_{\eta}$ is the mapping defined by 
$u\mapsto \left(\oint_{\eta(\Gamma_1)} u\cdot \tau d\sigma,\ldots, \oint_{\eta(\Gamma_g)} u\cdot \tau d\sigma\right)$.
\end{Lemma}
\begin{proof}
We apply the regularity Lemma \ref{Lemme1} to $u\circ \eta$:
there exists a constant $c_{\mathfrak r}$ depending only on $\Omega\setminus {\mathcal{S}}$ and $\Gamma_i$ $(1\leq i\leq g)$ such that
\begin{equation}
\label{newreg1}
\| u\circ \eta\| \leq c_{\mathfrak r} \left(|\div  (u\circ \eta)| + |\curl  (u\circ\eta)| + \| (u\circ \eta)\cdot n\| + |\Pi_{\Id} (u\circ \eta)|_{\R^{g}} \right),
\end{equation}
where $n$ is the normal on $\partial \Omega \cup \partial {\mathcal S}$. Now, using the exponent $j$ for the $j$-th coordinate, we have
\begin{equation*}
\partial_{i}(u \circ \eta) = \sum_{j} (\partial_{j} u)\circ \eta \,.\, \partial_{i} \eta^{j} \ \text{ and } \ 
(\partial_{i}u) \circ \eta = \sum_{j} (\partial_{j} u)\circ \eta \,.\, \partial_{i} \Id^{j}.
\end{equation*}
Moreover, it is clear that for $\| \eta - \Id \|_{C^{\lambda+2,r}} \leq 1/2$, one has for some constant $C>0$:
\begin{equation} \label{AvecOuSansEta}
C^{-1}|\psi| \leq |\psi \circ \eta| \leq C |\psi| \text{ and }
C^{-1} \|\psi\| \leq \|\psi \circ \eta\| \leq C \|\psi\|,
\end{equation}
It follows that form some constant $C>0$ (and $c \leq 1/2$):
\begin{equation} \nonumber
|\div  (u\circ \eta)-(\div u)\circ \eta| + |\curl (u\circ \eta)-(\curl u)\circ \eta| \leq C \| u \| \| \eta - \Id \|.
\end{equation}
Thus these terms can be absorbed by the left hand-side for $c$ small enough ($c$ being the constant in \eqref{Cond:regdivcurl}). Also, one has
\begin{gather} \label{RDCU1}
|\Pi_{\eta} (u\circ \eta) - \Pi_{\Id}(u)|_{\R^{g}} \leq C \|u\| \| \eta - \Id \|, \\
\label{RDCU2}
\| (u\circ \eta)\cdot n - (u\cdot n_{\eta}) \circ \eta \|_{C^{\lambda+1,r}(  \Omega \setminus {\mathcal S} )} \leq C \|u\| \| \eta - \Id \|_{C^{\lambda+2,r}(  \Omega \setminus {\mathcal S} )},
\end{gather}
where $n_{\eta}$ is the normal on $\partial \Omega \cup \partial [\eta({\mathcal S})]$. Indeed the normal $n_{\eta}$ can be obtained by using the differential of $\eta$ on two tangents of ${\mathcal S}$, taking the cross product and normalizing. Consequently the terms in \eqref{RDCU1}-\eqref{RDCU2} can be absorbed as well by the left hand-side. This gives
\begin{equation*}
\| u\circ \eta\| \leq c_{\mathfrak r} \left(|\div  (u) \circ \eta| + |\curl  (u) \circ\eta | + \| (u\cdot n_{\eta}) \circ \eta\| + |\Pi_{\eta}(u)|_{\R^{g}} \right).
\end{equation*}
Using again \eqref{AvecOuSansEta}, this concludes the proof of Lemma \ref{regdivcurl}.
\end{proof}
As a consequence of Lemma \ref{regdivcurl}, the constant in the elliptic estimate for the $\div$-$\curl$ system is uniform for all the domains that can be obtained from $\Omega \setminus {\mathcal S}(0)$ by moving the solid ${\mathcal S}$ inside $\Omega$ while keeping a minimal distance from ${\mathcal S}$ to the boundary. This is given in the following statement.
\begin{Lemma} \label{AutreRegdivcurl}
For $\varepsilon>0$, define
\begin{equation} \label{DEpsilon}
D_{\varepsilon} := \left\{ \tau \in SE(3) \ / \ \exists \gamma \in C^{0}([0,1],SE(3)) \text{ s.t. }
\gamma(0)= \Id, \ \gamma(1)=\tau, \ d(\gamma(t)[{\mathcal S}(0)], \partial \Omega) \geq \varepsilon \text{ in } [0,1]
\right\}.
\end{equation}
Choose the family $\Gamma_i$ $(i=1,\ldots,g)$ giving a homology basis of $\overline{\Omega} \setminus {\mathcal S}(0)$, inside $\partial \Omega \cup \partial {\mathcal S}(0)$, let us say $\Gamma_{1}, \dots, \Gamma_{k} \subset \partial \Omega$ and $\Gamma_{k+1}, \dots, \Gamma_{g} \subset  \partial{\mathcal S}(0)$. 
Then one can find a constant $c_{\mathfrak r}>0$ such that \eqref{reg3} is valid for all ${\mathcal G}= \Omega \setminus \tau({\mathcal S}(0))$, uniformly for $\tau \in D_{\varepsilon}$, where $\Pi$ is defined by  $u\mapsto \left(\oint_{\Gamma_1} u\cdot \tau d\sigma,\ldots,\oint_{\Gamma_k} u\cdot \tau d\sigma, \oint_{\tau(\Gamma_{k+1})} u\cdot \tau d\sigma, \ldots, \oint_{\tau(\Gamma_{g})} u\cdot \tau d\sigma \right)$.
\end{Lemma}
\begin{Remark} \label{RemDivCurl}
We also obtain the inequality with a uniform constant when we replace the curves $\tau(\Gamma_{k+1})$, \dots, $\tau(\Gamma_{g})$ by homotopic curves $\tilde{\Gamma}_{k+1}$, \dots, $\tilde{\Gamma}_{g}$ in $\partial \tau({\mathcal S})$. It is a direct consequence of the fact that the difference of circulations around $\tau(\Gamma_{i})$ and $\tilde{\Gamma}_{i}$ is obtained by the flux of $\curl u$ across the part of ${\mathcal S}$ between $\tau(\Gamma_{i})$ and $\tilde{\Gamma}_{i}$.
\end{Remark}

We will need the following.
\begin{Lemma} \label{LemmeDiffeoReg}
Let $\lambda \in \N$, $r \in (0,1)$, $\Omega$ and ${\mathcal S} \subset \Omega$ a smooth closed domain be given. Let $c>0$. There exists a neighborhood ${\mathcal U}$ of $\Id$ in $SE(3)$ such that for any $\tau \in {\mathcal U}$, there exists $\eta \in C^{\infty}(\overline{\Omega};\R^{3})$ a smooth diffeomorphism sending  $\overline{\Omega} \setminus \tau({\mathcal S})$ into $\overline{\Omega} \setminus {\mathcal S}$, such that $\eta=\Id$ in the neighborhood of $\partial \Omega$ and $\eta=\tau$ in the neighborhood of $\partial {\mathcal S}(0)$, and satisfying
\begin{equation} \label{PresID}
\| \eta - \Id \|_{C^{\lambda+2,r}} < c.
\end{equation}
\end{Lemma}
\begin{proof}
Denote $\underline{d}:=d({\mathcal S}, \partial \Omega)$, and
\begin{equation*}
{\mathcal H}_{r}:= \left\{ x \in \Omega \ / \ d(x, \partial {\mathcal S}) <r \right\}.
\end{equation*}
Fix $r \in (0,\frac{\underline{d}}{2})$ such that ${\mathcal H}_{r}$ is a tubular neighborhood of $\partial {\mathcal S}$. This allows (for instance) to define $\varphi \in C^{\infty}_{0}(\R^{3},\R)$ such that
\begin{equation*}
\varphi(x) = 1 \text{ on } {\mathcal H}_{r/3} \text{ and } \varphi(x) = 0 \text{ on } \R^{3} \setminus {\mathcal H}_{2r/3}.
\end{equation*}
For $\kappa>0$, we define (considering temporarily $\tau$ as a $C^{1}$ function on $\Omega$)
\begin{equation*}
{\mathcal U} := \left\{ \tau \in SE(3) \ \Big/ \ \| \tau - \Id \|_{C^{1}(\overline{\Omega})} < \min \left(\frac{r}{3}, \kappa\right)  \right\}.
\end{equation*}
Given $\tau \in {\mathcal U}$, let
\begin{equation*}
\eta(x) := (1-\varphi(x)) x + \varphi(x) \tau(x).
\end{equation*}
Clearly, for $\kappa>0$ small enough, $\eta$ is a diffeomorphism of $\R^{3}$, and hence a diffeomorphism of $\Omega$ on its image, satisfying \eqref{PresID}. Also, $\eta$ equals $\Id$ on $\R^{3} \setminus {\mathcal H}_{2r/3}$ which is a neighborhood of $\partial \Omega$ and $\tau$ on ${\mathcal H}_{r/3}$ which is a neighborhood of $\partial {\mathcal S}$. For $x \in {\mathcal H}_{2r/3} \setminus {\mathcal H}_{r/3}$, we see that $\eta(x) \in \Omega \setminus \tau({\mathcal S})$, hence $\eta$ is a diffeomorphism from $\Omega \setminus {\mathcal S}$ to $\Omega \setminus \tau({\mathcal S})$.
\end{proof}
\begin{proof}[Proof of Lemma \ref{AutreRegdivcurl}]
Since $\Omega$ is bounded, it is clear that $D_{\varepsilon}$ is compact (to prove that it is closed, one can for instance parameterize the curves $\gamma$ in order that $|\dot{\gamma}| \leq K$ where $K$ depends on the geometry only). For each $\tau \in D_{\varepsilon}$, apply Lemma \ref{LemmeDiffeoReg} with ${{\mathcal S}}=\tau({\mathcal S}(0))$ and $c$ such that Lemma \ref{regdivcurl} applies. A vicinity of $\tau \in SE(3)$ is composed of $\{ h \circ \tau, \ h \in {\mathcal U}\}$. Extract a finite subcover. This gives the claim since any $\tau \in D_{\varepsilon}$ can be connected to $\Id$ through a finite number of these vicinities.
\end{proof}

%
%
%
%
%
%
%
\subsection{Proof of Proposition \ref{pesti}}
\label{Proofpesti}
The functions $\nabla \Phi_i$ ($i=1,...,6$) defined by \eqref{t1.3}--\eqref{t1.5}, satisfy  
\begin{equation}
\div \nabla \Phi_i=0 \quad \text{in} \ \mathcal{F}(t), \quad \curl \nabla \Phi_i=0 \quad \text{in} \ \mathcal{F}(t), \quad n\cdot \nabla \Phi_i = K_i \quad \text{on} \ \partial \mathcal{S}(t),
\quad n\cdot \nabla \Phi_i = 0 \quad \text{on} \ \partial \Omega.
	\label{new0.4}
\end{equation}
Then by applying the regularity lemma (Lemma \ref{AutreRegdivcurl}), we obtain 
\begin{equation}
\|\nabla \Phi_i\| \leq C_0,
	\label{new0.6}
\end{equation}
where $C_0$ is a positive constant depending only on the geometry. \par
To prove the second point of the proposition, we proceed by induction. Assume that \eqref{pressureclphi} holds for all indices up to $j\leq k$. Let us prove that it holds at the index $j+1$. \par
\ \par
By applying $D^{j+1}$ to \eqref{new0.4} and by using Propositions \ref{P1New}, \ref{P3New}, and \ref{Prop:BodyDirichlet} we obtain that
$D^{j+1} \nabla \Phi_i$ satisfies the following relations
\begin{gather}
\div D^{j+1} \nabla \Phi_i=\trace\left\{F^{j+1} [u,\nabla \Phi_i]\right\} \quad \text{in} \ \mathcal{F}(t), 
\quad \curl D^{j+1}\nabla \Phi_i=\as\left\{G^{j+1} [u,\nabla \Phi_i  ]\right\} \quad \text{in} \ \mathcal{F}(t), \label{new0.5}\\
 n\cdot D^{j+1}\nabla \Phi_i = D^{j+1} K_i+H^{j+1}[r,u-v,\nabla \Phi_i] \quad \text{on} \ \partial \mathcal{S}(t),
\quad n\cdot D^{j+1}\nabla \Phi_i = H^{j+1}[u,\nabla \Phi_i] \quad \text{on} \ \partial \Omega.	
	\label{new0.3}
\end{gather}
Using these relations and Lemma \ref{AutreRegdivcurl}, we obtain
\begin{multline}\label{rt0.0}
	\|D^{j+1} \nabla \Phi_i\|
	\leq c_{\mathfrak r}  \big(  |F^{j+1}  [u,\nabla \Phi_i ] | +   |G^{j+1}  [u,\nabla \Phi_i ] | +   \|H^{j+1}  [u,\nabla \Phi_i] \|_{\partial \Omega}\\ 
	+ \|D^{j+1} K_i\|_{\partial S(t)}+ \|H^{j+1}  [r,u-v,\nabla \Phi_i ] \|_{\partial S(t)} 
+  | K^{j+1}  [u, \nabla \Phi_i ] | \big)
\end{multline}
Then we can proceed as in the proof of  Theorem  \ref{start2} and by using that \eqref{pressurehyp1} and \eqref{pressureclphi} hold for all indices up to $j\leq k$, we deduce 	
\begin{equation}\label{rt0.1}
|F^{j+1}  [u,\nabla \Phi_i ] | +   |G^{j+1}  [u,\nabla \Phi_i ] | +   \|H^{j+1}  [u,\nabla \Phi_i ] \|_{\partial S(t)}  +  | K^{j+1}  [u, \nabla \Phi_i ] |
\leq \frac{C_0}{\V} \gamma(L)\V_{j+1},
\end{equation}
where $\gamma$ is defined by \eqref{Eq:DefGammaL}.
On the other hand, using Proposition \ref{Prop:BodyDirichlet},
\begin{equation}
	\|H^{j+1}  [r,u-v,\nabla \Phi_i ] \|_{\partial {S}(t)} \leq \sum_{ \zeta  \in  \mathcal{B}_{j+1} } \ \frac{ 3^{s+s'} (j+1) ! }{\alpha ! (s-1)!} \  \|h (\zeta) [r,u-v,\nabla \Phi_i]\|_{\partial {S}(t)},
	\label{t6.0}
\end{equation}
where $\zeta :=(s,{\bf s}',\alpha)$ and $h (\zeta) [r,u-v,\nabla \Phi_i]$ is defined in \eqref{defh}.
To estimate the body velocity $v$ in $h (\zeta) [r,u-v,\nabla \Phi_i]$, we will use the following result.
\begin{Lemma}\label{bodyvesti}
Under the same assumptions as Proposition \ref{pesti}, there exists a geometric constant $C(\Omega)> 1$ such that for any $m \leq k$
\begin{equation}
\|D^{m} \, v\| \leq   C(\Omega)   \V_{m}.
\end{equation}
\end{Lemma}
\begin{proof}
For any $m \in \N^*$ applying $D^m$ to the equation \eqref{vietendue} yields
\begin{eqnarray*}
D^m v &=& \ell^{(m)} + r^{(m)}\wedge (x- x_{B})+ \sum_{l=0}^{m-1} \binom{m}{l} r^{(l)} \wedge \left(D^{m-l-1} u-\ell^{(m-l-1)}\right).
 \end{eqnarray*}
Consequently, using the fact that \eqref{pressurehyp1} is valid for indices $1,\dots, k$, we deduce that for any $m \leq k$,
\begin{eqnarray*}
\|D^{m} v\| &\leq& C(\Omega) \V_{m} + 2\sum_{l=0}^{m-1} \frac{m!}{(m-l)!\, l!} . \frac{l!\, L^{l}\V^{l+1}}{(l+1)^{2}} . \frac{(m-l-1)!\, L^{m-l-1}\V^{m-l}}{(m-l)^{2}} \\
&\leq & \left(C(\Omega) + \left[2\sum_{l=0}^{m-1} \frac{(m+1)^{2}}{(m-l)^{3}(l+1)^{2}} \right] \frac{1}{L}\right)
\frac{m! L^m }{(m+1)^2} \V^{m+1} 
 \leq C'   \V_{m},
\end{eqnarray*}
by noticing that the term inside brackets is bounded in $m$, as seen by distinguishing $l \geq m/2$ and $l \leq m/2$.
\end{proof}
Now from \eqref{RComposes} and the fact that \eqref{pressurehyp1} is true for indices $1,\dots,k$, we deduce the  following relation for $\beta=(\beta_1,\ldots,\beta_s)\in \mathbb{N}^s$ such that $|\beta| \|eq k$:
\begin{equation}
\| \mathcal{R}_{\beta}[r] \|\leq \beta! \, L^{|\beta|}  \,  \Upsilon(s,\beta)\V^{|\beta|+s},
\label{t7.1}
\end{equation}
(recall that $\Upsilon$ was defined in \eqref{DefUpsilon}). Hence using \eqref{laciao} and \eqref{defh} we deduce that for $\zeta \in {\mathcal B}_{j}$, we have
\begin{eqnarray}
\nonumber
\|h (\zeta) [r,u-v,\nabla \Phi_i]\|_{\partial S(t)} &\leq & c_{\rho}^{s} \, \prod_{i=1}^{s} \| {\mathcal R}_{\alpha_{i}}[r] \| \, \prod_{i=1}^{s} (\| D^{\alpha_{s'+i}} u \| + \| D^{\alpha_{s'+i}} v \|) \\
\nonumber
&\leq & c_{\rho}^{s} \, \prod_{i=1}^{s} \left[\underline{\alpha}_{i}!L^{|\underline{\alpha}_{i}|} \V^{|\underline{\alpha}_{i}|+s'_{i}} \prod_{m=1}^{s'_{i}} \frac{1}{(1+\underline{\alpha}_{i,m})^{2}}\right]
 \, 
\prod_{i=1}^{s} (1+C(\Omega)) \frac{\alpha_{s'+i}! L^{\alpha_{s'+i}} \V^{\alpha_{s'+i}+1} }{(1+\alpha_{s'+i})^{2}} \\
&\leq & \left[c_\rho (1+C(\Omega))\right]^s  \,  s! \,  \Upsilon(s+s',\alpha) \, \alpha! \, L^{|\alpha|}  \, 
\V^{j+1} \frac{C_0}{\V},\label{t6.1}
\end{eqnarray}
where $\underline{\alpha}_{i,m}$ is the $m$-th term in $\underline{\alpha}_{i}$. \par
Combining  \eqref{t6.0} and  \eqref{t6.1} yields
$$
\|H^{j+1}  [r,u-v,\nabla \Phi_i ] \|_{\partial S(t)} \leq  \sum_{2\leq s+s'\leq j+2} \frac{3^{s+s'} s \left[c_\rho (1+C(\Omega))\right]^s}{L^{s+s'-1}}  \sum_{|\alpha|=j+2-s-s'} \Upsilon(s+s',\alpha) 
(j+1)! \, L^{j+1} \, \V^{j+2} \frac{C_0}{\V}.
$$
Applying Lemma \ref{LemmeCheminSMF} in the above inequality implies 
$$
\|H^{j+1}  [r,u-v,\nabla \Phi_i ] \|_{\partial S(t)} \leq \sum_{2\leq s+s'\leq j+2} \frac{\left[c_\rho (1+C(\Omega))\right]^{s+s'} (s+s') }{L^{s+s'-1}} \left(\frac{j+2}{j+3-s-s'} \right)^2
  \V_{j+1} \frac{C_0}{\V}.
$$
We notice that for $2\leq j'\leq j+2$ we have
%
$\left(\frac{j+2}{j+3-j'} \right)^2 \leq 4 j'^{2}$,
%
by distinguishing $j' \geq (j+2)/2$ and $j' < (j+2)/2$. Hence we can set
$$
\tilde{\gamma}_1(L):= 
4 \sum_{j'\geq 2 } \frac{\left(60 c_\rho (1+C(\Omega)) \right)^{j'} j'^{3} }{L^{j'-1}},
$$
and deduce
\begin{equation}
\|H^{j+1}  [r,u-v,\nabla \Phi_i] \|_{\partial S(t)}  \leq    \tilde{\gamma}_1 (L)
\V_{j+1} \frac{C_0}{\V}.
	\label{rt0.2}
\end{equation}
%
%
%
%
%
For what concerns the term $D^{j+1} K_i$ we may apply Proposition \ref{Prop:BodyDirichlet} (giving the same estimates for $H^{k}$ and $\tilde{H}^{k}$) to deduce in the same manner
\begin{equation}\label{rt0.1bis}
\|D^{j+1} K_i\|_{\partial S(t)} \leq \frac{C_0}{\V} \tilde{\gamma}_1 (L)\V_{j+1}.
\end{equation}
As a consequence, combining \eqref{rt0.0}, \eqref{rt0.1} and \eqref{rt0.2} yields
\begin{equation}
	\|D^{j+1} \nabla \Phi_i \| \leq c_{\mathfrak r} {\gamma}_1 (L) \frac{C_0}{\V} \V_{j+1},
	\label{rt0.3}
\end{equation}
with $\gamma_{1}:=2\tilde{\gamma}_{1}+ \gamma$. Hence we obtain the second point of the Proposition for
$$
\gamma_2 = c_{\mathfrak r} C_0 \gamma_1.
$$
\ \par
We now turn to the claims concerning $\mu$. The function $\nabla \mu$ defined by \eqref{t1.0}--\eqref{t1.2} satisfies
\begin{gather*}
\div \nabla \mu=-\trace\left\{F^1[u,u]\right\}= -\trace\left\{\nabla u \cdot \nabla u \right\}, \quad \curl \nabla \mu=0 \quad \text{in} \ \partial \mathcal{F}(t), \\
\quad n\cdot \nabla \mu =  \sigma  \quad \text{on} \ \partial \mathcal{S}(t),
\quad n\cdot \nabla \mu = -H^1[u,u] \quad \text{on} \ \partial \Omega,
\end{gather*}
where $\sigma$ is defined by \eqref{new0.1}. Hence \eqref{pressuremudepart} follows again from Lemma \ref{AutreRegdivcurl}. \par
The proof that the validity of \eqref{pressurehyp1} for $j \leq k$ implies the one of \eqref{pressureclmu} for $1\leq j\leq k$ is completely similar to the equivalent proof for $\Phi_{i}$. It is mainly a matter of considering \eqref{rt0.1} where one multiplies by $\V$ rather than dividing by it; following the same lines we reach the conclusion. \par
%
%
%
%
%
%
\subsection{Proof of Proposition \ref{PropositionSolideFluide}}
\label{ProofPropositionSolide}
We cut the proof of Proposition \ref{PropositionSolideFluide} into two pieces. Under the same assumption that \eqref{pressurehyp1} is valid for all $j \leq k$, we first prove
\begin{equation} \label{MvtSolide}
\|\ell^{(k+1)}\| + \|r^{(k+1)}\| \leq  \V_{k+1} \,  \gamma_{4}(L),
\end{equation}
and then prove
\begin{equation} \label{MvtFluide}
\|D^{k+1} u\| \leq   \,   \V_{k+1} \,   \gamma_{5}(L),
\end{equation}
for positive decreasing functions $\gamma_{4}, \gamma_{5}$ with $\displaystyle {\lim_{L \rightarrow + \infty}} \gamma_{4}(L)+ \gamma_{5}(L)=0$. \par
\ \par
Let us first prove  \eqref{MvtSolide}. Differentiating the equations \eqref{EvoMatrice} $k$ times with respect to the time yields the formal identity:
\begin{equation}
\mathcal{M}\begin{bmatrix} \ell \\ r \end{bmatrix}^{(k+1)}
= -\sum_{j=1}^{k} \binom{k}{j} \mathcal{M}^{(j)} \begin{bmatrix} \ell \\ r \end{bmatrix}^{(k-j+1)}
+ \frac{d^k}{dt^k}\begin{bmatrix} 0 \\ \mathcal{J}r \wedge r \end{bmatrix}
+   \Xi^{(k)} .	\label{t4.7}
\end{equation}
Since $Q$ is orthogonal, we have $\|Q^{(0)}\| = 1$.
Applying Proposition \ref{Rota1}, using the fact that \eqref{pressurehyp1} is valid for $1 \leq j \leq k$, and using Lemma \ref{LemmeCheminSMF}, we obtain for $j\in \{1,\dots,k\}$
\begin{eqnarray*}
\|Q^{(j)}\| &\leq&  \sum_{s=1}^{j} \ \sum_{ \alpha  \in \mathcal{A}_{j-1,s}} \frac{(j-1) ! }{\alpha ! (s-1) !} \ \alpha!L^{|\alpha|}\V^{|\alpha|+s} \,  \prod_{l=1}^{s} \frac{1}{(1+\alpha_{l})^{2}} \\
&\leq& \frac{\V_{j}}{\V} \sum_{s=1}^j \frac{1}{(s-1)!}\left(\frac{j+1}{j-s+1}\right)^2 \frac{20^s}{L^{s-2}}.
\end{eqnarray*}
Thus, for all $j\in  \{1,\dots,k\}$, we have
\begin{equation*}
\|Q^{(j)}\| \leq \hat{\gamma}(L)  \frac{\V_{j}}{\V},
\end{equation*}
with
$$
\hat{\gamma}(L)=\sup_{j\geq 1} \left(\sum_{s=1}^j 
	\frac{1}{(s-1)!}\left(\frac{j+1}{j-s+1}\right)^2 \frac{20^s}{L^{s-2}} \right).
$$ 
Now, thanks to \eqref{Sylvester},
\begin{equation*}
\mathcal{J}^{(j)}=\sum_{i=0}^j \dbinom{j}{i} \left(Q^{(j-i)}\right) \mathcal{J}_0 \left(Q^{(i)}\right)^*.	
	\label{t5.0}
\end{equation*}
Using that for some $c>0$, one has $\hat{\gamma}^{2}(L) \leq C\hat{\gamma}$ (one can take $c=1$ for $L$ large enough), it follows that for $j\in  \{1,\dots,k\}$, 
\begin{eqnarray*}
\left\| \mathcal{J}^{(j)} \right\| &\leq& c\| \mathcal{J}_0\| \hat{\gamma}(L) \sum_{i=0}^{j} \dbinom{j}{i} \, \frac{i! \,L^{i} \V^{i}}{(i+1)^{2}} \, \frac{(j-i)! \,L^{j-i} \V^{j-i}}{(j-i+1)^{2}} \\
&\leq& c\| \mathcal{J}_0\| \hat{\gamma}(L) L^{j} \V^{j} \sum_{i=0}^{j} \frac{1}{(i+1)^{2}(j-i+1)^{2}} \\
&\leq&  c\frac{2\pi^{2}}{3}\| \mathcal{J}_0\| \hat{\gamma}(L) \frac{\V_{j}}{\V},
\end{eqnarray*}
but splitting again the sum according to $i\leq j/2$ and $i>j/2$. 
From \eqref{EvoMatrice2} we deduce for $j\geq 1$, 
\begin{equation*}
\mathcal{M}_1^{(j)}=\begin{bmatrix} 0 & 0 \\ 0 & \mathcal{J}^{(j)}\end{bmatrix}
	\label{t5.1} ,
\end{equation*}
hence we obtain
\begin{equation}
	\|\mathcal{M}_1^{(j)} \|=\|\mathcal{J}^{(j)} \| \leq \hat{\gamma}_1(L) \frac{\V_{j}}{\V},
	\label{t8.4}
\end{equation}
with $\hat{\gamma}_1=c\frac{2\pi^{2}}{3}\| \mathcal{J}_0\| \hat{\gamma}(L)$.
Using the definition of $\mathcal{M}_2$ (see \eqref{EvoMatrice2}), we have
\begin{equation}
	\left[\mathcal{M}_2^{(j)}\right]_{a,b}= \sum_{i=0}^j \dbinom{j}{i} \int_{\mathcal{F}(t)} D^i \nabla \Phi_a \cdot D^{j-i}\nabla \Phi_{b} \ dx.
	\label{t8.5}
\end{equation}
By using Proposition \ref{pesti}, we deduce as above that
$$
	\|\mathcal{M}_2^{(j)} \|\leq \hat{\gamma}_2(L) \frac{\V_{j}}{\V} \quad (j\geq 1), 
$$
with $\hat{\gamma}_2=c_{2} \max(C_{0},1) \, \frac{2\pi^{2}}{3} \gamma_2$, $c_{2}$ being a constant such that $\gamma_{2}^{2} \leq c_{2} \gamma_{2}$. \par
\ \par
Fixing $\hat{\gamma}_{3}:=\hat{\gamma}_{1}+\hat{\gamma}_{2}$ we can now estimate the first term of the right hand side of \eqref{t4.7} as follows:
\begin{equation}
\left\| \sum_{j=1}^{k} \binom{k}{j} \mathcal{M}^{(j)} \begin{bmatrix} \ell \\ r \end{bmatrix}^{(k-j+1)}\right\|
\leq  \sum_{j=1}^{k} \binom{k}{j} \hat{\gamma}_3(L) \frac{\V_{j}}{\V} \V_{k-j+1}
\leq \frac{2\pi^{2}}{3} \hat{\gamma}_3(L) \V_{k+1}.
\label{t8.6}
\end{equation}
Next, we consider the term
$$
\frac{d^k}{dt^k} \left(\mathcal{J}r \wedge r\right)=\sum_{i+j+a=k} \frac{k!}{i!j!a!} \mathcal{J}^{(i)}r^{(j)} \wedge r^{(a)}.
$$
Using that \eqref{pressurehyp1} is valid for $j \leq k$ and \eqref{t8.4}, we deduce
\begin{eqnarray} \nonumber
\left\|\frac{d^k}{dt^k} \left(\mathcal{J}r \wedge r\right)\right\| &\leq&
\sum_{\substack{{i+j+a=k} \\{i\not =0}}} \frac{k!}{i!j!a!} \hat{\gamma}_{1}(L) \frac{\V_{i}}{\V} \V_{j} \V_{a} + \sum_{j+a=k} \frac{k!}{j!a!}\V_{j} \V_{a}, \\
\nonumber
&\leq& \frac{k! L^{k} \V^{k+2}}{(k+1)^{2}} 
\left( \hat{\gamma}_{1}(L) \sum_{\substack{{i+j+a=k} \\{i\not =0}}} \frac{(k+1)^{2}}{(i+1)^{2}(j+1)^{2}(a+1)^{2}} 
+ \sum_{j+a=k} \frac{(k+1)^{2}}{(j+1)^{2}(a+1)^{2}}  \right), \\
&\leq&  \frac{k! L^{k}}{(k+1)^2} \V^{k+2}
\left(9 \left(\frac{\pi^{2}}{6}\right)^{2} \hat{\gamma}_{1}(L) + \frac{2\pi^{2}}{3}, \right)
\label{t8.8} 
\end{eqnarray}
where we distinguished the cases where $i \geq k/3$, $j \geq k$ and $a \geq k/3$. \par
Finally, we estimate for $ a \in \{1,\ldots,6\}$
$$
\Xi^{(k)}_a = \sum_{j=0}^k \dbinom{k}{j}  \int_{ \mathcal{F}(t)} D^j \nabla \mu \cdot D^{k-j} \nabla \Phi_a \ dx.
$$
Applying Proposition \ref{pesti}, we deduce from the above equality, in the same way as previously, that
\begin{equation}
	\left\|   \Xi^{(k)}_a  \right\| \leq \gamma_2(L) \max(C_{0},1)  \frac{2\pi^{2}}{3} \frac{(k+1)! L^{k+1}\V^{k+2}}{(k+1)^2}.
	\label{t9.0}
\end{equation}
Gathering \eqref{t4.7}, \eqref{t8.6}, \eqref{t8.8} and \eqref{t9.0}, we obtain \eqref{MvtSolide}. \par
\ \par
In order to obtain \eqref{MvtFluide}, we write
\begin{equation}
D^{k+1} u = - D^k \nabla p = -D^k \nabla \mu + D^k \left(\nabla \Phi \cdot \begin{bmatrix}  \ell \\ r \end{bmatrix}'\right).	
	\label{new2.1}
\end{equation}
We notice that
$$
D^k \left(\nabla \Phi \cdot \begin{bmatrix}  \ell \\ r \end{bmatrix}'\right)
=\sum_{i=0}^k \binom{k}{i} D^i \nabla \Phi \cdot  \begin{bmatrix}  \ell^{(k-i+1)} \\ r^{(k-i+1)} \end{bmatrix}.
$$
Thus, by using \eqref{pressurehyp1} (valid up to rank $k$) and \eqref{pressureclphi} (valid up to rank $k+1$ due to Proposition \ref{pesti}) to estimate the terms of the above sum corresponding to $i\geq 1$ and by using \eqref{MvtSolide} for the term corresponding to $i=0$, we deduce
$$
\left\|D^k \left(\nabla \Phi \cdot \begin{bmatrix}  \ell \\ r \end{bmatrix}'\right)\right\|\leq 
\left(\frac{2\pi^{2}}{3} \gamma_2(L)+\gamma_3(L)C_0\right) \V_{k+1}.
$$
Combining the above inequality, \eqref{new2.1} and \eqref{pressureclmu} (valid up to rank $k$), we obtain \eqref{MvtFluide}, and the proof is complete.
%
%
%

%
%
%
%
%
%
%
%
%
%
%
\section{Proof of the formal identities}
\label{Sec:PreuveP1}
We introduce the following notations:  
\begin{Definition}
\label{nota}
For any $s  \in \N^*$, for any $1  \leqslant  j \leqslant s $ we define the operators $T_{s,j}$ from $\N^s$ into $\N^s$
and $\tilde{T}_{s,j}$ from $\N^s$ into $\N^{s+1}$ by setting, for any $\alpha := ( \alpha_1,\ldots, \alpha_s )  \in \N^s$,  
\begin{gather*}
T_{s,j} (\alpha) = \alpha + e_j \text{ with } e_j := ( \delta_{i,j} )_{ 1 \leqslant  i \leqslant s }, \\
\tilde{T}_{s,j} (\alpha ) = ( \beta_1 , \ldots, \beta_{s+1} ) \text{ with } \beta_i = \alpha_i \text{ if } i<j,
\beta_j =  0 \text{ and } \beta_i = \alpha_{i-1} \text{ if } i>j.
\end{gather*} \par
The operators $T_{s,j}$ and $\tilde{T}_{s,j}$ naturally generate operators $T_{j}$ and $\tilde{T}_{j}$ from $\displaystyle\coprod_{s \in \N^{*}} \N^s$ into itself (with the convention that $T_{s,j} (\alpha) = \tilde{T}_{s,j} (\alpha) = \alpha$ for $s<j$).
\end{Definition}
\subsection{Proof of Proposition \ref{P1New}}
We consider a smooth vector field $\psi$. \par
\ \\
{\it Proof of \eqref{P1fNew}.} We proceed by induction on $k$. We first deduce from (\ref{t3.2}) the relation (\ref{P1fNew}) for $k=1$ with $c^{1}_1 (2,(0,0)) = 1 $ since in this case $s$ takes only the value $2$ and the set  $\mathcal{A}_{k}$ (defined in \eqref{TheDefinitionOfA}) reduces to $\{ (2,(0,0)) \}$. 

Now let us assume that  (\ref{P1fNew}) with estimate \eqref{Ci:3} holds up to rank $k$. 
Then using the rule of commutation (\ref{t3.2}) we obtain 
$$
\div D^{k+1} \psi =D^{k+1}\left(\div \psi\right)+\trace\left\{F^{k+1} [u,\psi]\right\},
$$ 
with 
\begin{equation} \label{Dimsoir}
F^{k+1}[u,\psi]= DF^k[u,\psi] + (\nabla u)\cdot(\nabla D^k \psi).
\end{equation}
 To simplify the notations we will from now on drop the dependence of the functions $f(\theta)[u,\psi]$ on $[u,\psi]$. \par
As a consequence of the Leibniz rule (\ref{t3.0}) and the rule of commutation  (\ref{t3.1}), we immediately infer that for any $ \theta := (s, \alpha)$ with $s  \in  \N^* $  and $\alpha  \in \N^s$, the derivative w.r.t. $D$ of the functional $f(\theta)$ (defined in \eqref{DefsFetHNew}) is given by
\begin{equation}
\label{Derivef}
D f(\theta)  = \sum_{1 \leqslant  j \leqslant s} \Big[ f \big(R^{j}_a  (\theta)\big)  -  f \big(R^{j}_b (\theta)\big)  \Big],
\end{equation}
where 
\begin{equation} \label{DefRj}
R^{j}_a  (\theta) := (s, T_{j}(\alpha))
\text{ and }
{R}^{j}_b (\theta) := (s+1, \tilde{T}_{j}(\alpha)).
\end{equation}
Hence $DF^k = F_a - F_b $ with for $l=a, b$
\begin{eqnarray*}
F_l  :=
\sum_{\theta \in \mathcal{A}_{k}} \sum_{1 \leqslant  j \leqslant s} c^{1}_k (\theta  )  \, f\big(R^{j}_l (\theta)\big)   = 
\sum_{j = 1}^{k+1} \sum_{ \theta \in {\mathcal A}^{j}_{k}}  c^{1}_k (\theta)  \,  f\big(R^{j}_l (\theta)\big)  ,
\end{eqnarray*}
with 
\begin{equation*}
{\mathcal A}_{k}^{j}:= \Big\{ \theta := (s,\alpha) \in {\mathcal A}_{k} \ / \ s \geq j\Big\}.
\end{equation*}
The mappings $R^{j}_a$ and $R^{j}_b$ are injective on ${\mathcal A}_{k}^{j}$ and take values in 
\begin{equation} \label{duSeigneur}
R^{j}_a ({\mathcal A}^{j}_{k}) \subset {\mathcal A}^{j}_{k+1} \text{ and in } {R}^{j}_b({\mathcal A}^{j}_{k}) \subset {\mathcal A}^{j+1}_{k+1}, \text{ respectively}.
\end{equation}
For $l=a, b$ operating $R^{j}_l$  yields
\begin{eqnarray*}
F_l    = 
\sum_{j = 1}^{k+1}
 \sum_{\theta \in R^{j}_l  ({\mathcal A}^{j}_{k})}  c^{1,j}_{k,l} (\theta) \,  f(\theta) 
 =   \sum_{\theta \in {\mathcal A}_{k+1}}  \sum_{j \in {\mathcal J}^{k}_l  (\theta) }  c^{1,j}_{k,l} (\theta)  \, f(\theta),
\end{eqnarray*}
where
\begin{eqnarray*}
c^{1,j}_{k,l} (\theta) := c^{1}_k \big((R^{j}_l)^{-1}(\theta)\big)
\text{ and }
{\mathcal J}^{k}_l (\theta) := \Big\{ j \in \N^{*} \ / \ \theta \in R^{j}_l ({\mathcal A}^{j}_{k}) \Big\}  .
\end{eqnarray*}
Recalling \eqref{Dimsoir} we finally get that (\ref{P1fNew}) holds at the order $k+1$ when setting for $\theta \in \mathcal{A}_{k+1}$
\begin{equation*}
c^{1}_{k+1} (\theta) :=  \sum_{j \in {\mathcal J}^{k}_a (\theta) }  c^{1,j}_{k,a} (\theta) - \sum_{j \in {\mathcal J}^{k}_b   (\theta) }   c^{1,j}_{k,b} (\theta) + \delta_{(2,(0,k))}(\theta).
\end{equation*}
When $j  \in  {\mathcal J}^{k}_a (\theta) $ (respectively $j  \in  {\mathcal J}^{k}_b (\theta)$) we have, thanks to the previous steps and recalling the definition \eqref{DefRj}:
\begin{equation*}
 | c^{1,j}_{k,a} (\theta)  | \leqslant \frac{k ! }{(T_j^{-1} (\alpha) )! }= \alpha_j \frac{k !  }{\alpha ! } \ \text{ and } \ 
 | c^{1,j}_{k,b} (\theta)  | \leqslant \frac{k ! }{(\tilde{T}_j^{-1}( \alpha) )! }=  \frac{k !  }{\alpha ! } .
\end{equation*}
Moreover it is a consequence of \eqref{duSeigneur} that 
${\mathcal J}^{k}_a  (\theta) \subset \{1,\ldots, s\} \text{ and } {\mathcal J}^{k}_b (\theta) \subset \{ 1, \ldots, s-1\}$.
Hence for $3  \leqslant s  \leqslant k+2$, we have 
\begin{eqnarray*}
| c^{1}_{k+1} (\theta) | &\leq & \sum_{j \in {\mathcal J}^{k}_a (\theta) } |  c^{1,j}_{k,a} (\theta) |
+ \sum_{j \in  {\mathcal J}^{k}_b (\theta)} |  c^{1,j}_{k,b} (\theta) | \\ 
& \leq & \left( \sum_{ 1 \leq  j  \leq s} \alpha_j + s-1 \right) \frac{k ! }{\alpha ! } \\ 
& \leq & \frac{(k+1) ! }{\alpha ! } ,
\end{eqnarray*}
since  $\theta \in \mathcal{A}_{k+1}$. \par
Besides, for  $s=2$, one can see that ${\mathcal J}^{k}_b (\theta) = \emptyset$. 
In that case we have $\alpha := (\alpha_1, \alpha_2 ) $
with  $\alpha_1 + \alpha_2 = k$, so that $\alpha ! \leqslant k !$ and
\begin{eqnarray*}
| c^{1}_{k+1} (2,\alpha ) | \leq 1+ ( \alpha_1 +  \alpha_2 )\frac{k!}{\alpha!}  
\leq \frac{(k+1) ! }{\alpha !},
\end{eqnarray*}
and \eqref{P1fNew} is proved. \par
\ \par
\noindent
{\it Proof of \eqref{P2fNew}.} Proceeding as previously we can also obtain formal identities of the curl of the iterated material derivatives $ (D^k \psi)_{ k \in \N^* }$. Substituting the rule of commutation  (\ref{t3.3}) to (\ref{t3.2}) in the proof of \eqref{P1fNew} we obtain \eqref{P2fNew}. Note that here the case $k=1$ is a direct consequence of \eqref{Euler1a}. \par
\ \par
\noindent
{\it Proof of \eqref{P4fNew}.} We now establish formal identities for the normal trace $ n\cdot D^k \psi$ of  iterated material derivatives $ (D^k \psi)_{ k \in \N^* }$ on the boundary $ \partial \Omega$ as combinations of the functionals $h(\theta)$ defined in \eqref{DefsFetHNew}. \par
As $D$ is tangential to the boundary $(-T,T) \times \partial \Omega$  we infer that
\begin{equation*}
n\cdot \left(D\psi\right)=D\left(n\cdot \psi\right)  -  \nabla^2 \rho_{\Omega} \{ u,\psi\},
\end{equation*}
so that \eqref{P4fNew} holds for $k=1$ since in this case the set  $\mathcal{A}_{1} $ reduces to $\{ (2,(0,0)) \}$. Hence  $c^{3}_1 (2,(0,0)) = -1 $ and satisfies $ | c^{3}_1 (2,(0,0))  | \leqslant \frac{ 1! }{0 ! 0 !}$. \par
Now let us assume that \eqref{P4fNew} with estimate \eqref{Ci:4} holds up to rank $k$. Applying $D$ to the relation (\ref{P4fNew}) yields the relation
\begin{equation}\label{t2.6}
H^{k+1}=DH^k-  \nabla^2 \rho_{\Omega} \{ u,D^k \psi \}.
\end{equation}
Now using the chain rule we easily get that for $s \in \N^{*}$ and $\alpha := ( \alpha_1,\ldots, \alpha_s )  \in \N^s$, one has
\begin{equation} \label{DerivH}
D h(\theta) = h\big(R^{1}_{b}(\theta)\big) + \sum_{1 \leqslant j \leqslant s}  h\big(R^{j}_{a}(\theta)\big),
\end{equation}
where $R^{j}_{a}$ and $R^{1}_{b}$ were defined in \eqref{DefRj}. We deduce that
\begin{eqnarray*}
DH^k &=& \sum_{\theta \in \mathcal{A}_{k}} c^{3}_k (\theta) \,  h(R^1_b (\theta))   +
\sum_{\theta \in \mathcal{A}_{k}} c^{3}_k (\theta)  \sum_{ 1 \leqslant  j  \leqslant s}  h(R_a^j(\theta))  \\
 &=& \sum_{\theta \in \mathcal{A}_{k}} c^{3}_k (\theta)  \, h({R}^1_b (\theta))   +
\sum_{1 \leqslant j \leqslant k+1} \sum_{\theta \in \mathcal{A}^{j}_{k}} c^{3}_k (\theta)  \, h(R_a^j (\theta)) .
\end{eqnarray*}
We proceed as previously and change the index $\theta$ in the sums via the maps $R^{j}_a$ and $R^{j}_b$:
\begin{eqnarray*}
DH^k  &= &
\sum_{\theta \in R^{1}_b({\mathcal A}_{k})}  c^{3,1}_{k,b} (\theta)  \, h(\theta)   +
\sum_{ 1 \leqslant  j  \leqslant k+1} \ \sum_{\theta \in R_a^{j} (\mathcal{A}^{j}_{k})} c^{3,j}_{k,a} (\theta)  \,   h(\theta) ,
\end{eqnarray*}
where $c^{3,j}_{k,l} (\theta) := c^{3}_k \big( (R^{j}_l )^{-1} (\theta) \big)$ for $l=a,b$. 
Hence 
\begin{eqnarray*}
DH^k  &= &
\sum_{\theta \in {\mathcal A}_{k+1}} 
    \Big[ {\mathbf 1}_{{R}_b^{1}({\mathcal A}_{k})}(\theta)  \,  c^{3,1}_{k,b}  (\theta)   +
\sum_{j \in {\mathcal J}^{k}_a(\theta)} c^{3,j}_{k,a} (\theta) \Big] h(\theta),
\end{eqnarray*}
where we define for $\theta \in {\mathcal A}_{k+1}$ the set
$$
{\mathcal J}^{k}_a (\theta) := \Big\{ j \in \N^{*} \ / \ \theta \in R^{j}_a ({\mathcal A}^{j}_{k}) \Big\}.
$$
Recalling \eqref{t2.6} we get that  (\ref{P4f}) holds  at the order $k+1$ when setting for $\theta \in \mathcal{A}_{k+1}$
\begin{equation*}
c^{3}_{k+1}(\theta) =  {\mathbf 1}_{{R}_b^{1}({\mathcal A}_{k})}(\theta) \, c^{3,1}_{k,b}  (\theta) 
+ \sum_{j \in {\mathcal J}^{k}_a (\theta)}  c^{3,j}_{k,a} (\theta)
-  \delta_{(2,(0,k))}(\theta).
 \end{equation*}
Thus, for  $3  \leqslant s  \leqslant k+2$, we have using the induction hypothesis \eqref{Ci:4} on $c^{3,1}_{k,b}$ and $c^{3,j}_{k,a}$,
\begin{eqnarray*}
| c^{3}_{k+1} (\theta ) | \leqslant 
\left( \sum_{ 1 \leqslant  j  \leqslant s}
 \alpha_j
+
( s-1 ) \right)
\frac{k ! }{\alpha ! (s-1) ! } \leqslant \frac{(k+1) ! }{\alpha ! (s-1) !} ,
\end{eqnarray*}
and since the inequality also holds for $s=2$, \eqref{P4fNew}-\eqref{Ci:4} is proved at rank $k+1$. \par
%
%
%
%
%
%
%
%
%
%
\subsection{Proof of Proposition \ref{Prop:BodyDirichlet}}
\label{Pbodyformal}
We first prove a lemma. Here and in the sequel, the symbol ``$\nabla$'' refers to a derivation with respect to the variable $x$ only.
\begin{Lemma}\label{geo8}
For all $k \geq 1$,  for any  $(u^1 , \ldots, u^k ) $ in $(\R^3 )^k$   
\begin{equation}
D (\nabla^k \rho (t,x) \{ u^1 , \ldots, u^k  \}) 
= - \sum_{ 1 \leqslant  j  \leqslant  k}  \nabla^k \rho  (t,x) \{u^1 ,\ldots,\mathcal{R}(r)u^j,\ldots , u^k  \}
+  \nabla^{k+1} \rho  (t,x) \{ u-v , u^1 , ..., u^k \} .
	\label{iter}
\end{equation}
\end{Lemma}
\begin{proof}
We recall that, since the motion of the body is rigid, $\rho(t,x)$ is given by \eqref{RhoRho}-\eqref{XChapeau}. Hence by spatial derivation we infer that for any  $k\geq 1$, 
\begin{equation}
\label{justi}
\nabla^{k} \rho (t,x) \{ u^1 , \dots, u^k \} =  \nabla^{k} \rho_0  (\hat{\mathcal X}(t,x)) \{ Q(t)^* u^1, \dots, Q(t)^* u^k \}.
\end{equation}
We then apply a time derivative:
\begin{multline*}
\partial_{t}(\nabla^{k} \rho(t,x)) = 
\nabla^{k+1} \rho_{0}(\hat{\mathcal X}(t,x)) \{ \partial_{t} \hat{\mathcal X}(t,x), Q(t)^* u^1, \dots, Q(t)^* u^k \} \\
+ \sum_{ 1 \leqslant  j  \leqslant  k}  \nabla^k \rho  (t,x)  \{ Q(t)^* u^1, \dots, \left(\frac{dQ^*}{dt}\right) u^j , \dots, Q(t)^* u^k \}.
\end{multline*}
Now we use \eqref{LoiDeQ} to infer
\begin{equation*}
\frac{d Q^{*}}{dt} =-Q^{*} \cdot[r(t)\wedge \cdot].
\end{equation*}
With \eqref{DeriveeXChapeau}, and using again \eqref{justi}, this gives the result.
\end{proof}
The derivative computed in \eqref{iter} is the partial derivative of $\nabla^k \rho (t,x) \{ u^1 , \ldots, u^k  \}$ with respect to the first variables $(t,x)$, the variables between brackets being let fixed. 
Now to differentiate $h(\zeta)[r,u-v,\psi]$ (defined in \eqref{defh}), we have to take the dependence of the second group of variables into account. This is the aim of the next lemma. 
%
%
%
\begin{Lemma}\label{geo9}
For any $s  \in  \N^*$, for any ${\bf s}' := (s'_1,...,s'_s )$ in $\N^s$,  for any $\alpha   \in   \N^{ s' + s}$,  we have
\begin{equation} \label{Dh}
D h (\zeta) =
\left[\sum_{ 1 \leqslant  j  \leqslant  s' + s}    h \big(  R_{a}^j  (\zeta ) \big) \right]
+ h \big(  R_{b}  (\zeta )  \big)
-\sum_{ 1 \leqslant  j  \leqslant  s}   h \big( R_{c}^j  (\zeta) \big) ,
\end{equation}
where all the functions $h$ are evaluated at $[r,u-v,\psi]$ and where
\begin{equation} \label{RaRbRc}
R_{a}^j  (\zeta ) := \big(s, {\bf s'},T_{j} (\alpha) \big),
\ R_{b}  (\zeta ) := \big(s+1, \tilde{T}_{1} ({\bf s'}), \tilde{T}_{s' + 1} (\alpha)  \big)
\text{ and }
R_{c}^j (\zeta )  := \big(s,T_{j} ({\bf s'}),\tilde{T}_{\tau_{j-1}({\bf s'}) +1 } (\alpha) \big) 
\end{equation}
for $1  \leqslant j   \leqslant s$, where we denote $\tau_{j}({\bf s'})$ the ``position'' of the index $s'_{j}$:
\begin{equation} \label{DefTau}
\tau_{j}({\bf s'}) = \sum_{k=1}^{j} s'_{k} ,
\end{equation}
with the convention that   $\tau_{0} ({\bf s'}):= 0 $.
\end{Lemma}
\begin{proof}
Roughly speaking, the three expressions in \eqref{Dh} correspond in the use of Leibniz's rule, to derivate once more an expression in the arguments of $\nabla^{s} \rho$, to differentiate $\nabla^{s} \rho$ once more with the additional argument $u-v$ in first position, and to add a rotation factor $r \wedge \cdot$ in one of the arguments. To be more precise, using the chain rule we deduce
\begin{multline} \label{RoCalcul}
D \left[ \nabla^{s} \rho (t,x) \{ \mathcal{R}_{\underline{\alpha}_1 } [r] D^{\alpha_{s'+1} } (u-v), \dots, \mathcal{R}_{\underline{\alpha}_{s-1} } [r] D^{\alpha_{s'+s-1}} (u-v), \mathcal{R}_{\underline{\alpha}_s } [r] D^{\alpha_{s' +s} } \psi \}  \right] \\
\hspace{2cm} 
= D \left[ \nabla^{s} \rho (t,x) \right] \{ \mathcal{R}_{\underline{\alpha}_1 } [r] D^{\alpha_{s'+1} } (u-v), \dots, \mathcal{R}_{\underline{\alpha}_{s-1} } [r] D^{\alpha_{s'+s-1}} (u-v), \mathcal{R}_{\underline{\alpha}_s } [r] D^{\alpha_{s' +s} } \psi \} \hfill \\
\hspace{2cm} 
+ \sum_{1 \leq j \leq s-1} \nabla^{s} \rho (t,x) \{ \mathcal{R}_{\underline{\alpha}_1 } [r] D^{\alpha_{s'+1} } (u-v), \dots, 
D \left[\mathcal{R}_{\underline{\alpha}_{j} } [r] D^{\alpha_{s'+j}} (u-v)\right], \dots
\mathcal{R}_{\underline{\alpha}_s } [r] D^{\alpha_{s' +s} } \psi \} \hfill \\
\hspace{2cm} 
+ \nabla^{s} \rho (t,x) \{ \mathcal{R}_{\underline{\alpha}_1 } [r] D^{\alpha_{s'+1} } (u-v), \dots, 
\mathcal{R}_{\underline{\alpha}_{s-1} } [r] D^{\alpha_{s'+s-1}} (u-v)
D \left[\mathcal{R}_{\underline{\alpha}_s } [r] D^{\alpha_{s' +s} } \psi \right] \}.  \hfill
\end{multline}
Now for the last two terms in \eqref{RoCalcul}, we use that $D({\mathcal R}_{\underline{\alpha}_{j}}[r] D^{\alpha_{s'+j}} \varphi) = {\mathcal R}_{\underline{\alpha}_{j}+1}[r] D^{\alpha_{s'+j}} \varphi + {\mathcal R}_{\underline{\alpha}_{j}}[r] D^{\alpha_{s'+j+1}} \varphi)$. Hence these two terms yield the first sum in \eqref{Dh}. For what concerns the first term in \eqref{RoCalcul}, we use Lemma \ref{geo8} and obtain the second and third part of \eqref{Dh}.
\end{proof}
Let us now establish Proposition \ref{Prop:BodyDirichlet}. We first prove \eqref{P4fbody}. First, according to \eqref{nDPsi},
\begin{equation} \nonumber
n \cdot D \psi = D\left(n\cdot \psi\right)+\nabla \rho \left\{ \mathcal{R}[r] \psi\right\} - \nabla^2 \rho  \{ u-v,\psi \}.
\end{equation}
Hence \eqref{P4fbody} holds for  $k=1$. 
Now let us assume that \eqref{P4fbody} with estimate \eqref{Di} holds up to rank $k$. 
Applying $D$ to relation (\ref{P4fbody}), using Leibniz's rule and \eqref{nDPsi} yields the relation
\begin{equation}\label{t2.60}
H^{k+1}[r,u-v,\psi]=DH^k [r,u-v,\psi]  + n \cdot (r \wedge D^k  \psi )
-  \nabla^2 \rho \{ u-v ,D^k \psi \}.
\end{equation}
To simplify the notations, from now on we omit the dependence on $[r,u-v,\psi]$.
According to Lemma  \ref{geo9} there holds
\begin{eqnarray}
\label{t2.61}
DH^k = H_a  + H_b -  H_c ,
\end{eqnarray}
where
\begin{eqnarray}
\label{t2.61a}
H_a  &:=&  \sum_{\zeta\in \mathcal{B}_k} \ d^{1}_k (\zeta)   
 \sum_{ 1 \leqslant  j  \leqslant  s' + s}    h \big( R_a^j (\zeta) \big)  ,  \\ 
\label{t2.61b}
H_b  &:=&   \sum_{\zeta\in \mathcal{B}_k} \  d^{1}_k (\zeta) \  h  \big( R_b (\zeta)  \big),
 \\ 
\label{t2.61c}
H_c  &:=&  \sum_{\zeta\in \mathcal{B}_k} \   d^{1}_k (\zeta)   \sum_{ 1 \leqslant  j  \leqslant  s}   h  \big( R_c^j (\zeta) \big).
\end{eqnarray}
Define for $j \geqslant 1$
\begin{equation*}
\mathcal{B}_{a,k}^{j}:= \big\{ \zeta := (s,{\bf s}' ,\alpha)  \in \mathcal{B}_{k} \ / \ j \leqslant s+s' \big\} 
\text{ and }
\mathcal{B}_{c,k}^{j}:= \big\{ \zeta := (s,{\bf s}' ,\alpha)    \in \mathcal{B}_{k} \ / \ j \leqslant s \big\},
\end{equation*}
so that
\begin{equation*}
H_{l} = \sum_{j=1}^{k+1} \sum_{  \zeta   \in \mathcal{B}^{j}_{l,k}} \,   d^{1}_{k}  (\zeta)  \,   g \big( R_{l}^{j} ( \zeta )\big)   \text{ for } l = a,c.
\end{equation*}
We notice that the mappings $R_{a}^{j}$,  $R_{b}$ and  $R_{c}^{j}$ are injective  respectively on $\mathcal{B}^{j}_{a,k}$,  $\mathcal{B}_{k}$  and $\mathcal{B}^{j}_{c,k}$ so that
\begin{eqnarray*}
H_{l} &=& \sum_{j=1}^{k+1} \,  \sum_{\zeta \in R_{l}^{j}(\mathcal{B}^{j}_{l,k})} \,  d^{1,j}_{k,l} (\zeta)  \, g(\zeta) 
  \text{ for }  l = a,c, \\
H_{b} &=& \sum_{\zeta \in R_{b}(\mathcal{B}_{k})} \,   d^{1}_{k,b} (\zeta)  \, g(\zeta ), 
\end{eqnarray*}
  where $ d^{1,j}_{k,l} (\zeta) := d^{1}_{k} \big((R_{l}^{j})^{-1} (\zeta) \big)$ for $ l = a,c$ and $d^{1}_{k,b} (\zeta) := d^{1}_{k}  \big((R_{b})^{-1} (\zeta) \big)$.
Now we introduce the sets associated to any $\zeta \in {\mathcal B}_{k+1}$:
\begin{equation*}
{\mathcal J}_{l}^{k}(\zeta) := 
\Big\{ j \in \{ 1, \dots, k+1 \} \ / \ \zeta \in R_{l}^{j}(\mathcal{B}^{j}_{l,k}) \Big\} \text{ for } l = a,c.
\end{equation*}
As previously ${\mathcal J}_{a}^{k}(\zeta) \subset \{ 1, \dots, s+s' \}$ and ${\mathcal J}_{c}^{k}(\zeta) \subset \{ 1, \dots, s \}$.
Now putting together the relations (\ref{t2.60}), (\ref{t2.61}), (\ref{t2.61a}), (\ref{t2.61b}) and (\ref{t2.61c}), we get that  (\ref{P4fbody}) holds at the order $k+1$  when setting for $\zeta \in \mathcal{B}_{k+1} $,
\begin{equation*}
d^{1}_{k+1} (\zeta)=
 \sum_{ j \in \mathcal{J}^{k}_{a}(\zeta) }  d^{1,j}_{k,a} (\zeta)
-	 \sum_{ j \in \mathcal{J}^{k}_{c}(\zeta) }  d^{1,j}_{k,c} (\zeta)
+	 {\mathbf 1}_{R_{b}(\mathcal{B}_{k}) }(\zeta) \  d^{1}_{k,b} (\zeta)    
+\delta_{(1,(1),((0),k))}(\zeta) 
-\delta_{(2,(0,0),(0,k))}(\zeta).
\end{equation*}
Thus  for $3 \leqslant s+s' \leqslant k+2$, we have 
\begin{eqnarray*}
| d^{1}_{k+1} (\zeta) | &\leqslant  &
\sum_{j \in \mathcal{J}^{k}_{a}(\zeta)} | d^{1,j}_{k,a} (\zeta) | 
+  \sum_{j \in \mathcal{J}^{k}_{c}(\zeta)} | d^{1,j}_{k,c} (\zeta) | 
+ | {\mathbf 1}_{R_{b}(\mathcal{B}_{k}) }(\zeta) \, d^{1}_{k,b} (\zeta) |,
 \end{eqnarray*}
so that we get, using the induction hypothesis and $\mbox{Card}[{\mathcal J}_{c}^{k}(\zeta) ]\leq  s$,
\begin{eqnarray*}
| d^{1}_{k+1} (\zeta) | &\leqslant  &
\left(\sum_{j \in \mathcal{J}^{k}_{a}(\zeta)} \alpha_j \frac{3^{s+s'} k! }{\alpha ! (s-1) ! }\right)
+s \frac{3^{s+s'-1} k!}{\alpha! (s-1)! } 
+ \frac{3^{s+s'-1}k!}{\alpha! (s-2)! }
\\   
& \leqslant & 
\left( \sum_{j \in \mathcal{J}^{k}_{a}(\zeta)}  \alpha_j
+ \frac{2s-1}{3}  \right)
\frac{3^{s+s'}    k ! }{\alpha ! (s-1) ! }.
\end{eqnarray*}
Now using that for $j \in \{ 1, \dots, s+s' \}$, $\zeta \in R_{a}^{j}(\mathcal{B}^{j}_{a,k}) \Longleftrightarrow \alpha_{j} \geq 1$, we see that 
\begin{equation*}
 \sum_{j \in \mathcal{J}^{k}_{a}(\zeta)}  \alpha_j = \sum_{j=1}^{s+s'} \alpha_j = k+1-s-s'.
\end{equation*}
Since $2-s-3s' \leq -2 s' \leq 0$, we deduce
\begin{equation*}
| d^{1}_{k+1} (\zeta) | \leq \left( k+1+\frac{2-s-3s'}{3} \right) \frac{3^{s+s'} k ! }{\alpha ! (s-1)!} \leq \frac{3^{s+s'} (k+1) ! }{\alpha ! (s-1) ! }  ,
\end{equation*}
since $(s,\mathbf{s}', \alpha) \in \mathcal{B}_{k} $. \par
Now if $s+s'= 2$, $\zeta$ is not in the range of $R_{b}$ or $R_{c}$, and there holds $\alpha_1 +  \alpha_2 = k$  so that $\alpha ! \leqslant k ! $. We deduce
\begin{eqnarray*}
| d^{1}_{k+1} (\zeta) | \leqslant  
1 +  \sum_{j  \in \mathcal{J}^{k}_{a}(\zeta)}  |  d^{1,j}_{k,a} (\zeta) |
 \leqslant  1+
k \frac{3^{2}  k! }{\alpha ! }  
 \leqslant  
 \frac{9 (k+1) ! }{\alpha ! } .
\end{eqnarray*}
Hence \eqref{P4fbody} with estimate \eqref{Di} is proved at rank $k+1$, which concludes the induction. \par
\ \par
The proof of \eqref{t7.4} with estimate \eqref{Di} is similar: it suffices to notice that 
$$
D^{k} K_i = D^{k} \nabla \rho\left\{ e_i\right\} \quad (i=1, 2, 3) \quad \text{or} \quad D^{k} K_i = D^{k} \nabla \rho\left\{ e_{i-3}\wedge (x-x_B)\right\} \quad (i=4, 5, 6)
$$
where the $(e_{1},e_{2},e_{3})$ is the canonical basis of $\R^{3}$.
%
%
%
%
%
%
%
%
\subsection{Proof of Proposition \ref{Rota1}}
\label{ProofRota1}

The case $k=1$ corresponds to  \eqref{LoiDeQ}. Let us assume that identity (\ref{Rota1f}) with estimate \eqref{Ci:5} holds up to order $k$. We have by derivation that 
\begin{eqnarray*}
Q^{(k+1)}y &=&  \sum_{s=1}^{k} \ \sum_{ \alpha  \in \mathcal{A}_{k-1,s}  } c_k (\alpha ) 
\left( \sum_{j=1}^{s} \ \mathcal{R}_{T_{j} ( \alpha) }  [r] Q y  + \mathcal{R}_{\tilde{T}_{s+1} ( \alpha) }  [r] Q y \right), \\
&=&  \sum_{s=1}^{k} \ \sum_{ \alpha  \in \mathcal{A}_{k,s}  }  \ \sum_{ j \in  \mathcal{J}^k(\alpha)  } c_k (T_{j}^{-1} (\alpha ) )  \mathcal{R}_\alpha  [r] Q y +
 \sum_{s=2}^{k+1} \ \sum_{\substack{ \alpha  \in \mathcal{A}_{k,s} \\ \text{s.t. } \alpha_{s}=0 }} c_k ( \tilde{T}_{s}^{-1} (\alpha ) )  \mathcal{R}_\alpha  [r] Q y , \\
&=&  \sum_{s=1}^{k+1} \ \sum_{ \alpha  \in \mathcal{A}_{k,s}  }  \  c_{k+1}  (\alpha ) \mathcal{R}_\alpha  [r] Q y ,
\end{eqnarray*}
with
$$
\mathcal{J}^k(\alpha)=\{j\in \mathbb{N}^* / \ \alpha\in T_j(\mathcal{A}_{k-1,s})\}
$$
and with
\begin{equation*}
c_{k+1}  (\alpha ) := \mathbf{1}_{1 \leqslant s  \leqslant k }(\alpha)  \sum_{ j \in  \mathcal{J}^k(\alpha)  }  c_k (T_{j}^{-1} (\alpha ) ) 
+\mathbf{1}_{2 \leqslant s  \leqslant k+1 }(\alpha) \mathbf{1}_{\alpha_{s}=0} (\alpha) c_k (\tilde{T}_{s}^{-1} (\alpha ) ) .
\end{equation*}
It is therefore easy to conclude. \par
%
%
%
%
%
%
%
%
%
%
%
%
%
\section{Proof of Corollary \ref{CorollaireInverse}} 
\label{Sec:PreuveCoroInverse}
It follows from the proof of Theorem \ref{start4} that both solid flows satisfy for some constant $\tilde{L}$ depending on $\Omega$, $\mathcal{S}_0$, $\rho_{\mathcal{S}_0}$ and $R$ only: 
\begin{equation} \label{LEstimee}
\left\| \frac{d^{k}}{d t^{k}} \Phi^{\mathcal S}_{i} \right\|_{L^{\infty}(-T_{*},T_{*};SE(3))} \leq \tilde{L}^{k} \, k!.
\end{equation}
This involves that the flows $\Phi_{i}$ defined on $[-T_{*},T_{*}]$ can be analytically extended in %
\begin{equation*}
{\mathcal O}:=\Big\{ z \in \C \ / \ d(z, [-T_{*},T_{*}]) < \frac{1}{2\tilde{L}}\Big\}.
\end{equation*}
Let ${\mathcal U}$ an open Jordan domain with analytic boundary (the interior of an ellipse for instance) such that 
\begin{equation*}
[-T_{*},T_{*}] \subset {\mathcal U} \subset \overline{\mathcal U} \subset {\mathcal O}.
\end{equation*}
Now the domain ${\mathcal U} \setminus [-\tau,\tau]$ is a doubly connected domain in the complex plane. It follows that there is a conformal mapping $\Psi$ from ${\mathcal U} \setminus [-\tau,\tau]$ to some annulus:
\begin{equation*}
{\mathcal A} :=  \Big\{ z \in \C \ / \ 1 < |z| < \rho \Big\},
\end{equation*}
with $[-\tau,\tau]$ sent to $S(0,1)$ and $\partial {\mathcal U}$ sent to $S(0,\rho)$. A way to realize this is to use a conformal map of $\widehat{\C} \setminus \overline{B}(0,1)$ to $\widehat{\C} \setminus [-\tau,\tau]$ (here $\widehat{\C}$ stands for the Riemann sphere) for instance the Joukowski map
\begin{equation*}
J: z \mapsto \frac{\tau}{2} \left( z + \frac{1}{z} \right).
\end{equation*}
Now $J^{-1}({\mathcal U} \setminus [-\tau,\tau])$ is a doubly connected domain in the complex plane with analytic boundaries boundaries. Such a domain can be made conformally equivalent to ${\mathcal A}$ by a mapping which is smooth up to the boundary, see for instance Ahlfors \cite[Section 6.5.1]{Ahlfors}. \par
Now, clearly, there exists $r \in (1,\rho)$ such that
\begin{equation*}
\Psi([-T_{*},T_{*}]) \subset B(0,r).
\end{equation*}
Define
\begin{equation*}
\varphi(z) = \Phi^{\mathcal S}_{1}(z) - \Phi^{\mathcal S}_{2}(z) \text{ on } \overline{\mathcal U}.
\end{equation*}
Apply Hadamard's three circle theorem to $\tilde{\varphi}:=\varphi\circ \Psi^{-1}$ in ${\mathcal A}$ (note that $\tilde{\varphi}$ is continuous up to the boundary since $J$ is). For $\delta = \log(\rho/r)/\log(\rho)$, we have
\begin{equation*}
\| \tilde{\varphi} \|_{L^{\infty}(S(0,r))} \leq \| \tilde{\varphi} \|_{L^{\infty}(S(0,1))}^{\delta} \| \tilde{\varphi} \|_{L^{\infty}(S(0,\rho))}^{1-\delta}.
\end{equation*}
Returning to ${\mathcal U}$, we deduce
\begin{equation*}
\| \varphi \|_{L^{\infty}(-T_{*},T_{*})} \leq \| \varphi \|_{L^{\infty}(\Psi^{-1}(B(0,r)))} \leq
\| \varphi \|_{L^{\infty}(-\tau,\tau)}^{\delta} \| {\varphi} \|_{L^{\infty}(\overline{\mathcal U})}^{1-\delta}.
\end{equation*}
(Here we could have put a stronger norm on the left hand side). Now \eqref{LEstimee} allows us to bound the factor $\| {\varphi} \|_{\overline{\mathcal U}}^{1-\delta}$ in terms of $\Omega$, $\mathcal{S}_0$, $\rho_{\mathcal{S}_0}$ and $R$, which concludes the proof of Corollary \ref{CorollaireInverse}.
%
%
%
%
%
%
%
%
%
%
%
%
%
%
%
%
\section{Appendix: Cauchy problem}
In this appendix, we prove Theorem \ref{start3} and Proposition \ref{Prostart3}.

\subsection{Preliminaries and notations}
In what follows, we prove existence and uniqueness for positive times, that is, on 
$[0,T]$. This is not a restriction since the system is clearly reversible. \par
Note that in this appendix, we will use the letter $\eta$ for the fluid flow and $\tau$ for the solid flow. \par
We suppose ${\mathcal S}_{0}$ and $\rho_{{\mathcal S}_{0}}$ fixed. By a geometric constant, we mean below a constant depending on $\Omega$, ${\mathcal S}_{0}$ and $\rho_{{\mathcal S}_{0}}$, $\lambda$ and $r$ only. The various constants $C>0$ that will appear, and which can grow from line to line, will be geometrical. \par
To $(\ell,r) \in C^{0}([0,T];\R^{6})$ we can associate $(x^{\ell,r}_{B},Q^{\ell,r}) \in C^{1}([0,T]; \R^{3} \times \R^{3 \times 3})$ by
\begin{equation} \label{Eq:xbq}
x^{\ell,r}_{B}(t)=x_{0} + \int_{0}^{t} \ell, \ \
\frac{d}{dt} Q^{\ell,r}(t)=r(t)\wedge Q^{\ell,r}(t) \ 
\text{ and } \ Q^{\ell,r}(0)=\Id,
\end{equation}
and the velocity
\begin{equation} \label{Defvsolide}
v^{\ell,r}(t,x) := \ell(t) + r(t) \wedge (x-x_{B}^{\ell,r}(t)).
\end{equation}
We also deduce the rigid displacement and the position of the solid, let us say $\tau^{\ell,r}(t)$ and ${\mathcal S}^{\ell,r}(t)$ defined by
\begin{equation} \label{Eq:St}
\tau^{\ell,r}(t): x \mapsto Q^{\ell,r}(t)[x-x_{0}] + x^{\ell,r}_{B}(t) \in SE(3), \ \text{ and } \  {\mathcal S}^{\ell,r}(t)=\tau^{\ell,r}(t){\mathcal S}_{0}.
\end{equation}
Then we fix the fluid domain as ${\mathcal F}^{\ell,r}(t):=\Omega \setminus {\mathcal S}^{\ell,r}(t)$. We may omit the dependence on $(\ell,r)$ when there is no ambiguity on the various objects defined above. \par
\ \par
We will use the following lemmas which are elementary consequences of Fa\`a di Bruno's formula and the fact that H\"older spaces are algebras, see \cite[Lemmas A.2 \& 4]{InitialsBB} in the case of Sobolev spaces.
\begin{Lemma} \label{BBLemmaA.2}
Let $k$ in $\N^*$ and $\alpha \in (0,1)$, and let $\omega$, $\omega'$ be smooth bounded domains. Let $F \in C^{k,\alpha}({\omega}')$ and $G \in C^{k,\alpha}({\omega})$ with $G({\omega}) \subset {\omega}'$. Then $F \circ G \in C^{k,\alpha}(\omega)$ with, for some constant $C$ depending only on $\omega$,  $\omega'$ and $k$:
\begin{equation} \label{EstBBLemmaA.2}
\| F \circ G \|_{C^{k,\alpha}(\omega)} \leq C \| F \|_{C^{k,\alpha}(\omega')} \Big( \|G\|_{C^{k,\alpha}(\omega)}^{k} +1 \Big).
\end{equation}
\end{Lemma}
\begin{Lemma} \label{BBLemma4}
Let $\omega$ a smooth bounded domain, $F \in C^{k,\alpha}({\omega})$ and $G \in \mbox{Diff}(\overline{\omega}) \cap C^{k,\alpha}({\omega})$. Then for some constant $C$ depending only on $\omega$, $k \in \N^* $, $\alpha \in (0,1)$ and $\| G\|_{C^{k,\alpha}({\omega})}$, one has
\begin{equation} \label{EstBBLemma4}
\| \partial_{i} (F \circ G^{-1}) \circ G - \partial_{i} F \|_{C^{k-1,\alpha}({\omega})}
\leq C \| G - \Id \|_{C^{k,\alpha}({\omega})} \|F \|_{C^{k,\alpha}({\omega})}.
\end{equation}
\end{Lemma}
%
%
%
%
%
%
\subsection{With a prescribed solid movement}
\label{Subsec:PrescribedMovement}
In this paragraph we prove the following results, which concern the Euler system with a prescribed solid movement of ${\mathcal S}(t)$ inside $\Omega$. The first result gives existence and uniqueness of a solution for small times. The second one estimates the dependance of the solution with respect to the prescribed movement (in Lagrangian coordinates). 
\begin{Proposition}
\label{SolideImpose}
Let $\lambda$ in $\N$, $r \in (0,1)$, $T_{1}>0$ and a regular closed connected subset $\mathcal{S}_0 \subset \Omega$. There exists a constant $C_* = C_*  (\Omega ,\mathcal{S}_0) > 0$ such that the following holds.
Consider $(\ell,r) \in C^{0}([0,T_{1}]; \R^{6})$ such that
\begin{equation} \label{LoinDuBord}
\text{ for any } t \in [0,T_{1}], \ \  d(\tau^{\ell,r}(t)[{\mathcal S}_{0}], \partial \Omega) >0.
\end{equation}
Consider $u_0$ in ${C}^{\lambda+1,r}( \mathcal{F}_0)$ satisfying
\begin{equation} \label{CondComp}
\div(u_{0})= 0 \text{ in } {\mathcal F}_{0}, \ \
u_{0}.n= 0 \text{ on } \partial \Omega \ \text{ and } \ 
u_{0}(x).n(x)=[\ell(0) + r(0) \wedge (x-x_{0})].n(x) \text{ on } \partial {\mathcal S}_{0}.
\end{equation}
Then for
\begin{equation} \label{TE}
T = \min \left(T_{1}, \frac{C_*}{\| u_0 \|_{C^{\lambda+1,r}({\mathcal F}_{0})} + \|(\ell,r)\|_{C^{0}([0,T] ;\R^{6})} } \right),
\end{equation}
the problem  (\ref{Euler1a2})-(\ref{Euler2a2})-(\ref{Euler3a2})-(\ref{Euler3b}) (with ${\mathcal S}(t):=\tau^{\ell,r}({\mathcal S}_{0})$ and ${\mathcal F}(t):= \Omega \setminus {\mathcal S}(t)$) admits a unique solution $u$ in $L^{\infty} (0,T ; C^{\lambda+1,r}  (\mathcal{F} (t) )   )$, which is moreover in $C_{w} ( [0,T] ; C^{\lambda+1,r}  (\mathcal{F} (t) )   )$, for $ r'  \in (0,r)$ and the same holds for $\partial_t u$ instead of $u$ with $\lambda$ instead of $ \lambda+1$. \par
\end{Proposition}
\begin{Proposition} \label{DeuxImpositions}
There exists $K>0$ such that, for $(\ell_{1},r_{1})$,  $(\ell_{2},r_{2})$ in $C^{0}([0,T_{1}];\R^{6})$ satisfying \eqref{LoinDuBord}, 
\begin{equation} \label{lrAuDebut}
\ell_{1}(0)= \ell_{2} (0) , \ \ r_{1}(0)=r_{2}(0),
\end{equation}
and 
\begin{equation} \label{TailleDesDeux}
\|(\ell_{1},r_{1})\|_{C^{0}([0,T] ;\R^{6})}, \ \|(\ell_{2},r_{2})\|_{C^{0}([0,T] ;\R^{6})}  \leq M,
\end{equation}
for any $u_{0} \in {C}^{\lambda+1,r}( \mathcal{F}_0)$ satisfying \eqref{CondComp} (with both $(\ell_{1},r_{1})$ and $(\ell_{2},r_{2})$), the following holds. Call $u_{1}$, $u_{2}$ the corresponding solutions given by Proposition \ref{SolideImpose} on $[0,T]$ with
\begin{equation} \label{TempsCommun}
T = \min \left(T_{1}, \frac{C_*}{\| u_0 \|_{C^{\lambda+1,r}({\mathcal F}_{0})} + M } \right).
\end{equation}
and $\eta^{1}$ and $\eta^{2}$ the corresponding flows. One has
\begin{multline} \label{Dependance}
\| \eta_{1} - \eta_{2}\|_{L^{\infty}([0,T];C^{\lambda+1,r}({\mathcal F}_{0}))} + 
T \| u_{1}(t,\eta_{1}(t,x)) - u_{2}(t,\eta_{2}(t,x)) \|_{L^{\infty}([0,T];C^{\lambda+1,r}({\mathcal F}_{0}))} \\ \leq KT \| (\ell_{1},r_{1}) - (\ell_{2},r_{2}) \|_{C^{0}([0,T];\R^{6})}.
\end{multline}
\end{Proposition}
\subsubsection{Proof of Proposition \ref{SolideImpose}}
{\bf 1.} Let $(\ell,r)$ be fixed so that \eqref{LoinDuBord} holds. We deduce $\tau(t)$, ${\mathcal S}(t)$ and ${\mathcal F}(t)$ as previously.
We introduce $(\Gamma_{i})_{i=1 \dots g}$ a family of of smooth oriented loops in ${\mathcal F}_{0}$ giving a homology basis of it. \par
We let
\begin{align*}
{\mathcal C}:=\Big\{ &\eta \in C^{0}([0,T]; C^{\lambda+1,r}({\mathcal F}_{0};\R^{3})) \ \Big/ \\
& \ \ \text{i. } \forall t \in [0,T], \ \eta(t,\cdot) \mbox{ is a volume-preserving diffeomorphism from } \overline{{\mathcal F}_{0}} \text{ to } \overline{{\mathcal F}(t)}, \\
&\ \ \  \text{ sending } \partial \Omega \text{ to } \partial \Omega \text{ and } \partial {\mathcal S}_{0} \text{ to } \partial {\mathcal S}(t), \\
&\ \ \text{ii. } \| \eta - \Id \|_{C^{0}([0,T]; C^{\lambda+1,r}({\mathcal F}_{0}))} \leq \frac{1}{2}
 \Big\}.
\end{align*}
Note that, due to the fact that ${\mathcal F}(t)$ has the same volume as ${\mathcal F}_{0}$, the property $\text{i.}$ defining ${\mathcal C}$ is equivalent to
\begin{equation} \label{CondC}
\forall t \in [0,T], \ \det[{\mbox{Jac}}(\eta(t,\cdot))]=1 \text{ on } {\Omega} \text{ and } \eta(t,\partial \Omega) = \partial \Omega, \ \eta(t,\partial {\mathcal S}_{0}) = \partial {\mathcal S}(t). \end{equation}
Hence it is not difficult to check that ${\mathcal C}$ is closed for the $C^{0}([0,T]; C^{\lambda+1,r}({\mathcal F}_{0};\R^{3}))$ distance. \par
\ \par
Now we define ${\mathcal T}= {\mathcal T}^{\ell,r}:{\mathcal C} \rightarrow {\mathcal C}$ as follows. Given $\eta \in {\mathcal C}$, we define $\omega: [0,T] \times {\mathcal F}(t) \rightarrow \R^{3}$ by
\begin{equation} \label{DefOmega}
\omega(t,x) = (\nabla \eta)(t,\eta^{-1}(t,x)) \cdot \omega_{0}(\eta^{-1}(t,x)),
\end{equation}
where $\omega_{0}:= \curl u_{0}$ in ${\mathcal F}_{0}$. (Note that when $\eta$ is the flow of a vector field $w$, one has
\begin{equation} \label{EqVorticite}
\partial_{t} \omega + (w\cdot\nabla) \omega = (\omega\cdot \nabla)w.) 
\end{equation}
Next we define $u: [0,T] \times {\mathcal F}(t) \rightarrow \R^{3}$ by the following system
\begin{equation} \label{Eq:vtransporte}
\left\{ \begin{array}{l}
\curl {u} = {\omega} \text{ in } [0,T] \times {\mathcal F}(t), \\
\div {u} = 0 \text{ in } [0,T] \times {\mathcal F}(t), \\
{u}. n =0 \text{ on } [0,T] \times \partial \Omega, \\
{u}(x). n = v(t).n \text{ on } [0,T] \times \partial {\mathcal S(t)}, \\
\oint_{\eta(\Gamma_{i})} {u} .d \tau = \oint_{\Gamma_{i}} u_{0} .d \tau \text{ for all } i =1 \dots g,
\end{array} \right.
\end{equation}
with $v$ defined in \eqref{Defvsolide}. \par
Then we define the flow $\tilde{\eta}(t,x)$ associated to $u$, which for each $t$ sends ${\mathcal F}_{0}$ to ${\mathcal F}(t)$. (In order to deal with the flow of a vector field on a fixed domain, for instance, extend $u$ on $\R^{3}$, define the flow, and then restrict it to ${\mathcal F}_{0}$.) \par
Finally, we let 
\begin{equation} \label{DefT}
{\mathcal T}(\eta):=\tilde{\eta}.
\end{equation}
\ \\
{\bf 2.} Let us prove that ${\mathcal T}$ has a unique fixed point by Banach-Picard's theorem. First, let us prove that ${\mathcal T}({\mathcal C}) \subset {\mathcal C}$. That ${\mathcal T}(\eta)$ satisfies the property $\text{i.}$ defining ${\mathcal C}$ is a direct consequence of \eqref{CondC} and \eqref{Eq:vtransporte}. Let us prove that for $T \leq T_{1}$ small enough, the property $\text{ii.}$ holds. Given $\eta \in {\mathcal C}$, using Lemma~\ref{AutreRegdivcurl} (and Remark \ref{RemDivCurl}) and \eqref{Eq:vtransporte}, we see that for some constant $C>0$ depending on the geometry only:
\begin{equation} \label{EstU}
\| u \|_{L^{\infty}(0,T; C^{\lambda+1,r}({\mathcal F} (t)))} \leq C \left( \| u_{0}\|_{C^{\lambda+1,r}({\mathcal F}_{0})} + \| v \|_{C^{0}([0,T] \times \overline{\Omega})}  \right).
\end{equation}
Also, for some geometric constant, one has
\begin{equation*}
\| v \|_{C^{0}([0,T] \times \overline{\Omega})} \leq C \| (\ell,r) \|_{C^{0}([0,T];\R^{6})}.
\end{equation*}
Now, since $\tilde{\eta}(t,\cdot) - \Id = \int_{0}^{t} u \circ \tilde{\eta}$,  using Lemma \ref{BBLemmaA.2}, we see that 
\begin{equation*}
\|\tilde{\eta} - \Id \|_{ C^{0}([0,t]; C^{\lambda+1,r}({\mathcal F}_{0})) } \leq Ct
\left( 1 + \| \tilde{\eta} -\Id \|_{C^{0}([0,t]; C^{\lambda+1,r}({\mathcal F}_{0})) }^{\lambda+1} \right) \| u \|_{L^{\infty}(0,t; C^{\lambda+1,r}({\mathcal F}(s)))}.
\end{equation*}
Hence if $T>0$ is such that $CT \left(1 + \left(\frac{1}{2}\right)^{\lambda+1} \right) \| u \|_{L^{\infty}(0,T; C^{\lambda+1,r}({\mathcal F}(t)))} \leq 1/2$, by a connectedness in time argument, wee see that $\tilde{\eta}$ satisfies the property $\text{ii}$. Hence there is a constant $C_{*}>0$ such that for
\begin{equation*}
T \leq \frac{C_*}{\| u_0 \|_{C^{\lambda+1,r}({\mathcal F}_{0})} + \|(\ell,r)\|_{C^{0}([0,T] ;\R^{6})}},
\end{equation*}
${\mathcal T}(\eta)$ satisfies property $\text{ii.}$, so that ${\mathcal T}(\eta) \in {\mathcal C}$. \par
\ \\
{\bf 3.} Let us now prove that ${\mathcal T}$ is contractive for small $T>0$. Given $\eta_{1}, \eta_{2} \in {\mathcal C}$, we let $\omega_{1}$, $\omega_{2}$, $u_{1}$, $u_{2}$, etc. be the various objects associated to $\eta_{1}$ and $\eta_{2}$ in the construction of ${\mathcal T}$. We also define
\begin{equation} \label{DefMathcalU}
{\mathcal U}_{i}(t,x) = u_{i}(t,\tilde{\eta}_{i}(t,x)) \ \ \text{ on } \ {\mathcal F}_{0} \ \text{ for } t\in [0,T] \text{ and } i=1,2.
\end{equation}
We have
\begin{equation*}
\tilde{\eta}_{1}(t,x) - \tilde{\eta}_{2}(t,x) = \int_{0}^{t} \left[{\mathcal U}_{1}(s,x) - {\mathcal U}_{2}(s,x)\right] \, ds.
\end{equation*}
Now let us prove that for some geometric constant $C>0$, we have
\begin{equation*}
\| {\mathcal U}_{1} - {\mathcal U}_{2} \|_{L^{\infty}([0,T];C^{\lambda+1,r}({\mathcal F}_{0}))} \leq C 
\left( \|u_{0}\|_{C^{\lambda+1,r}({\mathcal F}_{0})} + \| (\ell,r) \|_{C^{0}([0,T];\R^{6})}\right)
\| \eta_{1} -\eta_{2}\|_{L^{\infty}([0,T];C^{\lambda+1,r}({\mathcal F}_{0}))}.
\end{equation*}
We follow \cite{InitialsBB}. For $t \in [0,T]$, we have, omiting the dependence on $t$ to simplify the notations, using Lemma~\ref{AutreRegdivcurl} and Lemma~\ref{BBLemmaA.2}:
\begin{eqnarray}
\nonumber
\| {\mathcal U}_{1} - {\mathcal U}_{2} \|_{C^{\lambda+1,r}({\mathcal F}_{0})}
&=& 
\| u_{1} \circ \eta_{1} - u_{2} \circ \eta_{2} \|_{C^{\lambda+1,r}({\mathcal F}_{0})} \\
\nonumber
&\leq&
C \| u_{1} \circ \eta_{1} \circ \eta_{2}^{-1} - u_{2} \|_{C^{\lambda+1,r}({\mathcal F}(t))} \\
\nonumber
&\leq &
C 
\Big(
\| \curl(u_{1} \circ \eta_{1} \circ \eta_{2}^{-1}) - \curl (u_{2}) \|_{C^{\lambda,r}({\mathcal F}(t))} \\
\nonumber
&\ &\ \ 
+ \ \| \div(u_{1} \circ \eta_{1} \circ \eta_{2}^{-1}) - \div (u_{2}) \|_{C^{\lambda,r}({\mathcal F}(t))} \\ 
\nonumber
&\ &\ \ 
+\ \sum_{i=1}^{g} \left| \oint_{\eta_{2}(\Gamma_{i})} (u_{1} \circ \eta_{1} \circ \eta_{2}^{-1} - u_{2}) . d \tau  \right|  \\
\label{GEDC}
&\ &\ \ 
+\ \| (u_{1} \circ \eta_{1} \circ \eta_{2}^{-1}).n - u_{2}.n \|_{C^{\lambda+1,r}( \partial {\mathcal F}(t))} 
\Big).
\end{eqnarray}
Concerning the first term in the right-hand side, using \eqref{DefOmega}, \eqref{Eq:vtransporte} and Lemmas \ref{BBLemmaA.2} and \ref{BBLemma4}, we see that 
\begin{align*}
\| \curl(u_{1} \circ &\eta_{1} \circ \eta_{2}^{-1}) - \curl (u_{2}) \|_{C^{\lambda,r}({\mathcal F}(t))} \\
&\leq \| \curl(u_{1} \circ \eta_{1} \circ \eta_{2}^{-1})  -  (\curl u_{1}) \circ \eta_{1} \circ \eta_{2}^{-1} \|_{C^{\lambda,r}({\mathcal F}(t))} + \|  (\curl u_{1}) \circ \eta_{1} \circ \eta_{2}^{-1} - (\curl u_{2}) \|_{C^{\lambda,r}({\mathcal F}(t))} \\
&\leq C\| \eta_{1} - \eta_{2} \|_{C^{\lambda+1,r}({\mathcal F}_{0})} \| u_{1}(t) \|_{ C^{\lambda+1,r}({\mathcal F}(t))} +
C \|  (\curl u_{1}) \circ \eta_{1} - (\curl u_{2}) \circ \eta_{2} \|_{C^{\lambda,r}({\mathcal F}_{0})} \\
& \leq C \left( \|\omega_{0} \|_{ C^{\lambda,r}( {\mathcal F}_{0}) )} + \| u_{1}(t) \|_{ C^{\lambda+1,r}({\mathcal F}(t))} \right)
\| \eta_{1} - \eta_{2} \|_{C^{\lambda+1,r}({\mathcal F}_{0})}.
\end{align*}
The second term in \eqref{GEDC} is treated likewise (this is even slightly simpler since $\div u_{1}=\div u_{2}=0$); hence we can bound it by
\begin{equation} \label{Est136}
C \| u_{1}(t) \|_{C^{\lambda+1,r}({\mathcal F}(t))} \| \eta_{1} - \eta_{2} \|_{C^{\lambda+1,r}({\mathcal F}_{0})}.
\end{equation}
Using \eqref{Eq:vtransporte}, we see that the third term in \eqref{GEDC} can be bounded by \eqref{Est136} as well. Using \eqref{Eq:vtransporte}, we see that the last term can be bounded by 
$$C \| v\|_{C^{0}([0,T] \times \overline{\Omega})} \| \eta_{1} - \eta_{2} \|_{C^{\lambda+1,r}({\mathcal F}_{0})},$$
since $\| \cdot \|_{C^{0}(\overline{\Omega})}$ and  $\| \cdot \|_{C^{\lambda+1,r}({\Omega})}$ are equivalent as long as $v$ is concerned. 
Hence, using \eqref{EstU}, we get
\begin{equation} \label{Upareta}
\| {\mathcal U}_{1} - {\mathcal U}_{2} \|_{C^{\lambda+1,r}({\mathcal F}_{0})} \leq C \left( \| u_{0}\|_{C^{\lambda+1,r}({\mathcal F}_{0})} + \| (\ell,r) \|_{C^{0}([0,T]; \R^{6})} \right)
\| \eta_{1} - \eta_{2} \|_{C^{\lambda+1,r}({\mathcal F}_{0})},
\end{equation}
and it follows that for some $T$ of the form \eqref{TE}, the operator ${\mathcal T}$ is contractive. Now a fixed point in ${\mathcal C}$ gives a solution of the Euler equation and reciprocally. This comes from
\eqref{EqVorticite} which gives for a fixed point
\begin{equation*}
\curl \left( \frac{\partial u}{\partial t} + (u\cdot \nabla) u \right)= 0 \text{ in } (0,T) \times {\mathcal F}(t),
\end{equation*}
and the fact that:
\begin{equation*}
\frac{d}{dt} \int_{\eta(\Gamma_{i})} u.d \tau = \int_{\eta(\Gamma_{i})} \left(\frac{\partial u}{\partial t} + (u\cdot \nabla) u \right).d \tau =0 \text{ in } (0,T).
\end{equation*}
This proves the claim. Note that $\eta \in C^{0}([0,T];C^{\lambda+1,r} ({\mathcal F}_0 )  )$ and $u\circ\eta \in C^{0}([0,T];C^{\lambda+1,r}({\mathcal F}_{0}))$ involve $u \in L^{\infty}(0,T;C^{\lambda+1,r}({\mathcal F}(t)))$; the weak continuity then follows from the continuity into a weaker space, for instance $u \in C^{0}([0,T];C^{\lambda,r}({\mathcal F}(t)))$. \par
\subsubsection{Proof of Proposition \ref{DeuxImpositions}}
Given $(\ell_{1},r_{1})$ and $(\ell_{2},r_{2})$ satisfying \eqref{lrAuDebut} and \eqref{TailleDesDeux}, we introduce the respective fixed points $\eta_{1}$ and $\eta_{2}$ in $C^{0}([0,T];C^{\lambda+1,r}({\mathcal F}_{0}))$ of the operators ${\mathcal T}^{\ell_{1},r_{1}}$ and ${\mathcal T}^{\ell_{2},r_{2}}$ defined above (as well as the corresponding objects ${\mathcal S}_{i}(t)$, ${\mathcal F}_{i}(t)$, $u_{i}$, ${\mathcal U}_{i}$, etc.), defined on $[0,T]$ with $T$ introduced in \eqref{TempsCommun}. \par \noindent
We proceed as previously (again, we omit to write the dependence on $t$ to simplify the notations):
\begin{eqnarray*}
\| u_{1} \circ \eta_{1} - u_{2} \circ \eta_{2} \|_{C^{\lambda+1,r}({\mathcal F}_{0})}
&\leq&
C \| u_{1} \circ \eta_{1} \circ \eta_{2}^{-1} - u_{2} \|_{C^{\lambda+1,r}({\mathcal F}_{2}(t))} \\
&\leq &
C 
\Big(
\| \curl(u_{1} \circ \eta_{1} \circ \eta_{2}^{-1}) - \curl (u_{2}) \|_{C^{\lambda,r}({\mathcal F}_{2}(t))} \\
&\ &
+ \| \div(u_{1} \circ \eta_{1} \circ \eta_{2}^{-1}) - \div (u_{2}) \|_{C^{\lambda,r}({\mathcal F}_{2}(t))} \\ 
&\ &
+ \sum_{i=1}^{g} \left| \oint_{\Gamma_{i}} (u_{1} \circ \eta_{1} \circ \eta_{2}^{-1}- u_{2}) . d \tau \right|
+ \| (u_{1} \circ \eta_{1} \circ \eta_{2}^{-1}).n - u_{2}.n \|_{C^{\lambda+1,r}( \partial {\mathcal F}_{2}(t))} 
\Big).
\end{eqnarray*}
Using Lemma \ref{BBLemma4}, we deduce
\begin{gather*}
\| \curl(u_{1} \circ \eta_{1} \circ \eta_{2}^{-1}) - \curl (u_{2}) \|_{C^{\lambda,r}({\mathcal F}_{2}(t))}  \leq
C \left( \|u_{1}\|_{C^{\lambda+1,r}({\mathcal F}_{1}(t))} + \| \omega_{0} \|_{C^{\lambda,r}({\mathcal F}_{0})} \right) \| \eta_{1} - \eta_{2} \|_{C^{\lambda+1,r}({\mathcal F}_{0})}, \\
\| \div(u_{1} \circ \eta_{1} \circ \eta_{2}^{-1}) - \div (u_{2}) \|_{C^{\lambda,r}({\mathcal F}_{2}(t))}  \leq
C \|u_{1}\|_{C^{\lambda+1,r}({\mathcal F}_{1}(t))} \| \eta_{1} - \eta_{2} \|_{C^{\lambda+1,r}({\mathcal F}_{0})}.
\end{gather*}
And it is not difficult to see that
\begin{multline*}
\sum_{i=1}^{g} \left| \oint_{\Gamma_{i}} (u_{1} \circ \eta_{1} \circ \eta_{2}^{-1}- u_{2}) . d \tau \right| + 
\| (u_{1} \circ \eta_{1} \circ \eta_{2}^{-1}).n - u_{2}.n \|_{C^{\lambda+1,r}( \partial {\mathcal F}_{2}(t))} \\
\leq C \left( \|u_{1}\|_{C^{\lambda+1,r}({\mathcal F}_{1}(t))} \| \eta_{1} - \eta_{2} \|_{C^{\lambda+1,r}({\mathcal F}_{0})} + \| (\ell_{1},r_{1}) - (\ell_{2},r_{2}) \|_{\R^{6}} \right) .
\end{multline*}
Hence we have
\begin{equation*}
\| u_{1} \circ \eta_{1} - u_{2} \circ \eta_{2} \|_{C^{\lambda+1,r}({\mathcal F}_{0})}
\leq C \left( \|u_{1}\|_{C^{\lambda+1,r}({\mathcal F}_{1}(t))} + \| \omega_{0} \|_{C^{\lambda,r}({\mathcal F}_{0})} \right) \| \eta_{1} - \eta_{2} \|_{C^{\lambda+1,r}({\mathcal F}_{0})} + C  \| (\ell_{1},r_{1}) - (\ell_{2},r_{2}) \|_{\R^{6}} .
\end{equation*}
Since
\begin{equation*}
\eta_{1}(t) - \eta_{2}(t) = \int_{0}^{t} \left[u_{1} \circ \eta_{1} - u_{2} \circ \eta_{2}\right],
\end{equation*}
the conclusion easily follows from Gronwall's lemma. 
\begin{Remark} \label{AAssocA}
The operator ${\mathcal T}$ defined above can be defined for any initial datum $u_{0}$, with $u_{0}$ divergence-free, tangential to $\partial \Omega$, and satisfying the compatibility condition:
\begin{equation*}
u_{0}.n = (\ell(0) + r(0) \wedge (x-x_{0})) \text{ on } \partial {\mathcal S}_{0}.
\end{equation*}
Equivalently, we could associate an operator ${\mathcal A}$ to any initial data $(\omega_{0},\lambda_{0}^{1}, \dots, \lambda_{0}^{g})$ in $C^{\lambda,r}({\mathcal F}_{0}) \times \R^{g}$, and reconstruct $u_{0}$ satisfying the compatibility conditions by
\begin{equation} \label{OmegaAssocU}
\left\{ \begin{array}{l}
\curl {u}_{0} = {\omega}_{0} \text{ in } {\mathcal F}_{0}, \\
\div {u}_{0} = 0 \text{ in } {\mathcal F}_{0}, \\
{u}_{0}. n =0 \text{ on }  \partial \Omega, \\
{u}_{0}(x). n = (\ell(0) + r(0) \wedge (x-x_{0})).n \text{ on } \partial {\mathcal S}_{0}, \\
\oint_{\Gamma_{i}} {u}_{0} .d \tau = \lambda^{i}_{0} \text{ for all } i =1 \dots g.
\end{array} \right.
\end{equation}
Doing so, Proposition \ref{DeuxImpositions} extends to $(\ell_{1},r_{1})$ and $(\ell_{2},r_{2})$ which do no longer satisfy \eqref{lrAuDebut}. In that case, \eqref{Dependance} compares solutions with initial velocity fields $u^{1}_{0}$ and $u^{2}_{0}$ given by \eqref{OmegaAssocU} with $(\ell_{1}(0),r_{1}(0))$ and $(\ell_{2}(0),r_{2}(0))$, respectively.
\end{Remark}
%
%

%
%
%
%
%
%
\subsection{With a moving solid}
\subsubsection{Proof of Theorem \ref{start3}}
Here we prove Theorem \ref{start3}. Again we rely on a Banach-Picard fixed point strategy. \par
\ \\
{\bf 1.} We introduce
\begin{align*}
{\mathcal D}:= \Big\{ (\ell,r) &\in C^{0}([0,T];\R^{6}) \ \Big/ \\
&\ \ \text{i.} \  \tau^{\ell,r} \text{ satisfies }  d \left( \tau^{\ell,r}(t)({\mathcal S}_{0}), \partial \Omega \right)  \geq \frac{\underline{d}}{3}, \\
&\ \ \text{ii.} \ \| (\ell,r) - (\ell_{0},r_{0}) \|_{C^{0}([0,T];\R^{6})} \leq \| u_{0} \|_{C^{\lambda+1,r}} + \| (\ell_{0},r_{0}) \|_{\R^{6}}
\Big\}.
\end{align*}
\begin{Remark} \label{UneConstanteKEnplus}
As we follow from the proof, we could replace $\text{ii.}$ by

\begin{equation*}
\mathrm{\text{\rm ii.} \ \  \| (\ell,r) - (\ell_{0},r_{0}) \|_{C^{0}([0,T];\R^{6})} \leq C, }
\end{equation*}
for any positive constant $C>0$.
\end{Remark}
Now we construct an operator ${\mathcal A}$ on ${\mathcal D}$ in the following way. To $(\ell,r) \in {\mathcal D}$, we associate $Q(t)$, ${\mathcal S}(t)$ and ${\mathcal F}(t)$ defined from $(\ell,r)$. Next we associate the fixed point $\eta \in C^{0}([0,T];C^{\lambda+1,r}({\mathcal F}_{0}))$ of the operator ${\mathcal T}^{\ell,r}$ defined in Paragraph \ref{Subsec:PrescribedMovement} with $T$ of the form \eqref{TE}. Note that due to properties $\text{i.}$ and $\text{ii.}$ in the definition of ${\mathcal D}$ and Proposition \ref{SolideImpose}, there is a time
\begin{equation*}
T = \frac{C_*}{\| u_0 \|_{C^{\lambda+1,r}({\mathcal F}_{0})} + \|(\ell_{0},r_{0})\|_{\R^{6}}},
\end{equation*}
such that $\eta^{\ell,r}$ is defined on $[0,T_{*}]$, uniformly in $(\ell,r) \in {\mathcal D}$. Together with this flow $\eta$, we will consider the various functions $u$, ${\mathcal U}$, etc. defined on $[0,T]$. \par
Define ${\mathcal J}$, $(\Phi_{i})_{i=1\dots 6}$ and $\mu$ by
\begin{equation}\label{NCSylvester}
 \mathcal{J}(t) =Q(t) \mathcal{J}_0 Q^{*}(t) \ \text{ on } [0,T],
\end{equation}
\begin{equation} \label{NCDefPhii}
\left\{ \begin{array}{l}
-\Delta \Phi_i = 0 \quad   \text{for}  \ x\in \mathcal{F}(t), \\
\frac{\partial \Phi_i}{\partial n}=  0 \quad  \text{for}  \ x\in \partial \Omega, \\
\frac{\partial \Phi_i}{\partial n}=K_i \quad  \text{for}  \  x\in \partial \mathcal{S}(t), \\
\int_{\mathcal{F}(t)} \Phi_{i} \ dx =0,
\end{array} \right.
\end{equation}
where
\begin{equation} \label{NCt1.6}
K_i:=\left\{\begin{array}{ll} 
n_i & \text{if} \ i=1,2,3 ,\\ \relax
[(x- x_{B})\wedge n]_{i-3}& \text{if} \ i=4,5,6, 
\end{array}\right.
\end{equation}
and
\begin{equation} \label{NCEqMu}
\left\{ \begin{array}{l}
-\Delta \mu = \trace\{ \nabla u \cdot \nabla u \}  \quad   \text{for}  \ x\in \mathcal{F}(t), \\
\frac{\partial \mu}{\partial n} = - \nabla^{2} \rho(u,u) \quad   \text{for}  \ x\in \partial \Omega, \\
\frac{\partial \mu}{\partial n}= \nabla^{2} \rho \, \{ u- v , u- v \} - n \cdot \big(r\wedge \left(2u-v-\ell \right)\big),  \quad   \text{for}  \ x\in \partial \mathcal{S}(t), \\
\int_{\mathcal{F}(t)} \mu \ dx =0.
\end{array} \right.
\end{equation}
Introduce
\begin{equation} \label{NCEq:M}
\mathcal{M}(t) := \mathcal{M}_{1}(t) + \mathcal{M}_{2}(t)
:=\begin{bmatrix} m \Id_3 & 0 \\ 0 & \mathcal{J}\end{bmatrix}, \quad
+ \begin{bmatrix} \displaystyle\int_{\mathcal{F}(t)} \nabla \Phi_i \cdot \nabla \Phi_j \ dx \end{bmatrix}_{i,j \in \{1,\ldots,6\}},
\end{equation}
and then define ${\mathcal A}(\ell,r):=(\tilde{\ell},\tilde{r})$ as 
\begin{equation} \label{NCEM2}
 \begin{bmatrix} \tilde{\ell} \\[0.5cm] \tilde{r} \end{bmatrix} = 
\begin{bmatrix} \ell_{0} \\[0.5cm] r_{0} \end{bmatrix}
+ \int_{0}^{t} {\mathcal M}^{-1}(s) \left\{ \begin{bmatrix} 0 \\[0.5cm] \mathcal{J}(s)r(s) \wedge r(s) \end{bmatrix} + \begin{bmatrix}  \displaystyle\int_{ \mathcal{F}(t)} \nabla \mu \cdot \nabla \Phi_i \, dx   \end{bmatrix}_{i \in \{1,\ldots,6\}} \right\} \,ds.
\end{equation}
\ \\
{\bf 2.} We now show that for suitable $T$, the operator ${\mathcal A}$ maps ${\mathcal D}$ into itself. Then we will prove that it is contractive. \par
First, we see from \eqref{EstU} that we have the following bound on $u=u^{\ell,r}$, when $(\ell,r) \in {\mathcal D}$:
\begin{equation} \label{Estiu}
\| u \|_{L^{\infty}(0,T; C^{\lambda+1,r}({\mathcal F}(t)))} \leq C \left( 
\|u_{0} \|_{C^{\lambda+1,r}({\mathcal F}_{0})} + \| (\ell_{0},r_{0}) \|_{\R^{6}} \right).
\end{equation}
Also, the following bound is immediate from $\mathrm{ii}.$:
\begin{equation} \label{Estiv}
\| v\|_{C^{0}([0,T])} \leq C \left( \|u_{0} \|_{C^{\lambda+1,r}({\mathcal F}_{0})} + \| (\ell_{0},r_{0}) \|_{\R^{6}} \right).
\end{equation}
It follows easily using Lemma~\ref{AutreRegdivcurl} that for some geometric constant $C>0$:
\begin{equation} \label{EstiMu}
\| \nabla \mu\|_{L^{\infty}(0,T; C^{\lambda+1,r}({\mathcal F}(t)))} 
\leq C \left( \|u_{0} \|_{C^{\lambda+1,r}({\mathcal F}_{0})} + \| (\ell_{0},r_{0}) \|_{\R^{6}} \right)^{2}.
\end{equation}
Lemma~\ref{AutreRegdivcurl} also yields that
\begin{equation*}
\| \nabla \Phi_{i}\|_{L^{\infty}(0,T; C^{\lambda+1,r}({\mathcal F}(t)))} 
\leq C.
\end{equation*}
Next, the matrix $Q(t)$ is bounded since it is orthogonal, so ${\mathcal J}(t)$ is bounded as well by a geometric constant. Finally the matrix ${\mathcal M}_{2}(t)$ being always positive-definite, the matrix ${\mathcal M}^{-1}$ is also bounded by a geometric constant. We deduce that we have the following estimate uniformly on ${\mathcal D}$:
\begin{equation*}
\| (\tilde{\ell},\tilde{r}) -(\ell_{0},r_{0}) \|_{C^{0}([0,T];\R^{6})} \leq C T \left( \| u_{0} \|_{C^{\lambda+1,r}} + \| (\ell_{0},r_{0}) \|_{\R^{6}} \right)^{2}.
\end{equation*}
It follows easily that for some geometric constant $C_{*}>0$, one has ${\mathcal A}({\mathcal D}) \subset {\mathcal D}$ provided that
\begin{equation} \label{LeTemps}
T \leq \frac{C_*}{ \| u_0 \|_{C^{\lambda+1,r}({\mathcal F}_{0})} + \| (\ell_{0},r_{0}) \|_{\R^{6}} }.
\end{equation} \par
\ \\
{\bf 3.} Let us now prove that for $T$ of the form \eqref{LeTemps}, the operator ${\mathcal A}$ is contractive. Let $(\ell_{1},r_{1})$ and $(\ell_{2},r_{2})$ in ${\mathcal D}$. As previously we denote with an index $1$ or $2$ the objects associated to these couples above (except for $\Phi_{i}$ where $1$ and $2$ come as an exponent). \par
It is a straightforward consequence of Proposition \ref{DeuxImpositions} that for some constant $C>0$, one has
\begin{multline} \label{K1}
\| \eta_{1} - \eta_{2}\|_{L^{\infty}([0,T];C^{\lambda+1,r}({\mathcal F}_{0}))} + 
T\| u_{1}(t,\eta_{1}(t,x)) - u_{2}(t,\eta_{2}(t,x)) \|_{L^{\infty}([0,T];C^{\lambda+1,r}({\mathcal F}_{0}))} \\ \leq C T \| (\ell_{1},r_{1}) - (\ell_{2},r_{2}) \|_{C^{0}([0,T];\R^{6})}.
\end{multline}
Also, the following bound is immediate:
\begin{equation} \label{K2}
\| v_{1} - v_{2} \|_{C^{0}([0,T];\R^{6})} \leq C \| (\ell_{1},r_{1}) - (\ell_{2},r_{2}) \|_{C^{0}([0,T];\R^{6})}.
\end{equation}
Now proceeding as previously, using Lemma~\ref{AutreRegdivcurl}, we infer that for $t \in [0,T]$:
\begin{align} \nonumber
\| (\nabla \mu_{1}) \circ \eta_{1} - (\nabla \mu_{2}) \circ \eta_{2} \|_{C^{\lambda+1,r}({\mathcal F}_{0})} 
&\leq \| (\nabla \mu_{1}) \circ \eta_{1} \circ \eta_{2}^{-1} - (\nabla \mu_{2}) \|_{L^{\infty}(0,T;{\mathcal{F}_{2}(t)})} \\
\nonumber
&\leq C 
\Big(
\| \curl(\nabla \mu_{1} \circ \eta_{1} \circ \eta_{2}^{-1}) - \curl (\nabla \mu_{2}) \|_{C^{\lambda,r}({\mathcal F}_{2}(t))} \\
\nonumber
&\ 
+ \ \| \div(\nabla \mu_{1} \circ \eta_{1} \circ \eta_{2}^{-1}) - \div (\nabla \mu_{2}) \|_{C^{\lambda,r}({\mathcal F}_{2}(t))} \\ 
\nonumber
&\ 
+\ \sum_{i=1}^{g} \left| \oint_{\Gamma_{i}} (\nabla \mu_{1} \circ \eta_{1} \circ \eta_{2}^{-1}- \nabla \mu_{2}) . d \tau \right| \\
\label{GrosseEstimee}
&\ 
+\ \| (\nabla \mu_{1} \circ \eta_{1} \circ \eta_{2}^{-1}).n - \nabla \mu_{2}.n \|_{C^{\lambda+1,r}( \partial {\mathcal F}_{2}(t))} 
\Big).
\end{align}
For what concerns the second term in the right hand side, we have, using Lemma \ref{BBLemmaA.2}, Lemma \ref{BBLemma4} and \eqref{EstiMu},
\begin{align*}
\| \div(\nabla &\mu_{1} \circ \eta_{1} \circ \eta_{2}^{-1}) - \div \nabla \mu_{2} \|_{C^{\lambda,r}({\mathcal F}_{2}(t))} \\
& \leq 
\| \div(\nabla \mu_{1} \circ \eta_{1} \circ \eta_{2}^{-1}) - (\div \nabla \mu_{1}) \circ \eta_{1} \circ \eta_{2}^{-1}  \|_{C^{\lambda,r}({\mathcal F}_{2}(t))}
+ \| (\div \nabla \mu_{1}) \circ \eta_{1} \circ \eta_{2}^{-1} - \div \nabla \mu_{2} \|_{C^{\lambda,r}({\mathcal F}_{2}(t))} \\
& \leq 
C \left\{ \left[ \| u_0 \|_{C^{\lambda+1,r}({\mathcal F}_{0})} + \| (\ell_{0},r_{0}) \|_{\R^{6}}\right]^{2}
\| \eta_{1} - \eta_{2} \|_{C^{0}([0,T];C^{\lambda+1,r}({\mathcal F}_{0}))} 
+ \| \div(\nabla \mu_{1}) \circ \eta_{1} - (\div \nabla \mu_{2}) \circ \eta_{2} \|_{C^{\lambda,r}({\mathcal F}_{0})} \right\}.
\end{align*}
Now using \eqref{NCEqMu}, \eqref{Estiu} and $(\nabla u_{i}) \circ \eta_{i}=\nabla (u_{i} \circ \eta_{i}) \cdot (\nabla \eta_{i})^{-1}$, we see that
\begin{align} \nonumber
\| \div(\nabla \mu_{1}) \circ \eta_{1} &- \div (\nabla \mu_{2}) \circ \eta_{2} \|_{C^{\lambda,r}({\mathcal F}_{0})} \\
\nonumber
&= \| \tr \{ \nabla u_{1} \cdot \nabla u_{1} \} \circ \eta_{1} - \tr \{ \nabla u_{2} \cdot \nabla u_{2} \} \circ \eta_{2} \|_{C^{\lambda,r}({\mathcal F}_{0})} \\
\nonumber
&\leq C \left[ \| u_0 \|_{C^{\lambda+1,r}({\mathcal F}_{0})} + \| (\ell_{0},r_{0}) \|_{\R^{6}}\right]^{2}
\| \eta_{1} - \eta_{2} \|_{C^{0}([0,T];C^{\lambda+1,r}({\mathcal F}_{0}))} \\
\label{DLip}
& \ \ \ + \, C \left[ \| u_0 \|_{C^{\lambda+1,r}({\mathcal F}_{0})} + \| (\ell_{0},r_{0}) \|_{\R^{6}}\right]
\| u_{1} \circ \eta_{1} - u_{2} \circ \eta_{2} \|_{C^{0}([0,T];C^{\lambda+1,r}({\mathcal F}_{0}))}.
\end{align}
Using \eqref{K1}, we deduce that for $T$ of the form \eqref{LeTemps}, we have
\begin{equation} \label{TEL}
\| \div(\nabla \mu_{1}) \circ \eta_{1} - \div (\nabla \mu_{2}) \circ \eta_{2} \|_{C^{\lambda,r}({\mathcal F}_{0})} \leq C \| (\ell_{1},r_{1}) - (\ell_{2},r_{2}) \|_{C^{0}([0,T];\R^{6})}.
\end{equation}
The first term in \eqref{GrosseEstimee} can be also estimated by the right hand side of \eqref{TEL}, using Lemmas \ref{BBLemmaA.2} and \ref{BBLemma4} (it is simpler here since $\curl(\nabla \mu_{i})=0$). The third term in \eqref{GrosseEstimee} can also be estimated by the first term in the right hand side of \eqref{DLip}, using Lemma \ref{BBLemma4}.
The last term is estimated likewise, using \eqref{Estiu}, \eqref{Estiv} and \eqref{K2}. Hence we get that
\begin{equation*}
\| (\nabla \mu_{1}) \circ \eta_{1} - (\nabla \mu_{2}) \circ \eta_{2} \|_{C^{\lambda+1,r}({\mathcal F}_{0})} 
\leq C \| (\ell_{1},r_{1}) - (\ell_{2},r_{2}) \|_{C^{0}([0,T];\R^{6})}.
\end{equation*}
Also, it is again a consequence of Lemma \ref{BBLemma4} that
\begin{equation*}
\| (\nabla \Phi^{1}_{i}) \circ \eta_{1} - (\nabla \Phi^{2}_{i}) \circ \eta_{2} \|_{L^{\infty}(0,T;C^{\lambda+1,r}({\mathcal{F}_{0}})} \leq
C \left[ \| u_0 \|_{C^{\lambda+1,r}({\mathcal F}_{0})} + \| (\ell_{0},r_{0}) \|_{\R^{6}}\right]^{2}
\| \eta_{1} - \eta_{2} \|_{C^{0}([0,T];\R^{6})}.
\end{equation*}
This involves that the integrand in \eqref{NCEM2} is Lipschitz with respect to $(\ell,r)$: for instance, using that $\eta_{1}$ and $\eta_{2}$ are volume-preserving:
\begin{multline*}
\left|\, \int_{ \mathcal{F}_{1}(t)} \nabla \mu_{1} \cdot \nabla \Phi^{1}_i \, dx
- \int_{ \mathcal{F}_{2}(t)} \nabla \mu_{2} \cdot \nabla \Phi^{2}_i \, dx  \, \right| = \\
\left|\, \int_{ \mathcal{F}_{0}} (\nabla \mu_{1}) \circ \eta_{1} \cdot (\nabla \Phi^{1}_i) \circ \eta_{1} \, dx
- \int_{ \mathcal{F}_{0}} (\nabla \mu_{2}) \circ \eta_{2} \cdot (\nabla \Phi^{2}_i)\circ \eta_{2} \, dx  \, \right|,
\end{multline*}
and the claim follows. More precisely, we have
\begin{multline} 
\| (\tilde{\ell_{1}},\tilde{r}_{1}) - (\tilde{\ell_{2}},\tilde{r}_{2}) \|_{C^{0}([0,T])} \leq 
C T \Big\{
\left[ \| u_{0} \|_{C^{\lambda+1,r}({\mathcal F}_{0})} + \| (\ell_{0},r_{0}) \|_{\R^{6}} \right]^{2} \| \eta_{1} - \eta_{2} \|_{C^{0}([0,T];C^{\lambda+1,r}({\mathcal F}_{0}))} \\
\label{DiffLR}
+ \left[ \| u_0 \|_{C^{\lambda+1,r}({\mathcal F}_{0})} + \| (\ell_{0},r_{0}) \|_{\R^{6}}\right] 
\| (\ell_{1},r_{1}) - (\ell_{2},r_{2}) \|_{C^{0}([0,T];\R^{6})} \Big\}.
\end{multline}
%
Hence using Proposition \ref{DeuxImpositions}, we see that for some $T$ of the form \eqref{LeTemps} with a geometric constant $C_{*}$, the operator ${\mathcal A}$ is contractive. \par
\ \\
{\bf 4.} Hence, the operator ${\mathcal A}$ has a unique fixed point in ${\mathcal D}$, which proves the existence part of Theorem \ref{start3}. For what concerns uniqueness: if we are given a solution $(\ell,r,u)$ of the system, then is is easy to see that for $T$ sufficiently small, one has $(\ell,r) \in {\mathcal D}$, and the flow of $u$ belongs to ${\mathcal C}$. Then, because of the uniqueness in Proposition \ref{SolideImpose}, $(\ell,r)$ must be a fixed point of the operator ${\mathcal A}$, which proves that it must be equal to the one that we have constructed. \par
\ \\
{\bf 5.} It remains to prove that the velocity field in the solution $(\ell,r,u)$ that we constructed belongs to the space $C ( [0,T] ; C^{\lambda+1,r'} ({\mathcal F}(t)))$ in the sense of Remark \ref{RemEspTX}. Let $\rho>0$ such that
$\dist({\mathcal S}(t),\partial \Omega) \geq 3 \rho$ in $[0,T]$. Let
\begin{equation*}
{\mathcal G}_{\rho}:= \{ x \in \R^{3} \setminus {\mathcal S}_{0}, \ d(x,\partial {\mathcal S}_{0}) <\rho \}
\ \text{ and }
\ {\mathcal H}_{\rho}:= \{ x \in \overline{\Omega}, \ d(x,\partial \Omega) <\rho \}.
\end{equation*}
Let $\pi_{{\mathcal S}}$ (resp. $\pi_{\overline{\Omega}}$) be a continuous linear extension operator from functions defined in $\overline{{\mathcal G}_{\rho}}$ (resp. ${\mathcal H}_{\rho}$) to function defined in $\overline{{\mathcal G}_{\rho}} \cup {\mathcal S}_{0}$ (resp. ${\mathcal H}_{\rho} \cup (\R^{3} \setminus \Omega)$ and supported in some ball $\overline{B}(0,M)$), which sends $C^{\lambda+1,\alpha}(\overline{{\mathcal G}_{\rho}})$ to $C^{\lambda+1,\alpha}(\overline{{\mathcal G}_{\rho}} \cup {\mathcal S}_{0})$ (resp. $C^{\lambda+1,\alpha}(\overline{{\mathcal H}_{\rho}})$ to $C^{\lambda+1,\alpha}(\overline{{\mathcal H}_{\rho}} \cup (\R^{3} \setminus \Omega)$), for all $\alpha \in (0,1)$. (The construction of such an ``universal'' extension operator is classical, see \cite{Stein}.) 
For any $\tau \in D_{2\rho}$ (defined in \ref{DEpsilon}), define the extension operator $\pi_{{\mathcal S}}^{\tau}$:
\begin{equation*}
\pi_{{\mathcal S}}^{\tau}: C^{\lambda+1,\alpha}(\tau(\overline{{\mathcal G}_{\rho}})) \rightarrow C^{\lambda+1,\alpha}(\tau(\overline{{\mathcal G}_{\rho})} \cup \tau({\mathcal S}_{0})) \text{ by } \pi_{{\mathcal S}}^{\tau}:= \tau \circ \pi \circ \tau^{-1}.
\end{equation*}
Now we deduce the extension operator $\tilde{\pi}^{\tau}: C^{\lambda+1,\alpha}({\Omega} \setminus \tau({\mathcal S})) \to C^{\lambda+1,\alpha}(\R^{3})$ as follows. For $f \in C^{\lambda+1,\alpha}({\Omega} \setminus \tau({\mathcal S}))$, we let $\tilde{\pi}^{\tau}[f](x)$ equal to ${\pi}_{{\mathcal S}}^{\tau}[f](x)$ in $\tau({\mathcal S})$, to $f(x)$ in $\Omega \setminus \tau({\mathcal S})$ and to $\pi_{\overline{\Omega}}[f](x)$ in $\R^{3} \setminus {\Omega}$. Now we define
\begin{equation*}
\tilde{u}(t,\cdot) = \tilde{\pi}^{\tau(t)}[u(t,\cdot)] \ \text{ in }\  [0,T] \times \R^{3}.
\end{equation*}
Let us now check that the function $\tilde{u}$ is in $C([0,T];C^{\lambda+1,r'}(\R^{3}))$ for $r' \in (0,r)$. 
From the construction, we see that it suffices to prove that in $[0,T]$,
\begin{equation*}
\|u(t,x) - u(s,x) \|_{C^{\lambda+1,r'}({\mathcal H}_{\rho})}  + \|u(t,\tau(t)(x)) - u(s,\tau(s)(x)) \|_{C^{\lambda+1,r'}({\mathcal G}_{\rho})} 
\rightarrow 0 \text{ as } |t-s| \rightarrow 0.
\end{equation*}
Equivalently, it suffices that
\begin{multline*}
\|u(t,\eta(s,x)) - u(s,\eta(s,x)) \|_{C^{\lambda+1,r'}(\eta(s)^{-1}({\mathcal H}_{\rho}))} \\ + \|u(t,\tau(t)\circ \tau^{-1}(s)(\eta(s,x))) - u(s,\eta(s)(x)) \|_{C^{\lambda+1,r'}(\eta(s)^{-1}\circ \tau(s){\mathcal G}_{\rho})} 
\rightarrow 0.
\end{multline*}
But this follows from the facts that $u \in L^{\infty}(0,T;C^{\lambda+1,r}({\mathcal F}(t)))$, that $\eta \in C([0,T];C^{\lambda+1,r}({\mathcal F}_{0}))$, that $\tau \in C([0,T];SE(3))$ and that 
\begin{equation*}
\| u(t,\eta(t,x)) - u(s,\eta(s,x)) \|_{C^{\lambda+1,r}({\mathcal F}_{0})} \rightarrow 0 \text{ as } |t-s| \rightarrow 0.
\end{equation*}
Now,  that $\tilde{u}$ belongs to $C_{w}([0,T];C^{\lambda+1,r}(\R^{3}))$ is an automatic consequence of the fact that it belongs to $C([0,T];C^{\lambda+1,r'}(\R^{3}))$ and to $L^{\infty}(0,T;C^{\lambda+1,r}(\R^{3}))$. 
Finally, using the equations we infer that  $(x_{B},  r )  \in  C^2 ( (-T,T) ) \times  C^1 ( (-T,T) )$, $\partial_t u  \in C_w ( (-T,T) ; C^{\lambda ,r}  (\mathcal{F} (t) ) )$ and $\partial_t  u \in C ( (-T,T) ; C^{\lambda ,r'}  (\mathcal{F} (t) )   )$, for $ r'  \in (0,r)$. 

 \par
\subsubsection{Proof of Proposition \ref{Prostart3}}
Consider $(\ell^{1},r^{1},u^{1})$, $(\ell^{2},r^{2},u^{2})$, $\eta_{1}$ and $\eta_{2}$ as in the statement. Introduce $u^{m}$ as the solution given by Proposition \ref{SolideImpose}, with the solid movement given by $(\ell_{1},r_{1})$ and where the initial condition $u^{m}_{0}$ is given by \eqref{OmegaAssocU} associated to $(\ell^{1}_{0},r_{0}^{1},\omega^{2}_{0}, \oint_{\Gamma_{1}} {u}^{2}_{0} .d \tau, \dots, \oint_{\Gamma_{g}} {u}^{2}_{0} .d \tau)$. Call $\eta_{m}$ the corresponding fluid flow. \par
Consider the operator ${\mathcal A}_{2}$ (resp. ${\mathcal A}_{1}$) associated to the initial datum $(\ell^{2}_{0},r^{2}_{0},u_{0}^{2})$ (resp. $(\ell^{1}_{0},r^{1}_{0},u_{0}^{1})$). Since ${\mathcal A}_{2}$ is contractive and has $(\ell_{2},r_{2})$ as its fixed point, we have
\begin{equation*}
\| (\ell^{1},r^{1}) - (\ell^{2},r^{2}) \|_{C^{0}([0,T];\R^{6})} \lesssim \| (\ell^{1},r^{1}) - {\mathcal A}_{2}(\ell^{1},r^{1}) \|_{C^{0}([0,T];\R^{6})} .
\end{equation*}
(For instance, we use Remark \ref{UneConstanteKEnplus} with $C$ large enough, depending on $R$, so that both $(\ell^{1},r^{1})$ and $(\ell^{2},r^{2})$ belong to ${\mathcal D}$.) \par
Note that when computing ${\mathcal A}_{2}(\ell_{1},r_{1})$ by the formulas \eqref{NCSylvester}-\eqref{NCEM2}, the fluid domain is exactly ${\mathcal F}_{1}(t)$. Consequently when computing \eqref{NCEM2} corresponding to ${\mathcal A}_{2}(\ell_{1},r_{1})$ and comparing with \eqref{NCEM2} corresponding to ${\mathcal A}_{1}(\ell_{1},r_{1})=(\ell_{1},r_{1})$, the only differences concern the term $\nabla \mu$ and the initial data $(\ell_{0},r_{0})$. 
Hence proceeding as for \eqref{DiffLR}, one deduces that
\begin{multline} \label{32et2}
\| (\ell^{1},r^{1}) - (\ell^{2},r^{2}) \|_{C^{0}([0,T];\R^{6})}  \leq C \Big( \| (\ell_{0}^{1},r_{0}^{1}) - (\ell_{0}^{2},r_{0}^{2}) \|_{\R^{6}}
+ \| \eta_{m} - \eta_{1}\|_{L^{\infty}([0,T];C^{\lambda+1,r}({\mathcal F}_{0}))} \\ 
+ \| u_{m}(t,\eta_{m}(t,x)) - u_{1}(t,\eta_{1}(t,x)) \|_{L^{\infty}([0,T];C^{\lambda+1,r}({\mathcal F}_{0}))} \Big).
\end{multline}%
From Proposition \ref{DeuxImpositions} and Remark \ref{AAssocA}, we deduce 
\begin{multline}  \label{1et32}
\| \eta_{m} - \eta_{2}\|_{L^{\infty}([0,T];C^{\lambda+1,r}({\mathcal F}_{0}))} + 
T \| u_{m}(t,\eta_{m}(t,x)) - u_{2}(t,\eta_{2}(t,x)) \|_{L^{\infty}([0,T];C^{\lambda+1,r}({\mathcal F}_{0}))} \\ \leq KT \| (\ell_{1},r_{1}) - (\ell_{2},r_{2}) \|_{C^{0}([0,T];\R^{6})}.
\end{multline}
Since ${\mathcal T}^{\ell^{1},r^{1}}$ (associated to initial data $(\omega^{1}_{0}, \oint_{\Gamma_{1}} {u}^{1}_{0} .d \tau, \dots, \oint_{\Gamma_{g}} {u}^{1}_{0} .d \tau)$) is contractive, and using \eqref{Upareta}, we see that for some $C>0$,
\begin{multline} \label{AutreEstC}
\| \eta_{m} - \eta_{1}\|_{L^{\infty}([0,T];C^{\lambda+1,r}({\mathcal F}_{0}))} + 
 \| u_{m}(t,\eta_{m}(t,x)) - u_{1}(t,\eta_{1}(t,x)) \|_{L^{\infty}([0,T];C^{\lambda+1,r}({\mathcal F}_{0}))} \\
\leq C \| {\mathcal T}^{\ell^{1},r^{1}}(\eta_{m}) - \eta_{m} \|_{L^{\infty}([0,T];C^{\lambda+1,r}({\mathcal F}_{0}))}.
\end{multline}
We proceed as for \eqref{GEDC} (it is, in fact, simpler here). Calling ${\mathcal U}_{m}$ the function ${\mathcal U}$ constructed when computing ${\mathcal T}^{\ell^{1},r^{1}}(\eta_{m})$, we see that at each $t$:
\begin{eqnarray*}
\nonumber
\| {\mathcal U}_{m} - u_{m}\circ \eta_{m} \|_{C^{\lambda+1,r}({\mathcal F}_{0})}
\nonumber
&\leq &
C 
\Big(
\| \curl( {\mathcal U}_{m}\circ \eta_{m}^{-1} ) - \curl ( u_{m}) \|_{C^{\lambda,r}({\mathcal F}(t))} \\
\nonumber
&\ &\ \ 
+ \ \| \div({\mathcal U}_{m}\circ \eta_{m}^{-1}) - \div (u_{m}) \|_{C^{\lambda,r}({\mathcal F}(t))} \\ 
\nonumber
&\ &\ \ 
+\ \sum_{i=1}^{g} \left| \oint_{\eta_{2}(\Gamma_{i})} ({\mathcal U}_{m} \circ \eta_{m}^{-1} - u_{m}) . d \tau  \right|  \\
&\ &\ \ 
+\ \| ({\mathcal U}_{m} \circ \eta_{m}^{-1}).n - u_{m}.n \|_{C^{\lambda+1,r}( \partial {\mathcal F}(t))} 
\Big) \\
&\leq & C \Big( \| \omega_{0}^{1} -\omega_{0}^{2} \|_{C^{\lambda,r}({\mathcal F}_{0})} + \sum_{i=1}^{g} \left|\oint_{\eta(\Gamma_{i})} {u}^{1}_{0} .d \tau - \oint_{\Gamma_{i}} u^{2}_{0} .d \tau \right| \Big).
\end{eqnarray*}
Recalling that ${\mathcal T}^{\ell^{1},r^{1}}(\eta_{m}(t,\cdot)) - \eta_{m}(t,\cdot) = \int_{0}^{t} {\mathcal U}_{m} - u_{m}\circ \eta_{m}$, with  \eqref{32et2}, \eqref{1et32} and \eqref{AutreEstC}, we deduce the claim.
%
%
%
%
%
%


\end{document}